\documentclass[11pt]{article}
\usepackage{smile}

\usepackage{fullpage}
\usepackage{lscape}
\usepackage{bigints}
\usepackage{framed}
\usepackage{mdframed}
\usepackage{enumerate}
\usepackage[inline]{enumitem}
\usepackage[T1]{fontenc}
\usepackage{moresize}
\usepackage{bm}
\usepackage{bbm}
\usepackage{dsfont}
\usepackage{amsmath}
\usepackage{amssymb}
\usepackage{amsthm}
\usepackage{amsfonts}
\usepackage{stmaryrd}
\usepackage{array}
\usepackage{mathrsfs}
\usepackage{mathtools} 
\usepackage{extarrows}
\usepackage{stackrel}
\usepackage{relsize,exscale}
\usepackage{scalerel}
\usepackage[nodisplayskipstretch]{setspace}
\usepackage{color}
\usepackage[usenames,dvipsnames]{xcolor}
\usepackage{cancel}
\usepackage{soul}
\usepackage{undertilde}
\usepackage{xfrac}
\usepackage{siunitx}
\usepackage{graphicx}
\usepackage{float}
\usepackage{rotating}
\usepackage{subcaption}
\usepackage{overpic}
\usepackage[all]{xy}
\usepackage{tikz}
\usetikzlibrary{arrows,matrix,positioning,calc,automata,patterns}
\usepackage{booktabs}
\usepackage{dcolumn}
\usepackage{multirow}
\usepackage{diagbox}
\usepackage{tabularx}
\usepackage{verbatim}
\usepackage{listings}
\usepackage[ruled,vlined]{algorithm2e}
\usepackage{fancyvrb}
\usepackage{hyperref}
\usepackage[round]{natbib}
\usepackage{sectsty}

\hypersetup{
    bookmarks=true,         
    unicode=false,          
    pdftoolbar=true,        
    pdfmenubar=true,        
    pdffitwindow=false,     
    pdfstartview={FitH},    
    pdftitle={My title},    
    pdfauthor={Author},     
    pdfsubject={Subject},   
    pdfcreator={Creator},   
    pdfproducer={Producer}, 
    pdfkeywords={key1, key2}, 
    pdfnewwindow=true,      
    colorlinks=true,        
    linkcolor=blue,         
    citecolor=blue,         
    filecolor=blue,         
    urlcolor=cyan           
}

\usepackage{stackengine}
\stackMath
\newcommand\tenq[2][1]{%
\def\useanchorwidth{T}%
\ifnum#1>1%
\stackunder[0pt]{\tenq[\numexpr#1-1\relax]{#2}}{\!\scriptscriptstyle\thicksim}%
\else%
\stackunder[1pt]{#2}{\!\scriptstyle\thicksim}%
\fi%
}

\makeatletter
\DeclareRobustCommand\widecheck[1]{{\mathpalette\@widecheck{#1}}}
\def\@widecheck#1#2{%
    \setbox\z@\hbox{\m@th$#1#2$}%
    \setbox\tw@\hbox{\m@th$#1%
       \widehat{%
          \vrule\@width\z@\@height\ht\z@
          \vrule\@height\z@\@width\wd\z@}$}%
    \dp\tw@-\ht\z@
    \@tempdima\ht\z@ \advance\@tempdima2\ht\tw@ \divide\@tempdima\thr@@
    \setbox\tw@\hbox{%
       \raise\@tempdima\hbox{\scalebox{1}[-1]{\lower\@tempdima\box
\tw@}}}%
    {\ooalign{\box\tw@ \cr \box\z@}}}
\makeatother

\def\given{\,|\,}

\def\Biggiven{\,\Big{|}\,}
\def\tr{\mathop{\text{tr}}\kern.2ex}

\def\P{{\mathrm P}}

\def\E{{\mathrm E}}

\def\d{{\mathrm d}}
\def\cI{{\mathcal I}}
\def\supp{\operatorname{supp}}

\newcommand{\zahl}[1]{\llbracket #1\rrbracket}
\newcommand\yestag{\addtocounter{equation}{1}\tag{\theequation}}
\newcolumntype{L}[1]{>{\raggedright\let\newline\\\arraybackslash\hspace{0pt}}m{#1}}
\newcolumntype{C}[1]{>{  \centering\let\newline\\\arraybackslash\hspace{0pt}}m{#1}}
\newcolumntype{R}[1]{>{ \raggedleft\let\newline\\\arraybackslash\hspace{0pt}}m{#1}}
\newcolumntype{d}[1]{D{.}{.}{#1}}
\newcolumntype{H}{>{\setbox0=\hbox\bgroup}c<{\egroup}@{}}
\newcolumntype{Z}{>{\setbox0=\hbox\bgroup}c<{\egroup}@{\hspace*{-\tabcolsep}}}
\newcolumntype{b}{X}
\newcolumntype{s}{>{\hsize=.5\hsize}X}

\numberwithin{equation}{section}

\newtheorem{theorem}{Theorem}[section]
\newtheorem{lemma}{Lemma}[section]

\newtheorem{assumption}{Assumption}[section]

\newtheorem{innercustomgeneric}{\customgenericname}
\providecommand{\customgenericname}{}
\newcommand{\newcustomtheorem}[2]{%
  \newenvironment{#1}[1]
  {%
   \renewcommand\customgenericname{#2}%
   \renewcommand\theinnercustomgeneric{##1}%
   \innercustomgeneric
  }
  {\endinnercustomgeneric}
}
\newcustomtheorem{customdefinition}{Definition}
\newcustomtheorem{customdefinitions}{Definitions}
\newcustomtheorem{customtheorem}{Theorem}
\newcustomtheorem{customassumption}{Assumption}
\newcustomtheorem{customlemma}{Lemma}
\newcustomtheorem{customexample}{Example}
\theoremstyle{definition}

\newtheorem{remark}{Remark}[section]

\usepackage{enumitem}
\makeatletter
\newcommand{\mylabel}[2]{#2\def\@currentlabel{#2}\label{#1}}
\makeatother

\graphicspath{{./fig3/}}

\allowdisplaybreaks

\begin{document}

\setlength{\abovedisplayskip}{5pt}
\setlength{\belowdisplayskip}{5pt}
\setlength{\abovedisplayshortskip}{5pt}
\setlength{\belowdisplayshortskip}{5pt}
\hypersetup{colorlinks,breaklinks,urlcolor=blue,linkcolor=blue}

\title{\LARGE On the adaptation of causal forests to manifold data}

\author{Yiyi Huo\thanks{Department of Biostatistics, University of Washington, Seattle, WA 98195, USA; e-mail: {\tt yiyih@uw.edu}}, ~~~Yingying Fan\thanks{Data Sciences and Operations Department, Marshall School of Business, University of Southern California, Los Angeles, CA 90089, USA; e-mail: {\tt fanyingy@usc.edu}. Fan's research was partially supported by NSF Grant DMS--2310981},~~~and~
Fang Han\thanks{Department of Statistics, University of Washington, Seattle, WA 98195, USA; e-mail: {\tt fanghan@uw.edu}. Han's research was partially supported by NSF Grants SES-2019363 and DMS-2210019.}
}

\date{\today}

\maketitle

\vspace{-1em}

\begin{abstract}
Researchers often hold the belief that random forests are ``the cure to the world's ills'' \citep{bickel2010leo}. But how exactly do they achieve this? 
Focused on the recently introduced causal forests \citep{athey2016recursive,wager2018estimation}, this manuscript aims to contribute to an ongoing research trend towards answering this question, proving that causal forests can adapt to the unknown covariate manifold structure. In particular, our analysis shows, for the first time, that a causal forest estimator can achieve the optimal rate of convergence for estimating the conditional average treatment effect, with the covariate dimension automatically replaced by the manifold dimension. These findings align with analogous observations in the realm of deep learning and resonate with the insights presented in Peter Bickel's 2004 Rietz lecture. 
\end{abstract}


\maketitle

\section{Introduction}\label{sec:intro}

Suppose that we observe a binary treatment data. It contains $n$ observations of $(X, D, Y(D))$ from an independent and identically distributed sample $\{(X_i, D_i,Y_i(0), Y_i(1))\}_{i=1}^n$ of $\{X, D, Y(0), Y(1)\}$, with 
$$
X_i\in\mathcal{X}\subset\mathbb{R}^d, ~~D_i\in\{0,1\},~~ {\rm and}~~ (Y_i(0), Y_i(1))\in\mathbb{R}^2 
$$
as the potential outcomes \citep{neyman1923applications,rubin1974estimating}. The goal of interest is to, based on such a data, infer the conditional average treatment effect (CATE, \cite{athey2016recursive,athey2019machine}):
\begin{align}\label{eq:cate}
\tau(x) := \E\Big[Y_i(1)-Y_i(0) \Biggiven X=x\Big],~~~\text{ for some/all }x\in\cX.
\end{align}

Causal inference towards estimating $\tau(x)$ often invokes some nonparametric/semiparametric regression procedures \citep{van2006targeted,wager2018estimation,kunzel2019metalearners,kennedy2020towards,knaus2021machine,nie2021quasi}. To justify these methods, besides the commonly imposed causality assumptions of unconfoundedness and overlap \citep{imbens2015causal}, it is often to assume that: 
\begin{itemize}
\item[\bf (D)] {\bf distribution assumption}: the Lebesgue density and conditional density of $X$ given $D$ are (locally) ``sufficiently regular'';
\item[\bf (R)] {\bf regression assumption}: the regression functions $\mu_\omega(x):=\E[Y~|~X=x,D=\omega]$, for $\omega=0$ and $1$, are (locally) ``sufficiently smooth''.   
\end{itemize}

This paper aims to challenge the distribution assumption {\bf (D)} and proposes to replace it with the following covariate manifold assumption:
\begin{itemize}
\item[\bf (M)] {\bf manifold assumption}: the distribution of $X$ lives on a low-dimensional manifold and its Hausdorff (conditional) density is (locally) ``sufficiently regular''. 
\end{itemize}
There are three main motivations to consider the assumption {\bf (M)} in contrast to {\bf (D)}. Foremost and of paramount importance, in practical terms, instances of data exhibiting intrinsic low-dimensional structures seem to be ubiquitous. If we embrace this perspective, it becomes judicious to postulate a manifold assumption when scrutinizing statistical methodologies related to such data.

Second, from a methodological point of view, although the mathematical statistics literature has long been aware of the phenomenon that ``local methods'' can adapt to the covariate manifold structure \citep{bickel2007local,kpotufe2010curse},
the causal inference literature—except for a notable work of \cite{khosravi2019non}—is largely silent about such phenomena. Discussions that address the mitigation of the curse of dimensionality through local, in contrast to global, causal methods could prove beneficial in this regard.

Lastly, from a theoretical standpoint, since Peter Bickel's seminal 2004 Rietz lecture, there has been a surge of theoretical interest in understanding adaption to manifold structures. Noteworthy contributions in this realm include  \cite{bickel2007local} and \cite{cheng2013local} on local linear regression, \cite{kpotufe2011k} on nearest neighbor regression, \cite{kpotufe2013adaptivity} on kernel regression, \cite{yang2015minimax} on Gaussian process regression, \cite{kpotufe2012tree} on tree methods, \cite{liao2021multiscale} on multiscale regression, and \cite{schmidt2019deep}, \cite{chen2022nonparametric}, and \cite{jiao2023deep} on deep learning methods. 

Interestingly, amidst these advancements, discussions regarding the adaptation of random forests to manifold structures—albeit alluded to in various contexts (cf. \cite{biau2016random} and references therein)—still remain notably absent. Our work thus also complements the studies of \cite{biau2008consistency}, \cite{biau2012analysis}, \cite{louppe2013understanding}, \cite{scornet2015consistency}, \cite{zhu2015reinforcement}, \cite{oprescu2019orthogonal}, \cite{syrgkanis2020estimation}, \cite{mourtada2020minimax}, \cite{klusowski2021sharp},  \cite{cattaneo2022pointwise}, \cite{chi2022asymptotic}, among many other, which are focused on adaptation of random forests methods to different types of regression models.

Specifically, we provide a statistical analysis of the causal forest method of \cite{lin2022regression}—a slight modification to the causal forests introduced in \cite{athey2016recursive} and \cite{wager2018estimation}—based on the assumptions {\bf (R)} and {\bf (M)}. We explore pointwise rates of convergence and central limit theorems for the causal forest estimator $\hat\tau(x)$ of $\tau(x)$, showing that
\begin{itemize}
\item[(1)] the pointwise mean squared error (MSE) could achieve the minimax risk $n^{-2/(m+2)}$, with $m$ representing the manifold dimension in contrast to the ambient space dimension;
\item[(2)] a rescaled version of $\hat\tau(x)$ could admit a central limit theorem towards inferring $\tau(x)$, with the asymptotic variance consistently estimable;
\item[(3)] all the theoretical outcomes are established within a broader framework that encompasses a variety of nonparametric methods beyond the scope of the causal forest approach.
\end{itemize}

The focus of this paper is on the theoretical side—about the CATE estimation based on causal forests under a manifold setting—rather than developing new algorithms expressly crafted for manifold data. Our intent is to expand the already extensively explored theory in mathematical statistics to encompass random forests-based methods, and to connect them to causal inference. Through a particular version of the CATE estimator (i.e. the causal forest), our analysis suggests how precise a theoretical formulation can be made and  used for justifying other local methods. 

\vspace{0.3cm}
\noindent{\bf Paper organization.} The rest of this manuscript is organized as follows. Section \ref{sec:setup} introduces the setup of the problem and the studied causal forest estimator. Section \ref{sec:theory} presents the main theory of this manuscript, highlighting the adaption of causal forests to manifold structure based on the criteria of consistency, rate of convergence, and central limit theorem. Section \ref{sec:framework} lifts everything in Section \ref{sec:theory} to a general framework, covering potentially many other local methods. All proofs are relegated to the appendix.

\section{Setup}\label{sec:setup}

We consider the setting where there are $n$ studied units, indexed by $i\in[n]:=\{1,\ldots,n\}$. Each unit receives a binary treatment, indicated by $D_i\in\{0,1\}$, with $D_i=1$ or $D_i=0$ demonstrating that the $i$-th unit is in the treatment or control group, respectively. Following the Neyman-Rubin causal model, each unit then has a pair of potential outcomes, $(Y_i(0), Y_i(1))$. The realized outcome for the $i$-th unit is denoted by
$$
    Y_i = Y_i(D_i)= \begin{cases}
        Y_i(0), & \mbox{ if } D_i=0,\\
        Y_i(1), & \mbox{ if } D_i=1.
    \end{cases}
$$
Let $X_i\in\cX\subset\mathbb{R}^d$ represent a vector of pretreatment variables, not affected by the treatment. The observed data then is $\{(X_i,D_i,Y_i)\}_{i=1}^n$, which we believe to be exchangeable and, especially, contain no interference. 

In order to estimate the CATE $\tau(x)$ in \eqref{eq:cate}, we are interested in a causal forest estimator first proposed in \citet[Section 5]{lin2022regression}. 
To begin with, consider a tree estimator as \cite{athey2016recursive}. Let two {\it generic} trees, denoted by $T^1$ and $T^0$, be constructed for the treated and control groups, $\{(X_i,Y_i)\}_{i=1,D_i=1}^n$ and $\{(X_i,Y_i)\}_{i=1,D_i=0}^n$, respectively. Let $L^1$ and $L^0$ represent the sets of leaves generated by the two trees; they each represent a partition of the covariate space $\cX$. 
 The corresponding tree regression estimator of the {\it unobserved} potential outcome at the unit-$i$ is then defined as follows: it is the sample average of outcomes $Y_j(D_j)$ of individuals whose covariates $X_j$ are in the same leaf as $X_i$ but receiving the opposite treatment (i.e., $D_j=1-D_i$).

Next, we grow a forest from such trees through {\it subsampling} \citep{breiman2001random}. Let's first fix two parameters. Set $s$ to represent the subsample size and set $B$ to represent the number of replications. In each round of subsampling, indexed by $b\in[B]$, let $\cI_b^\omega$ ($\omega=0,1$) stand for a size-$s$ set being randomly sampled without replacement from $\{i:D_i=\omega\}$. For each round $b\in[B]$, one could then grow two causal trees, $T_b^1$ and $T_b^0$, from $\{(X_i,Y_i)\}_{i\in\cI_b^1}$ and $\{(X_i,Y_i)\}_{i\in\cI_b^0}$, respectively, leading to $2B$ many trees. 

The {\it forest} then aggregates all $2B$ trees by averaging. For each $i\in[n]$, the final forest estimator of $(Y_i(0),Y_i(1))$, combined with a regression adjustment step due to \cite{rubin1973use} and \cite{abadie2011bias}, is defined to be
\begin{align*}
    \hat{Y}_i^{\rm RF}(0) := \begin{cases}
        Y_i, & \mbox{ if } D_i=0,\\
       \displaystyle \frac1B \sum_{b=1}^B \Big[\Big(\Big\lvert \Big\{j \in \cI^0_b:X_j \in L^0_b(X_i)\Big\} \Big\rvert\Big)^{-1}\!\!\!\!\!\sum_{j \in \cI^0_b:X_j \in L^0_b(X_i)} \!\!\!(Y_j+\hat\mu_0(X_i)-\hat\mu_0(X_j)\Big], & \mbox{ if } D_i=1,    
    \end{cases}
\end{align*}
and
\begin{align*}
    \hat{Y}_i^{\rm RF}(1) := \begin{cases}
   \displaystyle     \frac1B \sum_{b=1}^B \Big[\Big(\Big\lvert \Big\{j \in \cI^1_b:X_j \in L^1_b(X_i)\Big\} \Big\rvert\Big)^{-1}\!\!\!\!\!\sum_{j \in \cI^1_b:X_j \in L^1_b(X_i)} \!\!\!(Y_j+\hat\mu_1(X_i)-\hat\mu_1(X_j)\Big], & \mbox{ if } D_i=0,\\
        Y_i, & \mbox{ if } D_i=1,
    \end{cases}
\end{align*}
where for any set $A$, $|A|$ represents its cardinality; for any $\omega\in\{0,1\}$, $b\in[B]$, and $x\in\cX$, $L_b^\omega(x)$ stands for the leaf of $T_b^\omega$ that contains $x$; and $\hat\mu_\omega(\cdot)$ is a {\it regression adjustment function} allowed to be set 0 in estimating the CATE.

Lastly, the estimator for $\tau(x)$ is articulated as a ``kernel smoothed average'' of $(\hat Y_i^{\rm RF}(1)-\hat Y^{\rm RF}_i(0))$ values corresponding to covariates proximate to $x$. Delving into the specifics, let $K_h(\cdot):\mathbb{R}^d\to\mathbb{R}$ represent a {\it multivariate kernel function} incorporating a user-defined tuning parameter $h\in \mathbb R^{> 0}$. The ultimate causal forest estimator is
\begin{align}\label{eq:causalf}
\hat\tau_{\rm RF}(x):=\frac{\sum_{i=1}^nK_h(X_i-x)(\hat Y_i^{\rm RF}(1)-\hat Y_i^{\rm RF}(0))}{\sum_{i=1}^nK_h(X_i-x)}.
\end{align}

\begin{remark}
A notable distinction between the investigated estimator in \eqref{eq:causalf} and the original causal tree and forest estimators introduced in \cite{athey2016recursive} and \cite{wager2018estimation} lies in the construction process. The former entails the creation of two distinct trees, in line with the methodology proposed in \cite{lin2022regression}. In contrast, the latter approach involves the development of a single tree that mixes both treatment and control data. While we conjecture that this divergence may not be pivotal, the separation approach significantly streamlines subsequent analysis and facilitates intricate computations to be developed in Section \ref{sec:theory}. Consequently, we adhere to this separation methodology throughout this manuscript, while acknowledging that, in practical applications, adopting a unified tree-building strategy as presented in \cite{athey2016recursive} and \cite{wager2018estimation} could potentially suffice.
\end{remark}

\begin{remark}
The estimator in \eqref{eq:causalf} further diverges from the causal forest estimator introduced in \cite{wager2018estimation} due to the incorporation of kernel smoothing. In its initial form, aside from the distinction of constructing trees using combined data versus two distinct groups, the causal forest employs a straightforward difference, $\hat Y_i^{\rm RF}(1)-\hat Y_i^{\rm RF}(0)$, to approximate $\tau(X_i)$. Our approach, on the other hand, is underpinned by \cite{lin2022regression}, which posits that \eqref{eq:causalf} can also yield a natural estimator for the average treatment effect (ATE). As a theoretical exposition, we refrain from engaging in a direct comparison between these two variants, instead concentrating our efforts on the adaptation of such methodologies to manifold structures.
\end{remark}

\begin{remark}
In their seminal paper, \cite{athey2019generalized} introduced the concept of ``local-centering'', a notion advocating the utilization of residuals from outcome and treatment variables—rather than original values—as data inputs when implementing a causal forest algorithm. This approach aligns with the ``residual-on-residual'' strategy \citep{robinson1988root} and the paradigm of double machine learning \citep{chernozhukov2018double}, while also exhibiting a fundamental connection to the concept of ``double robustness'' \citep{robins1994estimation, bang2005doubly}. In a series of recent studies, \cite{lin2021estimation} and \cite{lin2022regression} demonstrated that, in the context of ATE estimation, the advantage of double robustness can also be harnessed by incorporating a regression adjustment approach devised in \cite{rubin1973use} and \cite{abadie2011bias}, with $\hat\mu_w(\cdot)$ designed to approximate the conditional expectation $\mu_w(x):=\E[Y|X=x,D=w]$. In Section \ref{sec:framework}, we will delve deeply into an enhanced iteration of \eqref{eq:causalf}, which discusses this supplementary regression adjustment step in more details.
\end{remark}

\section{Theory}\label{sec:theory}

To present the theory, we first introduce the assumptions that regulate the data generating process and permit causal identification of $\tau(x)$ for {\it any} $x\in\cX$. 

\begin{assumption}\label{asp:causal} \phantomsection 
    \begin{enumerate}[itemsep=-.5ex,label=(\roman*)]
        \item\label{asp:causal-1} $[(X_i,D_i,Y_i)]_{i=1}^n$ are independent and identically distributed (i.i.d.) following the joint distribution of $(X,D,Y(D))$.
          
      \item\label{asp:causal-2} For almost all $x \in \cX$, $D$ is independent of $(Y(0),Y(1))$ conditional on $X=x$, and there exists a fixed constant $\eta > 0$ such that $\eta < \P(D=1 \given X=x) < 1-\eta$.
      
      \item\label{asp:causal-3} Letting $\mu_\omega(x):=\E(Y(\omega)|X=x, D=\omega)$ and $U_\omega(x):=Y(\omega)-\mu_\omega(x)$, it is assumed that $U_\omega(\cdot)$ is uniformly bounded and $\mu_\omega(\cdot)$ is bounded continuous for all $x \in \cX$ and $\omega \in \{0,1\}$. 
    \end{enumerate}
\end{assumption}


\begin{remark}
Assumption \ref{asp:causal}\ref{asp:causal-1} is the standard i.i.d.-ness condition commonly adopted in observational studies. Assumption \ref{asp:causal}\ref{asp:causal-2} is the unconfoundedness and overlap conditions for identifying the causal effect. Assumption \ref{asp:causal}\ref{asp:causal-3} regulates the regression function and its residual, constituting Assumption {\bf (R)} presented in Section \ref{sec:intro}.
\end{remark} 

The next set of assumptions regulates the distribution of $X$, constituting Assumption {\bf (M)} presented in Section \ref{sec:intro}.

\begin{assumption}  \phantomsection \label{asp:manifold} 
    \begin{enumerate}[itemsep=-.5ex,label=(\roman*)]
      \item\label{asp:manifold-dim} It is assumed that $X \in \cM$, where $\cM$ is an $m$-dimensional $C^{\infty}$ manifold in $\mathbb{R}^d$ with the manifold dimension $m \leq d$.

      \item\label{asp:manifold-law}The marginal law of $X$, denoted by $\zeta$, is assumed to be absolutely continuous with respect to the restricted m-dimensional Hausdorff measure in $\mathbb{R}^d$ on $\cM$ \citep[Definition 2.1]{gariepy2015measure}, denoted by $\mathcal{H}^m$. 

      \item\label{asp:manifold-compact}The support of $\zeta$, denoted by $\supp(\zeta)$, is assumed to be compact.
      
      \item \label{asp:manifold-f} Write $f:=\d \zeta / \d \cH^m$ to be the Radon-Nikodym derivative of $\zeta$ with respect to $\cH^m$.
      We assume that $f(\cdot)$ is bounded and bounded away from zero for any $x\in\supp(\zeta)$.

        
        \item\label{asp:conditional} Let the law of $X|D=\omega$ for $\omega\in\{0,1\}$ be denoted by $\zeta_\omega$, and the corresponding restriction to a set $U$ be denoted by $\zeta_{\omega, U}$. For any point $x \in \supp(\zeta)$ and any chart $(U, \psi)$ such that $U=U_x$ is a coordinate neighborhood of $x$ and $\psi: U \rightarrow$ $V \subset \mathbb{R}^m$, let $g_{\omega,x}=g_{\omega,U_x}:=\mathrm{d}\left(\psi_* \zeta_{\omega,U}\right) / \mathrm{d} \lambda$ be the Radon-Nikodym derivative of the restricted pushforward measure $\psi_* \zeta_{\omega,U}$ with respect to the Lebesgue measure $\lambda$. It is assumed that $g_{\omega,x}(\cdot)$ is locally Lipschitz.
    \end{enumerate}
    \end{assumption}
    
    \begin{remark}
    Assumption \ref{asp:manifold}\ref{asp:manifold-dim} regulates the random variable $X$ to be resided on an $m$-dimensional smooth manifold. Assumptions \ref{asp:manifold}\ref{asp:manifold-law} and \ref{asp:manifold}\ref{asp:manifold-compact} require the Hausdorff density to exist and being compactly supported. Assumptions  \ref{asp:manifold}\ref{asp:manifold-f} and \ref{asp:manifold}\ref{asp:conditional} further require the density to be ``sufficiently regular'', namely, bounded and bounded away from 0 and (locally) Lipschitz smooth. 
    \end{remark}
    
    \begin{remark}
Two additional remarks are worth noting. First, when $\cM$ is set to be the $d$-dimensional space without any more structure, Assumption \ref{asp:manifold} simplifies to the commonly employed density assumptions as seen in, e.g., \citet[Theorems 3.1 and 4.1]{wager2018estimation}. Second, although in Assumption \ref{asp:manifold} and at various other places ahead we formulate ``global'' assumptions that need to apply to {\it any} $x\in\cM$, it is straightforward to adapt these assumptions to a {\it local} context by considering those within a small neighborhood of a particular $x$. This is by confining our investigation to a localized subset of the manifold.
    \end{remark}

The next set of assumptions regulates the kernel smoothing function used in \eqref{eq:causalf}.

\begin{assumption} \phantomsection \label{asp:kernel}
    \begin{enumerate}[itemsep=-.5ex,label=(\roman*)]  
    \item \label{asp:kernel1} We define $K_h(x)=h^{-d/2}K(\|h^{-1/2}x\|)$, with $\|\cdot\|$ representing the Euclidean norm and $K(\cdot):\mathbb{R}^{\geq 0}\to \mathbb{R}^{\geq 0}$ assumed to be of a bounded support, $\supp(K)$, i.e, the diameter of $\supp(K)$, denoted by $d_K$, is finite.   
    
    \item \label{asp:kernel3} We assume $K$ is Lipchitz over $\supp(K)$.
    
        \item \label{asp:kernel2} We assume $h \to 0$ and $n h^{m/2} \to \infty$.
    \end{enumerate}
\end{assumption}

\begin{remark}
Assumptions \ref{asp:kernel}\ref{asp:kernel1} and \ref{asp:kernel}\ref{asp:kernel3} are added for simplifying the subsequent theoretical analysis. They are satisfied by using, e.g., the box kernel or a truncated Gaussian kernel. Lastly, Assumption \ref{asp:kernel}\ref{asp:kernel2} is easily satisfied by choosing not too small an $h$.
\end{remark}

In the end, we regulate the causal forest algorithm introduced in Section \ref{sec:setup}. In the following, for any real sequences $\{a_n\}, \{b_n\}$, we adopt the notation ``$a_n\lesssim b_n$'' to mean $|a_n|\leq C|b_n|$ for all sufficiently large $n$. We write $a_n\asymp b_n$ if both $a_n\lesssim b_n$ and $b_n\lesssim a_n$ hold.

\begin{assumption} \phantomsection \label{asp:rf1}
We assume that 
\begin{enumerate}[itemsep=-.5ex,label=(\roman*)]
\item\label{asp:rf1-1} $s=s_n \rightarrow \infty$ and $n/B=O(1)$;
\item\label{asp:rf1-2} there exists a constant $ \epsilon \in(0, 1/(m+2))$ such that for any tree $T$ considered in Section \ref{sec:setup} and any leaf $L_t$ in $T$,    
$$
\Big\lVert{\rm diam}(L_t\cap\cM)\Big\rVert_\infty \lesssim h^{1/2+\epsilon},~~{\rm uniformly},
$$
where ${\rm diam}(\cdot)$ outputs the diameter of the input set and $\|\cdot\|_{\infty}$ stands for the $L^{\infty}$ norm.
\end{enumerate}
\end{assumption}

\begin{assumption} \phantomsection \label{asp:rf-honest}
\begin{enumerate}[itemsep=-.5ex,label=(\roman*)]
\item\label{asp:rf-honest-1} We assume that for the same $\epsilon$ in Assumption~\ref{asp:rf1},  there exists a constant $\beta>0$ such that the subsample size $s$ satisfies $\lim _{n \rightarrow \infty} s^3 h^{\beta}=\infty$, and for $\omega\in\{0,1\}$,
$$
\Big\lVert \E\Big[\rvert L^{\omega}(X_1) \rvert^{-\beta}\Biggiven D_1=\omega, D_2,\ldots,D_n,X_1,1\in\cI^\omega\Big] \Big\lVert_\infty \lesssim s^{-2-\beta/2} (\log s)^{-\beta/2}h^{(2+\beta)\epsilon},
$$
where $\lvert L \rvert$ represents the number of samples in a leaf $L$ and the subscript $b$ in both $L^\omega$ and $\cI^\omega$ is hidden.
\item \label{asp:rf1-3} The tree is honest, that is, the partition is independent of $Y_i$'s.
\end{enumerate}
\end{assumption}

\begin{assumption} \phantomsection \label{asp:rf-extr-honest}
\begin{enumerate}[itemsep=-.5ex,label=(\roman*)]
\item \label{asp:rf-extr-honest-1}
We assume that the leaf size satisfies that for $\omega\in\{0,1\}$,
$$
\Big\lVert \E\Big[\Big(\zeta_\omega(L^\omega(X_1)\cap\cM)\Big)^{-1}\Biggiven  D_1=\omega, D_2,\ldots,D_n,X_1, 1 \in \cI^{\omega}\Big]\Big\rVert_\infty
\lesssim sh^{2\epsilon}.
$$

\item \label{asp:rf-extr-honest-2}
The tree is extremely honest, that is, the partition is independent of the data.
\end{enumerate}
\end{assumption}

\begin{remark}
Assumption \ref{asp:rf1}\ref{asp:rf1-1} weakens the first part of Assumption 5.1(ii) in \cite{lin2022regression} about the subsample size $s$. Assumption \ref{asp:rf1}\ref{asp:rf1-2} as well as Assumptions \ref{asp:rf-honest} and \ref{asp:rf-extr-honest} enforce different decaying rates for the diameter of $L_t$, the number of samples in each leaf, and the subsample size $s$. They are stronger than the corresponding requirements made in Assumptions 5.1 and 5.3 of \cite{lin2022regression}, but are necessary for studying CATE, instead of ATE, estimation. These assumptions can be satisfied by properly trimming the trees while growing the forest; see also Lemma 5.1 and Proposition 5.1 in \cite{lin2022regression}. Assumption \ref{asp:rf-honest}\ref{asp:rf1-3} and Assumption \ref{asp:rf-extr-honest}\ref{asp:rf-extr-honest-2} are the famous ``honesty''  \citep{wager2018estimation} and ``purely random forests'' conditions \citep{breiman2004consistency,biau2008consistency,biau2012analysis}, both of which can be satisfied by, e.g., implementing sample splitting or employing a pre-trained random forest model. 
\end{remark}

\begin{remark}
By imposing a stronger independence condition between the partitions and the data, as compared to Assumption \ref{asp:rf-honest}, Assumption \ref{asp:rf-extr-honest} effectively eases the constraints on the leaf size, $|L_t|$. This relaxation plays a pivotal role in the subsequent derivation of a minimax optimal result for $\hat\tau_{\rm RF}(x)$. It is noteworthy that we employ the ``extremely honest'' condition primarily for its technical convenience, as elaborated in, for instance, the discussion presented in \citet[Section 3.1]{biau2016random}.
\end{remark}

\begin{remark}
As a matter of fact, under extremely honest condition, Assumptions \ref{asp:rf1}\ref{asp:rf1-2} and \ref{asp:rf-extr-honest}\ref{asp:rf-extr-honest-1} can be replaced by
\begin{align*}
    \Big\lVert \max_t{\rm diam}(L_t\cap\cM) \Big\rVert_{4m} \lesssim h^{1/2+4 \epsilon} ~~\text{and}~~ \Big\lVert \Big(\max_t \zeta_{\omega}(L_t^{\omega}\cap\cM)\Big)^{-1} \Big\rVert_{4} \lesssim sh^{2\epsilon},
\end{align*}
for $\omega \in \{0,1\}$. Here $\lVert\cdot\rVert_{p}$ stands for the $L^{p}$ norm. This is by modifying the proof of Lemmas~\ref{lemma:rf-clt}-\ref{lemma:rf-mse-discrepancy} in the appendix and conditioning on $\{L^{\omega}_{bt}\}_{t \ge 1}$ instead of just on $\{D_j\}_{j=1}^n$ and $X_i$. However, in practice, this difference appears to be not huge, and accordingly we choose to present them at their current forms.
\end{remark}


With the above assumptions held, we are then ready to introduce the first theorem about the causal forest estimator $\hat\tau_{\rm RF}$.

\begin{theorem}[Consistency] \label{thm:rf-consistency}
Assume Assumptions~\ref{asp:causal}-\ref{asp:rf1} hold, with either Assumption~\ref{asp:rf-honest} or Assumption~\ref{asp:rf-extr-honest}, and in \eqref{eq:causalf} set $\hat\mu_0=\hat\mu_1=0$. We then have, for any interior point $x\in\supp(\zeta)$,
$$
\hat\tau_{\rm RF}(x)-\tau(x)\stackrel{\sf p}{\longrightarrow} 0.
$$
\end{theorem}

Next, we establish the rate of convergence for $\hat\tau_{\rm RF}(x)$ approximating $\tau(x)$ based on the criterion of the MSE. To this end, some additional assumptions on $\mu_\omega$ and the kernel function are needed.

\begin{assumption} \phantomsection \label{asp:mse-mu}  For $\omega \in \{0,1\}$, assume $\mu_\omega$ to be Lipchitz over $\supp(\zeta)$.
\end{assumption}



\begin{assumption} \phantomsection \label{asp:mse-k}  
We assume $K$ is bounded away form zero over $\supp(K)$.
\end{assumption}

\begin{remark}
Assumption \ref{asp:mse-mu} is a stronger version of the continuity condition in Assumption \ref{asp:causal}\ref{asp:causal-3} and requires Lipschitz continuity of the regression functions. This type of smoothness is routinely posed; cf. \citet[Theorem 3.1]{wager2018estimation}. 
\end{remark}

\begin{theorem}[MSE] \label{thm:rf-mse}
Assume Assumptions~\ref{asp:causal}-\ref{asp:rf1}, \ref{asp:mse-mu}-\ref{asp:mse-k} hold, with either Assumption~\ref{asp:rf-honest} or Assumption~\ref{asp:rf-extr-honest} held, and in \eqref{eq:causalf} set $\hat\mu_0=\hat\mu_1=0$. We then have, for any interior point $x\in\supp(\zeta)$,
$$
{\rm MSE}(\hat\tau_{\rm RF}(x)):=\E\Big[\Big(\hat\tau_{\rm RF}(x)-\tau(x)\Big)^2\Big]\lesssim h+\frac{1}{nh^{m/2}},
$$
In particular, when Assumption~\ref{asp:rf-extr-honest} holds, and as we choose $h\asymp n^{-2/(m+2)}$,
it holds that the MSE admits the following minimax risk for estimating a Lipschitz function, 
$$
{\rm MSE}(\hat\tau_{\rm RF}(x)) \lesssim n^{-2/(m+2)},
$$
with the ambient dimension $d$ replaced by the intrinsic manifold dimension $m$.
\end{theorem}

Of note, when Assumption \ref{asp:rf-extr-honest}  holds, the rate in Theorem \ref{thm:rf-mse} is {\it minimax optimal} for pointwisely estimating a Lipschitz smooth function \citep{yang1999information} when the covariate resides in an $m$-dimensional space. The message conveyed here then is clear: causal forests are able to rate-optimally adapt to the unknown manifold structure.

If only an honest condition as made in Assumption \ref{asp:rf-honest} holds, the situation is a little bit more complex. In detail,  
if in Assumption \ref{asp:rf1}\ref{asp:rf1-2} we indeed have ${\rm diam}(L_t) \asymp h^{1/2+\epsilon}$, then by Assumption \ref{asp:manifold}\ref{asp:manifold-f},  the displayed constraint in  Assumption~\ref{asp:rf-honest}\ref{asp:rf-honest-1} restricts that 
\[
  \frac{n^{1-4/\beta}}{\log(n)} h^{m+2(m+1)\epsilon+ 4\epsilon/\beta}\gtrsim 1,
 \]
 with $s \asymp n$. If we we set $\epsilon$ in Assumption \ref{asp:rf1}\ref{asp:rf1-2} to be sufficiently close to 0 and $\beta$ in Assumption \ref{asp:rf-honest} to be sufficiently large, then when $m>1$, the above displayed constraint yields $h\gtrsim (nh^{m/2})^{-1}$ and thus an ${\rm MSE}(\hat\tau_{\rm RF}(x)) \lesssim h$, yielding an approximately ``best'' rate 
 $$
 (\log n)^{1/(m+2(m+1)\epsilon+ 4\epsilon/\beta)}/n^{(1-4/\beta)/(m+2(m+1)\epsilon+ 4\epsilon/\beta)},
 $$ which is apparently suboptimal. The phenomenon of obtaining a sub-optimal rate in analyzing even a simple random forest model (e.g., honest/pure ones) is again well known; see, for example, \cite{biau2012analysis} and \cite{klusowski2021sharp}. The fact that we can obtain an improved, minimax optimal, rate via adopting Assumption \ref{asp:rf-extr-honest}, on the other hand, is surprising and suggests the importance of leveraging kernel smoothing in \eqref{eq:causalf}.

Lastly, we consider the problem of statistically inferring $\tau(x)$ using $\hat\tau_{\rm RF}(x)$. This is done by {\it under-smoothing} $\hat\tau_{\rm RF}(x)$, as well as enforcing the following additional moment assumptions on the data generating distribution. Below, Assumption \ref{asp:clt-se} corresponds to Assumption 3.5 in \cite{lin2022regression}, which is put in order to leverage the Lindberg-Feller central limit theorem for analyzing $\hat\tau_{\rm RF}(x)$. Assumption \ref{asp:clt-sigma} is added for facilitating a closed form of the asymptotic variance of $\hat\tau_{\rm RF}(x)$; cf. the term $\Sigma(x)$ in Equation \eqref{eq:sigmax} ahead.

\begin{assumption}  \phantomsection \label{asp:clt-se} 
For any interior point $x \in \supp(\zeta)$ and $\omega \in \{0,1\}$, $\E [U^2_\omega(x) \given X=x]$ is uniformly bounded away from zero.
\end{assumption}

\begin{assumption} \phantomsection \label{asp:clt-sigma} 
We assume 
$\sigma^2_\omega(x) := \E [U^2_\omega(x) \given X=x] $ is Lipchitz over $\supp(\zeta)$ for $\omega \in \{0,1\}$.
\end{assumption}

\begin{theorem}[Central limit theorem] \label{thm:rf-clt}
Assume Assumptions~\ref{asp:causal}-\ref{asp:rf1}, \ref{asp:rf-extr-honest}-\ref{asp:clt-sigma} hold and in \eqref{eq:causalf} set $\hat\mu_0=\hat\mu_1=0$. Further, assume that $nh^{m/2+1}\to 0$. We then have, for any interior point $x \in \supp(\zeta)$, 
    \begin{align*}
        \sqrt{nh^{m/2}} \left(\hat\tau_{\rm RF}\left(x\right)  - \tau\left(x\right) \right) \stackrel{\sf d}{\longrightarrow} N(0,\Sigma(x)),
    \end{align*} 
where    
\begin{align}\label{eq:sigmax}
\Sigma(x):=\frac{1}{c_K^2 f(x)}\left(\frac{\sigma^2_1(x)}{e(x)} + \frac{\sigma^2_0(x)}{1-e(x)}\right)\int_{\mathbb{R}^m} K^2(\lVert t\rVert) \d t
\end{align}
with $c_K:=\int K(\|t\|)\d t$ and $e(x):=\P(D=1|X=x)$ representing the propensity score \citep{rosenbaum1983central}.
\end{theorem}

\begin{remark}
The asymptotic variance $\Sigma(x)$ can be consistently estimated combining consistent estimators of the manifold density $f(x)$, the propensity score $e(x)$, and the residual variances $\sigma_\omega^2(x)$ for $\omega=0,1$. Here, consistent estimators of manifold densities can be found in, e.g., \cite{pelletier2005kernel} and \cite{le2019approximation}. Propensity score estimation is a classic regression fitting problem, and we refer the readers to \citet[Chapter 13]{imbens2015causal} for a complimentary review. Lastly, estimation of the residual variance in a (nonparametric) regression model has been well studied in the literature; cf. \cite{wang2008effect} and \cite{shen2020optimal} and the references therein for details.
\end{remark}

\begin{remark}
In Theorems \ref{thm:rf-mse} and \ref{thm:rf-clt} and various other places in Section \ref{sec:framework} ahead, for tuning the parameters in order to obtain the optimal rate or check the assumption validity, one needs to have some knowledge about the manifold dimension, $m$. While we agree with \cite{levina2004maximum} that cross-validation should be able to automatically select the ``optimal'' tuning parameters, in theory one could use some of the existing methods for consistently estimating the manifold dimension; see, e.g., \cite{levina2004maximum}, \cite{farahmand2007manifold}, and \cite{block2022intrinsic}. 
\end{remark}

\begin{remark}
In this section, all theoretical results are demonstrated by setting $\hat\mu_0=\hat\mu_1=0$, representing misspecified mean function models. Despite this, our causal forest estimator \eqref{eq:causalf} still enjoys consistency due to its nature as a doubly robust estimator, as will be formally demonstrated in Section \ref{sec:framework}. Using random forests to impute missing potential outcomes automatically provides us with correctly specified odds ratios for the propensity score, a similar phenomenon as formally characterized in \cite{lin2021estimation} for the nearest neighbor matching algorithm and in \cite{lin2022regression} for more general cases.     
\end{remark}

To conclude this section, we make a brief observation regarding the manifold assumption itself. Although our paper primarily concentrates on scenarios where $X$ resides precisely within a smooth manifold, we posit that all the findings should, in principle, extend to cases where $X$ is distributed over an {\it approximate} manifold. This perspective aligns with existing work such as \citet[Section 2.3]{kpotufe2011k} and \citet[Assumption 3]{jiao2023deep}. Nevertheless, the elegant exact manifold assumption proves sufficient to fulfill our research objectives, and we avoid confusing our readers more by delving to more technical assumptions.

\section{A general framework}\label{sec:framework}

This section aims to establish a general framework for analyzing the CATE estimator introduced in \eqref{eq:causalf}, while bearing the potential to cover more. To this end, we adopt the framework introduced in \cite{lin2022regression}, and consider such imputation estimators that belong to the family of  {\it linear smoothers} \citep{buja1989linear}. 

Using the same notation as above, instead of imputing the missing potential outcomes as the forest-based ones, we consider a more general approach, defining 
\begin{align*}
    \hat{Y}_i(0) := \begin{cases}
        Y_i, & \mbox{ if } D_i=0,\\
        \displaystyle \sum_{j:D_j=0} w_{i\leftarrow j} (Y_j + \hat{\mu}_0(X_i) - \hat{\mu}_0(X_j)), & \mbox{ if } D_i=1,     
    \end{cases}  
\end{align*}
and
\begin{align*}
    \hat{Y}_i(1) := \begin{cases}
    \displaystyle  \sum_{j:D_j=1} w_{i\leftarrow j} (Y_j + \hat{\mu}_1(X_i) - \hat{\mu}_1(X_j)), & \mbox{ if } D_i=0,\\   
      Y_i, & \mbox{ if } D_i=1,
    \end{cases}    
\end{align*}
where $w_{i\leftarrow j}$'s are the {\it smoothing parameters} usually learnt from the data. The corresponding CATE estimator is then defined to be
$$
\hat\tau_{w}(x):=\frac{\sum_{i=1}^nK_h(X_i-x)(\hat Y_i(1)-\hat Y_i(0))}{\sum_{i=1}^nK_h(X_i-x)},
$$
with the subindex $w$ in $\hat\tau_{w}(x)$ highlighting the dependence of the general CATE on the smoothing parameters. In particular, when setting
$$
    w_{i\leftarrow j}:= \frac1B \sum_{b=1}^B \frac{\ind\Big(j \in \cI^{1-D_i}_b:X_j \in L^{1-D_i}_b(X_i)\Big)}{\Big\lvert \Big\{k \in \cI^{1-D_i}_b:X_k \in L^{1-D_i}_b(X_i)\Big\} \Big\rvert},
$$
we recover the forest-based estimator $\hat\tau_{\rm RF}(x)$.

We now provide some general regulations on the smoothing parameters $w_{i\leftarrow j}$'s and the regression adjustment terms $\hat\mu_0,\hat\mu_1$. They are in parallel to the corresponding assumptions in \cite{lin2022regression} (suggested in the parenthesis of each assumption) and are intent to be general.

\begin{assumption}[Assumption 3.2 in \cite{lin2022regression}] \phantomsection \label{asp:permutation}
    \begin{enumerate}[itemsep=-.5ex,label=(\roman*)]
        \item\label{asp:permutation-1} Consider any permutation $\pi: \zahl{n} \to \zahl{n}$. Let $[w_{i\leftarrow j}]_{D_i+D_j=1}$ and $[w^\pi_{i\leftarrow j}]_{D_i+D_j=1}$ be the weights constructed by $[(X_i,D_i,Y_i)]_{i=1}^n$ and $[(X_{\pi(i)},D_{\pi(i)},Y_{\pi(i)})]_{i=1}^n$, respectively. It is assumed that for any $i,j \in \zahl{n}$ such that $D_i+D_j=1$ and any permutation $\pi$, $w_{i\leftarrow j} = w^\pi_{\pi(i)\leftarrow \pi(j)}$ holds true.
        \item \label{asp:permutation-2} The weights satisfy
        \begin{align*}
            \sum_{j:D_j=1-D_1} w_{1\leftarrow j}=1.
        \end{align*}
    \end{enumerate}
\end{assumption}

\begin{assumption} \phantomsection \label{asp:wij} 
Assume that, $\sum_{j:D_j=1-D_i} \lvert w_{i\leftarrow j}\rvert$ is bounded.
\end{assumption}

\begin{assumption}[Assumption 3.3 in \cite{lin2022regression}] \phantomsection \label{asp:dr-propensity}
    \begin{enumerate}[itemsep=-.5ex,label=(\roman*)]
        \item\label{asp:dr-propensity-1} We assume that for $\omega \in \{0,1\}$ a deterministic function $\bar{\mu}_\omega(\cdot):\bR^d \to \bR$ exists such that (a) $\bar{\mu}_\omega(x)$ is uniformly bounded for almost all $x \in \supp(\zeta)$, and (b) the regression adjustment term $\hat{\mu}_\omega(x)$ satisfies
        $$
            \lVert \hat{\mu}_\omega - \bar{\mu}_\omega \rVert_\infty = o_\P(1),
        $$
        where $\|\cdot\|_{\infty}$ represents the infinity norm (conditional on the data).
        \item\label{asp:dr-propensity-2} For all $x \in \supp(\zeta)$, the weights satisfy that there exists $\epsilon>0$, such that
        \begin{equation}\label{asp:dr-propensity-2-sum}
            h^{d-m}\E \Big[\Big( \frac{1}{n} \sum_{i=1}^n \Big(\sum_{j:D_j=1-D_i} K_{h,j}w_{j\leftarrow i} - K_{h,i}\Big(D_i \frac{1-e(X_i)}{e(X_i)} + (1-D_i) \frac{e(X_i)}{1-e(X_i)} \Big)\Big) \Big)^2\Big]
            = o(1),
        \end{equation}
        and
        \begin{equation}\label{asp:dr-propensity-2-single}
            h^{d-m}\E \Big[ \Big(\sum_{j:D_j=1-D_1} K_{h,j}w_{j\leftarrow 1} - K_{h,1}\Big(D_1 \frac{1-e(X_1)}{e(X_1)} + (1-D_1) \frac{e(X_1)}{1-e(X_1)} \Big) \Big)^2\Big] \lesssim h^{-m/2} \Big(h^{\epsilon}+\frac{1}{n}\Big),
        \end{equation}
 where $K_{h,i}:= h^{-d/2}K(\|h^{-1/2}(X_i-x)\|)$.       
    \end{enumerate}
  \end{assumption}

\begin{assumption}[Assumption 3.4 in \cite{lin2022regression}] \phantomsection \label{asp:dr-outcome}
    \begin{enumerate}[itemsep=-.5ex,label=(\roman*)]
        \item\label{asp:dr-outcome-1} It is assumed that, for $\omega \in \{0,1\}$, the regression adjustment term $\hat{\mu}_\omega(x)$ satisfies
        $$
            \lVert \hat{\mu}_\omega - \mu_\omega \rVert_\infty = o_\P(1).
        $$
        \item\label{asp:dr-outcome-2} The weights $[w_{1\leftarrow j}]_{D_j=1-D_1}$ are constructed by $[(X_i,D_i)]_{i=1}^n$ only and not using any outcome information in $[Y_i]_{i=1}^n$. 
        
        \item\label{asp:dr-outcome-3} For all $x \in \supp(\zeta)$, the weights satisfy
        \begin{align*}
         h^{-m/2} \E\Big[\Big\lvert\sum_{D_j=1-D_1} K\left(\lVert h^{-1/2}\left(X_j-x\right)\rVert \right) w_{j \leftarrow 1}\Big\rvert\Big] \lesssim 1,
        \end{align*}
        and
        \begin{align*}
         \E\Big[\Big( h^{-m/2} \sum_{D_j=1-D_1} K\left(\lVert h^{-1/2}\left(X_j-x\right)\rVert \right) w_{j \leftarrow 1}\Big)^2\Big] \lesssim h^{-m/2}.
        \end{align*}
    \end{enumerate}
\end{assumption}

\begin{remark} Assumption \ref{asp:dr-propensity} corresponds to the case where the mean function models are misspecified but the propensity score model is correctly specified, and the parallel Assumption \ref{asp:dr-outcome} corresponds to the scenario where the mean function models are correctly specified.  
\end{remark}

With the above assumptions held, the first theorem of this section gives a {\it double robustness} result on $\hat\tau_w(\cdot)$. See, e.g., \citet[Proposition 2, Theorem 2]{kennedy2020towards} and \citet[Theorem 1]{diaz2018targeted} about similar observations for other CATE estimators.

\begin{theorem}[Double robustness of $\hat\tau_w$]  \label{thm:dr}
Suppose Assumptions~\ref{asp:causal}-\ref{asp:kernel} and \ref{asp:permutation}-\ref{asp:wij} hold, and either Assumptions~\ref{asp:dr-propensity} or Assumptions~\ref{asp:dr-outcome} is true. We then have for any interior point $x \in \supp(\zeta)$,
\begin{align*}
\hat\tau_w(x)  - \tau(x) \stackrel{\sf p}{\longrightarrow} 0.
\end{align*}
\end{theorem}

We then move on to study the rates of convergence. 

\begin{assumption}  \phantomsection \label{asp:mse-muhat} Assume that, for $\omega \in \{0,1\}$, $\hat{\mu}_\omega$ is Lipchitz over $\supp\left(\zeta\right)$.
\end{assumption}

\begin{assumption}  \phantomsection \label{asp:mse-discrepancy} There exists some constant $\gamma>1$ such that the weights satisfy
    \begin{equation*}
        \E \Big[\Big(\frac{1}{n} \sum_{i=1}^n h^{-m/2} K\left(\lVert h^{-1/2}\left(X_i-x\right)\rVert\right) \sum_{j:D_j=1-D_i} \lvert w_{i\leftarrow j} \rvert \lVert X_i-X_j \rVert\Big)^{2\gamma}\Big] \lesssim h^{\gamma}.
    \end{equation*}
\end{assumption}
\begin{remark}
    In the context of the central limit theorem, it is important to note that we can ease the Assumption~\ref{asp:mse-discrepancy} to $L^1$ convergence.
\end{remark}

\begin{assumption}  \phantomsection \label{asp:mse-weight} 
    For $\omega \in \{0,1\}$, the estimator $\hat{\mu}_\omega(x)$ satisfies
    $$
    \E\Big[\lVert \mu_\omega - \hat{\mu}_\omega \rVert_\infty^2\Big] \lesssim h.
    $$
\end{assumption}

\begin{theorem}[MSE of $\hat\tau_w(x)$]\label{thm:mse}
Suppose Assumptions~\ref{asp:causal}-\ref{asp:kernel}, \ref{asp:mse-mu}-\ref{asp:mse-k}, \ref{asp:permutation}-\ref{asp:wij}, \ref{asp:dr-outcome}\ref{asp:dr-outcome-2} hold, with either Assumptions~\ref{asp:dr-propensity}\ref{asp:dr-propensity-2}, \ref{asp:mse-muhat}-\ref{asp:mse-discrepancy} or Assumptions~\ref{asp:dr-outcome}\ref{asp:dr-outcome-3}, \ref{asp:mse-weight} held. We then have, for any interior point $x \in \supp(\zeta)$,
    \begin{align*}
        \text{MSE}\left(\hat\tau_w \left(x\right)\right)=\E \Big[\Big(\hat\tau_w(x)  - \tau(x)\Big)^2\Big] \lesssim h + \frac{1}{n h^{m/2}}.
    \end{align*}
In particular, when $h \asymp n^{-{\frac{2}{m+2}}}$,
\begin{equation} \label{thm-mse-h}
   \text{MSE}\left(\hat\tau_w \left(x\right)\right) \lesssim n^{-{\frac{2}{m+2}}}.
\end{equation}
\end{theorem}

Lastly, we consider the central limit theorem. 

\begin{assumption}  \phantomsection \label{asp:clt-phi} There exist some functions, $\phi_0(x)$ and $\phi_1(x)$, as well as $\epsilon>0$, such that for all points $x \in \supp(\zeta)$, the weights satisfy
\begin{equation*}
h^{d-m}\E \Big[\Big(\sum_{j:D_j=1-D_1} K_{h,j}w_{j\leftarrow 1} - K_{h,1}\Big(D_1 \phi_1(X_1) + (1-D_1) \phi_0(X_1)\Big) \Big)^2\Big] \lesssim h^{-m/2}\Big(h^{\epsilon} + \frac{1}{n}\Big).
\end{equation*}

Furthermore, assume that the following limit exists:
\begin{equation*}
 \lim_{n \to \infty } h^{-m/2} \E \Big[K^2\left(\lVert h^{-1/2}\left(X_1-x\right)\rVert\right)\Big[D_1\left(1+\phi_1\left(X_1\right)\right)+\left(1-D_1\right)\left(1+\phi_0\left(X_1\right)\right)\Big]^2 \sigma_{D_1}^2(X_1)\Big]=: \tilde{\Sigma}(x).
\end{equation*}
\end{assumption}

\begin{remark} We note that under Assumption~\ref{asp:dr-propensity}\ref{asp:dr-propensity-2},  $\phi_0(x)$ and $\phi_1(x)$  in Assumption~\ref{asp:clt-phi} take the following specific forms
\begin{equation*}
    \quad \phi_0(x) = \frac{e(x)}{1-e(x)} \quad \text{and}\quad \phi_1(x) = \frac{1-e(x)}{e(x)}.
\end{equation*}
For the doubly robust estimator, the propensity score function can be misspecified, in which case $\phi_0(x)$ and $\phi_1(x)$ can be general functions without the above specific forms. From a technical standpoint, a central limit theorem without closed forms for $ \phi_0(x)$ and $ \phi_1(x)$   allows for more flexible assumption on the smoothing weight. 
\end{remark}

\begin{theorem}[Central limit theorem of $\hat\tau_w(x)$]\label{thm:clt}
Suppose Assumptions~\ref{asp:causal}-\ref{asp:kernel}, \ref{asp:mse-mu}-\ref{asp:clt-sigma}, \ref{asp:permutation}-\ref{asp:wij}, \ref{asp:dr-outcome}\ref{asp:dr-outcome-2} hold with either Assumptions~\ref{asp:dr-propensity}\ref{asp:dr-propensity-2}, \ref{asp:mse-muhat}-\ref{asp:mse-discrepancy},  or Assumptions~\ref{asp:mse-weight}-\ref{asp:clt-phi} held. Furthermore, assume $nh^{m/2+1}\to 0$. We then have, for any interior point $x \in \supp(\zeta)$,
    \begin{align*}
        \sqrt{nh^{m/2}} \left(\hat\tau_w\left(x\right)  - \tau\left(x\right) \right) \stackrel{\sf d}{\longrightarrow} N(0,\Sigma(x)).
    \end{align*}
Here under Assumption~\ref{asp:dr-propensity}\ref{asp:dr-propensity-2}, 
$$
\Sigma(x)=\frac{1}{c_K^2 f(x)}\left(\frac{\sigma^2_1(x)}{e(x)} + \frac{\sigma^2_0(x)}{1-e(x)}\right)\int_{\mathbb{R}^m} K^2(\lVert t\rVert) \d t;
$$
and under Assumption~\ref{asp:clt-phi},
$$
\Sigma(x)=\frac{\tilde{\Sigma}(x)}{c_K^2 f^2(x)},
$$
which, if $\phi_0(x)$ and $\phi_1(x)$ are Lipchitz and bounded and bounded away from zero, is 
$$
\Sigma(x)=\frac{1}{c_K^2 f(x)}\Big( e(x)\left(1+\phi_1(x)\right)^2\sigma^2_1(x) + \left(1-e(x)\right)\left(1+\phi_0(x)\right)^2\sigma^2_0(x)\Big)\int_{\mathbb{R}^m} K^2(\lVert t\rVert) \d t.
$$
\end{theorem}

\section{Discussions}

In this study, we have demonstrated that causal forests can effectively adapt to unknown covariate manifold structures, under the (extremely) honesty assumption of trees in the forests. An interesting future direction is relaxing the honesty assumption and exploring data-dependent splits. We have considered fixed dimensions for the ambient space and manifold. Recent research \citep{scornet2015consistency,chi2022asymptotic} has shown that random forests \citep{breiman2001random} can adapt to model sparsity and high dimensionality, sparking interest in investigating a similar adaptability aspect in CATE estimation using random forests. This avenue of exploration holds potential for enhancing CATE estimation methodologies. 


\section*{Ackowledgement}

The authors thank Vasilis Syrgkanis for pointing out related research, and Jason M. Klusowski for helpful comments on the theory of random forests.

\appendix

\section{Key lemmas}\label{sec:key-lemma}

Starting from the appendix, without loss of generality, we are focused on a special chart that is the tangent space $T_x\cM$ combined with its projection if not specified.

In the following we put some auxiliary lemmas that we will use later in the proofs.


\begin{lemma}[Lemma 2.2, Chapter 7 in \cite{stein2009real}] \label{lemma:gbound}
Assume Assumptions~\ref{asp:manifold}\ref{asp:manifold-dim}-\ref{asp:manifold-compact}. The following two statements are equivalent:
\begin{enumerate}
    \item Assumption~\ref{asp:manifold}\ref{asp:manifold-f}.
    \item Write $g_x=\mathrm{d}\left(\psi_* \zeta\right) / \mathrm{d} \lambda$ to be the Radon-Nikodym derivative of $\psi_* \zeta_U$ with respect to $\lambda$. For any $x \in \cM$ and any chart $(U, \psi)$ defined above, $g_x$ is bounded and bounded away from zero.
\end{enumerate}
\end{lemma}

\begin{lemma} \label{lemma:fx_gphi}
Under Assumption~\ref{asp:manifold}, we have, for any interior point $x\in\supp(\zeta)$,  $f(x)=g_x(\psi(x))$.
\end{lemma}


\begin{lemma}\label{lemma:thetafunc}
    For any $x \in \cM$, let $\theta_x \left(x_i\right)$ be the angle between the vector $x_i-x$ and its projection onto $T_x\cM$. Then there exist a $U_x \subset \cM$ that includes $x$ and a constant $c_x>0$, such that for every $x_i \in U_x$, 
    $$
    \theta_x \left(x_i\right) \leq c_x \lVert x_i-x \rVert.
    $$
\end{lemma}

\begin{lemma} \label{lemma:manifold}
    Assume Assumption~\ref{asp:manifold}, and let $\psi$ be the orthogonal projection onto the tangent plane $T_x\cM$. Consider any function $Q: \bR \to \bR^{\ge 0}$ that satisfies
    \begin{enumerate}[itemsep=-.5ex,label=(\roman*)]        
        \item \label{asp-lemma:manifold-1} $Q$ is Lipchitz on its support, $\supp(Q)$, and $\supp(Q)$ is bounded;

        \item \label{asp-lemma:manifold-2} $Q$ is bounded;

        \item \label{asp-lemma:manifold-3}There exists a constant $\delta_0>0$, such that when $0<h\leq \delta_0$, for any $z \in V$, there exist $0 < c_z \leq C_z < \infty$ such that 
        $$
        c_z \leq h^{-m/2}\int_{B \left(z, h^{1/2}d_{Q}\right)} Q\left( h^{-1/2} \lVert z_i-z\rVert \right) g_x\left(z_i\right) \d\lambda(z_i) \leq C_z,
        $$
        where $c_z$ and $C_z$ are independent of $h$, and $B(x,r)$ represents the ball with center $x$ and radius $r$ (measured under the Euclidean norm).

    \end{enumerate}
    Then for any interior point $x \in \supp(\zeta)$,
    \begin{equation} \label{eq:lemma-manifold-result}
        \Big\lvert 1- \frac{\int_{U} Q\left( h^{-1/2} \lVert x_i-x\rVert \right) \d \zeta(x_i)}{\int_{V} Q\left( h^{-1/2} \lVert z_i-\psi(x)\rVert \right) g_x\left(z_i\right) \d\lambda(z_i)} \Big\rvert  \leq c_x h^{1/2}d_Q,
    \end{equation}
    where $c_x$ is independent of $h$.
\end{lemma}

\begin{lemma} \label{lemma:manifoldfunc}
    Under the same Assumptions as Lemma~\ref{lemma:manifold}, if we further have a bounded continuous function $\phi: \cM \to \mathbb{R}$, then for any interior point $x \in \supp(\zeta)$,

    \begin{enumerate}[itemsep=-.5ex,label=(\roman*)]
    \item when $\phi(x)>0$, we have
    \begin{equation} \label{lemma:manifoldfunc-1}
    \lim _{\substack{h \rightarrow 0}} \frac{\int_{U} Q\left( h^{-1/2} \lVert x_i-x\rVert \right) \phi(x_i) \d \zeta(x_i)}{\phi(x)\int_{V} Q\left( h^{-1/2} \lVert z_i-\psi(x)\rVert \right) g_x\left(z_i\right) \d\lambda(z_i)} =1, 
    \end{equation}
    and when $\phi(x)=0$,
    \begin{equation} \label{lemma:manifoldfunc-2}
    \lim _{\substack{h \rightarrow 0}} h^{-m/2}\int_{U} Q\left( h^{-1/2} \lVert x_i-x\rVert \right) \phi(x_i) \d \zeta(x_i) = 0.
    \end{equation}

    \item If we further assume $\phi$ to be (locally) Lipschitz, then when $\phi(x)>0$, we have
    \begin{equation} \label{lemma:manifoldfunc-1-h}
    \Big\lvert1- \frac{\int_{U} Q\left( h^{-1/2} \lVert x_i-x\rVert \right) \phi(x_i) \d \zeta(x_i)}{\phi(x)\int_{V} Q\left( h^{-1/2} \lVert z_i-\psi(x)\rVert \right) g_x\left(z_i\right) \d\lambda(z_i)}\Big\rvert \lesssim h^{1/2}d_Q,
    \end{equation}
    and when $\phi(x)=0$,
    \begin{equation} \label{lemma:manifoldfunc-2-h}
    \Big\lvert h^{-m/2}\int_{U} Q\left( h^{-1/2} \lVert x_i-x\rVert \right) \phi(x_i) \d \zeta(x_i) \Big\rvert \lesssim h^{1/2}d_Q,
    \end{equation}
    as $h\to 0$.
    \end{enumerate}
\end{lemma}

\section{Proofs in Section \ref{sec:framework}} \label{sec:pf-framework}

For simplicity, we denote 
$\mX:=[X_i]_{i=1}^n$ and $\mD:=[D_i]_{i=1}^n$.

\subsection{Proof of Theorem~\ref{thm:dr}}
\begin{proof}[Proof of Theorem~\ref{thm:dr}] 

{\bf Part I.} Assume the accuracy of the propensity score model, specifically, the validity of Assumption~\ref{asp:dr-propensity}. And we decompose $\hat\tau_w$ as the following five parts,
\begin{align*}
    \hat\tau_w  =  & \frac{1}{\sum_{i=1}^n K_{h, i}} \sum_{i=1}^n K_{h, i}\left[\hat{\mu}_1\left(X_i\right)-\bar{\mu}_1\left(X_i\right)\right] -\frac{1}{\sum_{i=1}^n K_{h, i}} \sum_{i=1}^n K_{h, i}\left[\hat{\mu}_0\left(X_i\right)-\bar{\mu}_0\left(X_i\right)\right] \\
    & +\frac{1}{\sum_{i=1}^n K_{h, i}}\Big[\sum_{i=1}^n D_i\Big(K_{h, i}+\sum_{D_j=1-D_i} K_{h, j} w_{j \leftarrow i} \Big)\left(\bar{\mu}_1\left(X_i\right)-\hat{\mu}_1\left(X_i\right)\right) \\
    & \quad -\sum_{i=1}^n\left(1-D_i\right)\Big(K_{h, i}+\sum_{D_j=1-D_i} K_{h,j} w_{j \leftarrow i} \Big)\left(\bar{\mu}_0\left(X_i\right)-\hat{\mu}_0\left(X_i\right)\right)\Big] \\
    & +\frac{1}{\sum_{i=1}^n K_{h, i}}\Big[\sum_{i=1}^n D_i\Big(K_{h, i}+\sum_{D j=1-D_i} K_{h,j} w_{j \leftarrow i}-\frac{K_{h, i}}{e\left(X_i\right)}\Big) \Big(Y_i - \bar{\mu}_{D_i}(X_i)\Big) \\
    & \quad \quad -\sum_{i=1}^n\left(1-D_i\right)\Big(K_{h, i}+\sum_{D j=1-D_i} K_{h,j} w_{j \leftarrow i}-\frac{K_{h, i}}{1 - e\left(X_i\right)}\Big) \Big(Y_i - \bar{\mu}_{D_i}(X_i)\Big)\Big] \\
    & +\frac{1}{\sum_{i=1}^n K_{h, i}}\Big[\sum_{i=1}^n K_{h, i}\Big(1-\frac{D_i}{e\left(X_i\right)}\Big) \bar{\mu}_1\left(X_i\right) - \sum_{i=1}^n K_{h, i}\Big(1-\frac{1-D_i}{1-e\left(X_i\right)}\Big) \bar{\mu}_0\left(X_i\right)\Big] \\
    & +\frac{1}{\sum_{i=1}^n K_{h, i}}\Big[\sum_{i=1}^n K_{h, i} \frac{D_i}{e\left(X_i\right)} Y_i-\sum_{i=1}^n K_{h, i} \frac{1-D_i}{1-e\left(X_i\right)} Y_i\Big] . \label{eq:thm-dr1-5part}
    \yestag
\end{align*}

Without loss of generality, we will only demonstrate the first half under treatment conditions. The second half under control conditions can be established using a similar approach.

For the first term in \eqref{eq:thm-dr1-5part}, according to Assumptions~\ref{asp:kernel} and \ref{asp:dr-propensity}\ref{asp:dr-propensity-1},
$$
    \Big|\frac{1}{\sum_{i=1}^n K_{h, i}} \sum_{i=1}^n K_{h, i}\left[\widehat{\mu}_1\left(X_i\right)-\bar{\mu}_1\left(X_i\right)\right]\Big| \leqslant \Big\lvert\frac{1}{\sum_{i=1}^n K_{h, i}} \sum_{i=1}^n K_{h, i}\Big\rvert \cdot \lVert \hat{\mu}_1 - \bar{\mu}_1 \rVert_\infty = o_\P(1).
$$
Therefore,
\begin{align}\label{eq:thm-dr1-p1:result}
    \frac{1}{\sum_{i=1}^n K_{h, i}} \sum_{i=1}^n K_{h, i}\left[\widehat{\mu}_1\left(X_i\right)-\bar{\mu}_1\left(X_i\right)\right]-\frac{1}{\sum_{i=1}^n K_{h, i}} \sum_{i=1}^n K_{h, i}\left[\widehat{\mu}_0\left(X_i\right)-\bar{\mu}_0\left(X_i\right)\right] = o_\P(1).
\end{align}


For the second term in \eqref{eq:thm-dr1-5part}, it is equivalent to considering the term
\begin{align*}
\frac{1}{h^{\frac{d-m}{2}} \sum_{i=1}^n K_{h, i}} h^{\frac{d-m}{2}} & \Big[\sum_{i=1}^n D_i\Big(K_{h, i}+\sum_{D_j=1-D_i} K_{h,j} w_{j \leftarrow i}\Big)\left(\bar{\mu}_1\left(X_i\right)-\widehat{\mu}_1\left(X_i\right)\right) \\
& - \sum_{i=1}^n\left(1-D_i\right)\Big(K_{h, i}+\sum_{D_j=1-D_i} K_{h,j} w_{j \leftarrow i}\Big)\left(\bar{\mu}_0\left(X_i\right)-\widehat{\mu}_0\left(X_i\right)\right)\Big].
\end{align*}
In the context of kernel density estimation, based on Lemma~\ref{lemma:manifold}, we obtain that
\begin{align*}
& \Big\lvert h^{-m/2} \E \Big[ K\left(\lVert h^{-1/2}\left(X_1-x\right)\rVert\right)\Big] - c_K g_x\left(\psi(x)\right)\Big\rvert\\
\leq & \frac{\int_U K\left( h^{-1/2} \lVert x_1-x\rVert \right) \d \zeta(x_1)}{\int_V K\left( h^{-1/2} \lVert z_1-\psi(x)\rVert \right) g_x\left(z_1\right) \d\lambda(z_1)} \Big\lvert \int K(\lVert t\rVert) g_x\left(\psi(x)+h^{1/2} t\right) \d t - \int K(\lVert t\rVert) g_x\left(\psi(x)\right) \d t\Big\rvert\\
& + \int K(\lVert t\rVert) g_x\left(\psi(x)\right) \d t\Big\rvert \Big\lvert \frac{\int_U K\left( h^{-1/2} \lVert x_1-x\rVert \right) \d \zeta(x_1)}{\int_V K\left( h^{-1/2} \lVert z_1-\psi(x)\rVert \right) g_x\left(z_1\right) \d\lambda(z_1)}-1 \Big\rvert \\
\lesssim & \int K(\lVert t\rVert) \Big\lvert g_x\left(\psi(x)+h^{1/2} t\right)- g_x\left(\psi(x)\right) \Big\rvert \d t + \Big\lvert \frac{\int_U K\left( h^{-1/2} \lVert x_1-x\rVert \right) \d \zeta(x_1)}{\int_V K\left( h^{-1/2} \lVert z_1-\psi(x)\rVert \right) g_x\left(z_1\right) \d\lambda(z_1)}-1 \Big\rvert \\
\lesssim & \int K(\lVert t\rVert) \lVert h^{1/2} t\rVert \d t + h^{1/2}d_K
\lesssim h^{1/2} \rightarrow 0,\yestag\label{eq:mse-kernelexp}
\end{align*}
where $c_K := \int K(\lVert t\rVert) \d t$.
Define $\varphi_i(x)=K\left(h^{-1/2} \lVert X_i-x\rVert\right)-\E\left[K\left(h^{-1/2} \lVert X_i-x\rVert\right)\right]$, and by Lemma~\ref{lemma:manifold},
\begin{equation}\label{eq:thm-dr:kernel2}
    \E\left[\varphi_1^2(x)\right]
\leqslant \E\left[K^2\left(h^{-1/2} \lVert X_1-x\rVert\right)\right]
\lesssim \int_V K^2\left(h^{-1/2}\lVert z_1-\psi(x)\rVert \right) g_x\left(z_1\right) \mathrm{d} \lambda\left(z_1\right)
\lesssim h^{m/2}.
\end{equation}
Therefore,
\begin{align*}
\E\Big[\Big(\frac{1}{n} \sum_{i=1}^n h^{\frac{d-m}{2}}K_{h, i}-\E[h^{\frac{d-m}{2}}K_{h, i}]\Big)^2\Big]
= \frac{1}{n h^m} \E\left[\varphi_1^2(x)\right] \lesssim \frac{1}{n h^{m/2}}. \yestag\label{eq:thm-dr:kernelvar}
\end{align*}
By combining Equations \eqref{eq:mse-kernelexp} and \eqref{eq:thm-dr:kernelvar} and applying Slutsky's theorem, we can conclude that
\begin{equation} \label{eq:thm-dr:kde}
    h^{\frac{d-m}{2}}  \frac{1}{n} \sum_{i=1}^n  K_{h, i} \stackrel{\sf p}{\longrightarrow} c_K g_x\left(\psi(x)\right) \quad \text{and} \quad \frac{1}{h^{\frac{d-m}{2}} \frac{1}{n} \sum_{i=1}^n K_{h, i}} \stackrel{\sf p}{\longrightarrow} \frac{1}{c_K g_x\left(\psi(x)\right)}.
\end{equation}
Thus, we only need to consider the following terms
\begin{align*} 
& \Big|h^{\frac{d-m}{2}} \frac{1}{n}\sum_{i=1}^n D_i\Big(K_{h, i}+\sum_{D_j=1-D_i} K_{h,j} w_{j \leftarrow i}\Big)\left(\bar{\mu}_1\left(X_i\right)-\widehat{\mu}_1\left(X_i\right)\right) \Big| \\
\leq & \Big| h^{\frac{d-m}{2}} \frac{1}{n} \sum_{i=1}^n D_i\Big(K_{h, i}+\sum_{D_j=1-D_i} K_{h,j} w_{j \leftarrow i}-\frac{K_{h,i}}{e\left(X_i\right)}\Big)\left(\bar{\mu}_1\left(X_i\right)-\hat{\mu}_1\left(X_i\right)\right)\Big| \\
& +\Big|h^{\frac{d-m}{2}} \frac{1}{n} \sum_{i=1}^n \frac{K_{h, i}}{e\left(X_i\right)} \left(\bar{\mu}_1\left(X_i\right)-\widehat{\mu}_1\left(X_i\right)\right)\Big|. \yestag\label{eq:thm-dr1-p2:2part}
\end{align*}
For the first term in \eqref{eq:thm-dr1-p2:2part}, by Assumption~\ref{asp:dr-propensity}, 
\begin{align*}
& \Big\lvert h^{\frac{d-m}{2}} \frac{1}{n} \sum_{i=1}^n D_i\Big(K_{h, i}+\sum_{D_j=1-D_i} K_{h,j} w_{j \leftarrow i}-\frac{K_{h,i}}{e\left(X_i\right)}\Big)\left(\bar{\mu}_1\left(X_i\right)-\hat{\mu}_1\left(X_i\right)\right)\Big\rvert \\
\le & \Big\lvert h^{\frac{d-m}{2}} \frac{1}{n} \sum_{i=1}^n D_i\Big(K_{h, i}+\sum_{D_j=1-D_i} K_{h,j} w_{j \leftarrow i}-\frac{K_{h,i}}{e\left(X_i\right)}\Big)\Big\rvert \lVert \widehat{\mu}_1-\bar{\mu}_1\rVert_{\infty}
= o_P(1). \yestag\label{eq:thm-dr1-p2:1term}
\end{align*}
For the second term in \eqref{eq:thm-dr1-p2:2part}, by Assumption~\ref{asp:causal}\ref{asp:causal-2} and \eqref{eq:mse-kernelexp}
\begin{align*}
& \Big| h^{\frac{d-m}{2}} \frac{1}{n} \sum_{i=1}^n \frac{K_{h, i}}{e\left(X_i\right)} \left(\bar{\mu}_1\left(X_i\right)-\widehat{\mu}_1\left(X_i\right)\right) \Big|
\leq \Big\rvert h^{\frac{d-m}{2}} \frac{1}{n} \sum_{i=1}^n \frac{K_{h, i}}{e\left(X_i\right)}\Big\rvert \lVert \widehat{\mu}_1-\bar{\mu}_1\rVert_\infty
= o_\P(1). \yestag\label{eq:thm-dr1-p2:2term}
\end{align*}
Hence, by merging \eqref{eq:thm-dr1-p2:1term} and \eqref{eq:thm-dr1-p2:2term} we have
\begin{align*}
\frac{1}{\sum_{i=1}^n K_{h, i}} & \Big[\sum_{i=1}^n D_i\Big(K_{h, i}+\sum_{D_j=1-D_i} K_{h,j} w_{j \leftarrow i}\Big)\left(\bar{\mu}_1\left(X_i\right)-\widehat{\mu}_1\left(X_i\right)\right)\\
& - \sum_{i=1}^n\left(1-D_i\right)\Big(K_{h, i}+\sum_{D_j=1-D_i} K_{h,j} w_{j \leftarrow i}\Big)\left(\bar{\mu}_0\left(X_i\right)-\widehat{\mu}_0\left(X_i\right)\right)\Big] =o_\P(1). \yestag\label{eq:thm-dr1-p2:result}
\end{align*} 


For the third term in \eqref{eq:thm-dr1-5part}, according to \eqref{eq:thm-dr:kde} and Hölder's inequality we only need to consider
\begin{align*}
    & \Big\lvert h^{\frac{d-m}{2}} \frac{1}{n}\sum_{i=1}^n D_i\Big(K_{h, i}+\sum_{D j=1-D_i} K_{h,j} w_{j \leftarrow i}-\frac{K_{h, i}}{e\left(X_i\right)}\Big)  \Big(Y_i - \bar{\mu}_{D_i}(X_i)\Big) \Big\rvert \\
    \leq & \Big\lvert h^{\frac{d-m}{2}} \frac{1}{n}\sum_{i=1}^n D_i\Big(K_{h, i}+\sum_{D j=1-D_i} K_{h,j} w_{j \leftarrow i}-\frac{K_{h, i}}{e\left(X_i\right)}\Big) \Big\rvert \lVert Y_i - \bar{\mu}_{D_i}(X_i)\rVert_\infty \\
    = & o_\P(1).
\end{align*}
Here we use the fact that 
$$
\lvert Y_i - \bar{\mu}_{D_i}(X_i)\rvert \leq  \lvert Y_i - \mu_{D_i}(X_i) \rvert + \lvert \mu_{D_i}(X_i) - \bar{\mu}_{D_i}(X_i)\rvert
$$
is universally bounded by Assumptions~\ref{asp:causal}\ref{asp:causal-3} and \ref{asp:dr-propensity}\ref{asp:dr-propensity-1}. Thus, we conclude that
\begin{align*}
    & \frac{1}{\sum_{i=1}^n K_{h, i}} \Big[\sum_{i=1}^n D_i\Big(K_{h, i}+\sum_{D j=1-D_i} K_{h,j} w_{j \leftarrow i}-\frac{K_{h, i}}{e\left(X_i\right)}\Big)  \Big(Y_i - \bar{\mu}_{D_i}(X_i)\Big) - \\
    & ~~~\sum_{i=1}^n\left(1-D_i\right)\Big(K_{h, i}+\sum_{D j=1-D_i} K_{h,j} w_{j \leftarrow i}-\frac{K_{h, i}}{1-e\left(X_i\right)}\Big)  \Big(Y_i - \bar{\mu}_{D_i}(X_i)\Big)\Big] = o_\P(1). \yestag \label{eq:thm-dr1-p3:result}
\end{align*}

For the fourth term in \eqref{eq:thm-dr1-5part}, according to Assumption~\ref{asp:dr-propensity}\ref{asp:dr-propensity-1},
$$
\E\Big[K_{h, 1}\Big(1-\frac{D_1}{e\left(X_1\right)}\Big) \bar{\mu}_1\left(X_1\right) \Big] = \E\Big[\E\Big[K_{h, 1}\Big(1-\frac{D_1}{e\left(X_1\right)}\Big) \bar{\mu}_1\left(X_1\right) \Biggiven X_1\Big]\Big]=0.
$$
Consider 
\[
\kappa_i(x):=K\left(h^{-1/2}\lVert X_i-x\rVert\right) \left(1-\frac{D_i}{e\left(X_i\right)}\right) \bar{\mu}_1\left(X_i\right). 
\]
By utilizing Lemma~\ref{lemma:manifold} and Assumption~\ref{asp:kernel}, it is known that $K$ remains bounded within its support. Furthermore, we can express
\begin{align*}
\E\left[\kappa_1^2(x)\right]
\lesssim \E\Big[K\left(h^{-1/2}\lVert X_1-x\rVert\right)\Big]
\lesssim \int_V K\left(h^{-1/2}\lVert z_1-\psi(x)\rVert\right) g_x\left(z_1\right) \mathrm{d} \lambda\left(z_1\right)
\lesssim h^{m/2}.
\end{align*}
Meanwhile, since $\{\kappa_i(x)\}_{i=1}^n$ are i.i.d., we obtain
\begin{align*}
\E\Big[\Big(h^{\frac{d-m}{2}} \frac{1}{n} \sum_{i=1}^n K_{h, i} \Big(1-\frac{D_i}{e\left(X_i\right)}\Big) \bar{\mu}_1\left(X_i\right)\Big)^2\Big]
\leq \E\Big[\Big(\frac{1}{nh^{m/2}} \sum_{i=1}^n \kappa_i(x)\Big)^2\Big]
\lesssim \frac{1}{nh^{m/2}}.
\end{align*}
Thus by \eqref{eq:thm-dr:kde}, we have
\begin{align}\label{eq:thm-dr1-p4:result}
    \frac{1}{\sum_{i=1}^n K_{h, i}}\Big[\sum_{i=1}^n K_{h, i}\Big(1-\frac{D_i}{e\left(X_i\right)}\Big) \bar{\mu}_1\left(X_i\right) - \sum_{i=1}^n K_{h, i}\Big(1-\frac{1-D_i}{1-e\left(X_i\right)}\Big) \bar{\mu}_0\left(X_i\right)\Big]  = o_\P(1).
\end{align}

For the fifth term in \eqref{eq:thm-dr1-5part}, let 
\[
\eta_i(x):=K\left(h^{-1/2}\lVert X_i-x\rVert\right) \frac{D_i}{e\left(X_i\right)} Y_i-\E\left[K\left(h^{-1/2}\lVert X_i-x\rVert\right) \frac{D_i}{e\left(X_i\right)} Y_i\right].
\]
By Assumption~\ref{asp:causal}\ref{asp:causal-3},
$$
\E[Y_1^2(1) \given X_1] = \E[\left(U_1(X_1)+\mu_1\left(X_1\right)\right)^2 \given X_1] \lesssim \E[{U_1(X_1)}^2 \given X_1] + \E[\mu_1^2\left(X_1\right)\given X_1] = O(1).
$$
Thus according to Lemma~\ref{lemma:manifold},
\begin{align*}
& \E\left[\eta_1^2(x)\right]
\leqslant \E\Big[K^2\left(h^{-1/2}\lVert X_1-x\rVert\right)\Big(\frac{D_1}{e\left(X_1\right)}\Big)^2 Y_1^2(1)\Big]
\lesssim \E\left[K^2\left(h^{-1/2}\lVert X_1-x\rVert\right) \E\left[Y_1^2 (1)\mid X_1\right]\right] \\
\lesssim & \E\Big[K^2\left(h^{-1/2}\lVert X_1-x\rVert\right)\Big]
\lesssim \int_V K^2\left(h^{-1/2}\lVert z_1-\psi(x)\rVert\right) g_x\left(z_1\right) \mathrm{d} \lambda\left(z_1\right)
\lesssim h^{m/2}.
\end{align*}
Meanwhile,
\begin{align*}
\E\Big[\Big(h^{\frac{d-m}{2}} \frac{1}{n} \sum_{i=1}^n K_{h, i} \frac{D_i}{e\left(X_i\right)} Y_i-\E\Big[h^{\frac{d-m}{2}} K_{h, i} \frac{D_i}{e\left(X_i\right)} Y_i\Big]\Big)^2\Big]
= \E\Big[\Big(\frac{1}{nh^{m/2}} \sum_{i=1}^n \eta_i(x)\Big)^2\Big] 
\lesssim \frac{1}{nh^{m/2}},
\end{align*}
which yields
\begin{equation} \label{eq:thm-dr1-p5:wlln}
h^{\frac{d-m}{2}} \frac{1}{n} \sum_{i=1}^n K_{h, i} \frac{D_i}{e\left(X_i\right)} Y_i \stackrel{\sf p}{\longrightarrow} \E \Big[h^{\frac{d-m}{2}} K_{h, 1} \frac{D_1}{e\left(X_1\right)} Y_1\Big].
\end{equation}
Furthermore, we have
\begin{align*}
& \E\Big[h^{\frac{d-m}{2}} K_{h, 1} \frac{D_1}{e\left(X_1\right)} Y_1\Big]
= \E \left[h^{\frac{d-m}{2}} K_{h, 1} \mu_1\left(X_1\right)\right] \\
= & \int h^{-m/2} K\left(h^{-1/2}\lVert x_1-x\rVert\right) \mu_1\left(x_1\right) \d \zeta(x_1)
\rightarrow c_K \mu_1(x) g_x\left(\psi(x)\right),  \yestag \label{eq:thm-dr1-p5:kr}
\end{align*}
which is because when $\mu_1(x)>0$, Lemma~\ref{lemma:manifoldfunc} reveals that
\begin{align*}
& \int h^{-m/2} K\left(h^{-1/2}\lVert x_1-x\rVert\right) \mu_1\left(x_1\right) \d \zeta(x_1) \\
\rightarrow & \mu_1(x) h^{-m/2}\int_{B\left(\psi(x), h^{1/2} d_Q\right)} K\left(h^{-1/2}\lVert z_1-\psi(x)\rVert \right) g_x\left(z_1\right) \mathrm{d} \lambda\left(z_1\right) \\
= &  \mu_1(x) \int K\left(\lVert t\rVert\right) g_x\left(\psi(x)+h^{1/2}t\right) \mathrm{d} \lambda\left(t\right)
\rightarrow c_K \mu_1(x) g_x\left(\psi(x)\right),
\end{align*}
and when $\mu_1(x)=0$, Lemma~\ref{lemma:manifoldfunc} shows that
\begin{equation}\label{eq:thm-dr1-p5:kr0}
    \lim _{h \rightarrow 0} h^{-m/2} \int K\left(h^{-1/2}\lVert x_1-x\rVert \right) \mu_1\left(x_1\right) \mathrm{d} \zeta(x_1)=0.
\end{equation}
By combining Equations \eqref{eq:thm-dr:kde}, \eqref{eq:thm-dr1-p5:wlln}, \eqref{eq:thm-dr1-p5:kr}, and \eqref{eq:thm-dr1-p5:kr0}, we can derive
\begin{align}\label{eq:thm-dr1-p5:result}
    \frac{1}{\sum_{i=1}^n K_{h, i}}\Big[\sum_{i=1}^n \frac{K_{h, i} D_i}{e\left(X_i\right)} Y_i-\sum_{i=1}^n \frac{K_{h, i}\left(1-D_i\right)}{1-e\left(X_i\right)} Y_i\Big] \stackrel{\sf p}{\longrightarrow} \E\Big[Y_i(1)-Y_i(0) \Biggiven X=x\Big] = \tau(x).
\end{align}

Substituting Equations \eqref{eq:thm-dr1-p1:result}, \eqref{eq:thm-dr1-p2:result}, \eqref{eq:thm-dr1-p3:result}, \eqref{eq:thm-dr1-p4:result}, and \eqref{eq:thm-dr1-p5:result} into Equation \eqref{eq:thm-dr1-5part} finalizes the proof.
 
{\bf Part II.} Assume the accuracy of the outcome model, specifically, the validity of Assumption~\ref{asp:dr-outcome}. We can decompose $\hat\tau_w$ as the following four parts:
\begin{align*}
    \hat\tau_w
    = & \frac{1}{\sum_{i=1}^n K_{h, i}} \sum_{i=1}^n K_{h, i}\left[\hat{\mu}_1\left(X_i\right)-\mu_1\left(X_i\right)\right] -\frac{1}{\sum_{i=1}^n K_{h, i}} \sum_{i=1}^n K_{h, i}\left[\hat{\mu}_0\left(X_i\right)-\mu_0\left(X_i\right)\right] \\
    & +\frac{1}{\sum_{i=1}^n K_{h, i}}\Big[\sum_{i=1}^n D_i\Big(K_{h, i}+\sum_{D_j=1-D_i} K_{h,j} w_{j \leftarrow i} \Big)\left(\mu_1\left(X_i\right)-\hat{\mu}_1\left(X_i\right)\right)\Big. \\
    & \quad \Big.-\sum_{i=1}^n\left(1-D_i\right)\Big(K_{h, i}+\sum_{D_j=1-D_i} K_{h,j} w_{j \leftarrow i} \Big)\left(\mu_0\left(X_i\right)-\hat{\mu}_0\left(X_i\right)\right)\Big] \\
    & +\frac{1}{\sum_{i=1}^n K_{h, i}}\Big[\sum_{i=1}^n D_i\Big(K_{h, i}+\sum_{D_j=1-D_i} K_{h,j} w_{j \leftarrow i} \Big)\left(Y_i-\mu_1\left(X_i\right)\right)\Big. \\
    & \quad \Big.-\sum_{i=1}^n\left(1-D_i\right)\Big(K_{h, i}+\sum_{D_j=1-D_i} K_{h,j} w_{j \leftarrow i} \Big)\left(Y_i-\mu_0\left(X_i\right)\right)\Big] \\
    & + \frac{1}{\sum_{i=1}^n K_{h, i}} \sum_{i=1}^n K_{h, i}\left[\mu_1\left(X_i\right)-\mu_0\left(X_i\right)\right]. \yestag \label{eq:thm-dr2-4part}
\end{align*}

Just as in the case of Equation \eqref{eq:thm-dr1-5part}, we will only illustrate the first part under treatment conditions.


For the first term in \eqref{eq:thm-dr2-4part}, according to Assumption~\ref{asp:dr-outcome}\ref{asp:dr-outcome-1},
$$
    \Big\lvert\frac{1}{\sum_{i=1}^n K_{h, i}} \sum_{i=1}^n K_{h, i}\left[\widehat{\mu}_1\left(X_i\right)-\mu_1\left(X_i\right)\right]\Big\rvert
    \leq \Big\lvert\frac{1}{\sum_{i=1}^n K_{h, i}} \sum_{i=1}^n K_{h, i}\Big\rvert \cdot \lVert \hat{\mu}_1 - \mu_1 \rVert_\infty = o_\P(1).
$$
Therefore,
\begin{align}\label{eq:thm-dr2-p1:result}
    \frac{1}{\sum_{i=1}^n K_{h, i}} \sum_{i=1}^n K_{h, i}\left[\widehat{\mu}_1\left(X_i\right)-\mu_1\left(X_i\right)\right]-\frac{1}{\sum_{i=1}^n K_{h, i}} \sum_{i=1}^n K_{h, i}\left[\widehat{\mu}_0\left(X_i\right)-\mu_0\left(X_i\right)\right] = o_\P(1).
\end{align}

For the second term in \eqref{eq:thm-dr2-4part}, according to Assumption~\ref{asp:dr-outcome} and \eqref{eq:mse-kernelexp}, 
\begin{align*}
&\Big\lvert h^{\frac{d-m}{2}} \frac{1}{n} \sum_{i=1}^n D_i\Big(K_{h, i}+\sum_{D_j=1-D_i} K_{h,j} w_{j \leftarrow i}\Big)\left(\mu_1\left(X_i\right)-\widehat{\mu}_1\left(X_i\right)\right)\Big\rvert \\
\le & h^{\frac{d-m}{2}} \frac{1}{n}\sum_{i=1}^n D_i K_{h, i} \lVert \widehat{\mu}_1-\mu_1\rVert_\infty + h^{\frac{d-m}{2}}\frac{1}{n} \sum_{i=1}^n \Big\lvert \sum_{D_j=1-D_i} K_{h,j} w_{j \leftarrow i}\Big\rvert \lVert \widehat{\mu}_1-\mu_1\rVert_\infty
= o_\P(1).\yestag \label{eq:thm-dr2-p2:2term}
\end{align*}
Therefore,
\begin{align*}
\frac{1}{\sum_{i=1}^n K_{h, i}} & \Big[\sum_{i=1}^n D_i\Big(K_{h, i}+\sum_{D_j=1-D_i} K_{h,j} w_{j \leftarrow i}\Big)\left(\mu_1\left(X_i\right)-\widehat{\mu}_1\left(X_i\right)\right)\Big. \\
& \Big.- \sum_{i=1}^n\left(1-D_i\right)\Big(K_{h, i}+\sum_{D_j=1-D_i} K_{h,j} w_{j \leftarrow i}\Big)\left(\mu_0\left(X_i\right)-\widehat{\mu}_0\left(X_i\right)\right)\Big] =o_\P(1). \yestag\label{eq:thm-dr2-p2:result}
\end{align*}


For the third term in \eqref{eq:thm-dr2-4part}, define
$$
E_{n,i}(x) := h^{\frac{d-m}{2}}\left(2 D_i-1\right)\Big(K_{h, i}+\sum_{j: D_j=1-D_i} K_{h,j} w_{j \leftarrow i}\Big)\left(Y_i-\mu_1\left(X_i\right)\right).
$$
We notice that for each $n$, condition on $[(X_i,D_i)]_{i=1}^n$, $[E_{n,i}]_{i=1}^n$ are independent. To apply the weak law of triangular arrays (Theorem 2.2.11 in \cite{durrett2019probability}), we only need to prove
\begin{equation} \label{eq:thm-dr2-p3:triang-1}
\sum_{i=1}^n \P\left(\left|E_{n,i}(x)\right|>n \Biggiven \mX, \mD \right) \rightarrow 0
\end{equation}
and
\begin{equation} \label{eq:thm-dr2-p3:triang-2}
\frac{1}{n^2} \sum_{i=1}^n \E \left[E^2_{n,i}(x)\ind\Big(\lvert E_{n,i}(x) \rvert \leq n \Big) \Biggiven \mX, \mD \right] \rightarrow 0.
\end{equation}
According to Assumption~\ref{asp:dr-outcome}\ref{asp:dr-outcome-2},
\begin{equation} \label{eq:thm-dr2-p3:triang-mean0}
    \E\Big[E_{n,i}(x)\Big]=\E\Big[\E\Big[E_{n,i}(x)\Biggiven \mX, \mD\Big]\Big] =0.
\end{equation}
Furthermore, utilizing Hölder's inequality we have
\begin{align*}
& \E\Big[ \lvert E_{n,i}(x)\rvert \Biggiven \mX, \mD  \Big] \\
= & h^{\frac{d-m}{2}}\Big\lvert K_{h, i}+\sum_{D_j=1-D_i} K_{h,j} w_{j \leftarrow i}\Big\rvert\E \Big[ \lvert Y_i-\mu_1\left(X_i\right)\rvert \Biggiven X_i, D_i=1 \Big] \\
\le & h^{\frac{d-m}{2}} \Big\lvert K_{h, i}+\sum_{D_j=1-D_i} K_{h,j} w_{j \leftarrow i}\Big\rvert \lVert \sigma_1(X_i) \rVert_\infty,
\end{align*}
where by Assumption~\ref{asp:causal}\ref{asp:causal-3} $\sigma_\omega^2(x)$ is bounded for $\omega \in \{0,1\}$. And according to Assumption~\ref{asp:dr-outcome}\ref{asp:dr-outcome-3},
\begin{align*}
\E\Big[\lvert E_{n,i}(x)\rvert\Big]
\leq \E \Big[h^{\frac{d-m}{2}}K_{h, i}+ \Big\lvert h^{\frac{d-m}{2}}\sum_{D_j=1-D_1} K_{h,j} w_{j \leftarrow i} \Big\rvert\Big] \lVert \sigma_1 \rVert_\infty
= O(1). \yestag \label{eq:thm-dr2-p3:triang-l1}
\end{align*}
For \eqref{eq:thm-dr2-p3:triang-1}, we obtain that
\begin{align*}
& \sum_{i=1}^n \P\left(\left|E_{n,i}(x)\right|>n \Biggiven \mX, \mD \right)
= \frac{1}{n} \sum_{i=1}^n n \P\left(\lvert E_{n,i}(x)\rvert>n \Biggiven \mX, \mD \right) \\
\leq & \frac{1}{n} \sum_{i=1}^n \E\Big[ \lvert E_{n,i}(x)\rvert \ind\Big(\lvert E_{n,i}(x) \rvert > n \Big) \Biggiven \mX, \mD  \Big]
= o_\P(1),
\end{align*}
where the last step is because by \eqref{eq:thm-dr2-p3:triang-l1} we have
\begin{align*}
\E\Big[ \Big\lvert \frac{1}{n} \sum_{i=1}^n \E\Big[ \lvert E_{n,i}(x)\rvert \ind\Big(\lvert E_{n,i}(x) \rvert > n \Big) \Biggiven \mX, \mD  \Big] \Big\rvert \Big]
= \E\Big[ \lvert E_{n,1}(x)\rvert \ind\Big(\lvert E_{n,1}(x) \rvert > n \Big)  \Big]
\rightarrow 0.
\end{align*}
For \eqref{eq:thm-dr2-p3:triang-2}, we notice that
\begin{align*}
\E \left[E^2_{n,i}(x) \Biggiven \mX, \mD \right]
= h^{d-m} \Big(K_{h, i}+\sum_{j: D_j=1-D_i} K_{h,j} w_{j \leftarrow i}\Big)^2 \sigma_{D_i}^2(X_i).
\end{align*}
Thus, according to the proof of \eqref{eq:thm-dr:kernel2} and Assumption~\ref{asp:dr-outcome}\ref{asp:dr-outcome-3}, we have
\begin{align*}
& \E \Big[ \frac{1}{n^2} \sum_{i=1}^n \E \left[E^2_{n,i}(x)\ind\Big(\lvert E_{n,i}(x) \rvert \leq n \Big) \Biggiven \mX, \mD \right] \Big]
\leq \frac{1}{n} \E \left[ E^2_{n,1}(x)\right] \\
\leq & \frac{2}{n} \E \left[h^{d-m}K^2_{h, 1} \right] \lVert \sigma_1^2\rVert_\infty + \frac{2}{n} \E \Big[h^{d-m}\Big(\sum_{j: D_j=1-D_1} K_{h,j} w_{j \leftarrow 1}\Big)^2  \Big] \lVert \sigma_1^2\rVert_\infty \\
\lesssim & \frac{1}{n} \E \left[h^{d-m}K^2_{h, 1} \right] + \frac{1}{n} h^{-m/2} \\
\lesssim & \frac{1}{nh^{m/2}}.
\end{align*}
Accordingly, using the weak law of triangular arrays (Theorem 2.2.11 in \cite{durrett2019probability}), we obtain that for any $\varepsilon > 0$,
$$
\lim _{n \rightarrow \infty} P\Big(\Big|\frac{1}{n} \sum_{i=1}^n E_{n,i}(x)\Big|>\varepsilon \given \mX, \mD\Big)=0.
$$
And by dominated convergence theorem,
\begin{align*}
& \lim _{n \rightarrow \infty} \P\Big(\Big|\frac{1}{n} \sum_{i=1}^n E_{n,i}(x)\Big|>\varepsilon \Big)
= \lim _{n \rightarrow \infty} \E\Big[\E\Big[\ind\Big(\Big|\frac{1}{n} \sum_{i=1}^n E_{n,i}(x)\Big|>\varepsilon \Big)\given \mX, \mD\Big]\Big] \\
= & \E\Big[ \lim _{n \rightarrow \infty} 
 \E\Big[\ind\Big(\Big|\frac{1}{n} \sum_{i=1}^n E_{n,i}(x)\Big|>\varepsilon \Big)\given \mX, \mD\Big]\Big]
= \E\Big[ \lim _{n \rightarrow \infty} \P\Big(\Big|\frac{1}{n} \sum_{i=1}^n E_{n,i}(x)\Big|>\varepsilon \given \mX, \mD\Big)\Big]
=0.
\end{align*}
Therefore,
\begin{align*}
\frac{1}{\sum_{i=1}^n K_{h, i}} & 
 \Big[\sum_{i=1}^n D_i\Big(K_{h, i}+\sum_{D_j=1-D_i} K_{h,j} w_{j \leftarrow i}\Big)\left(Y_i-\mu_1\left(X_i\right)\right) \\
& -\sum_{i=1}^n\left(1-D_i\right)\Big(K_{h, i}+\sum_{D_j=1-D_i} K_{h,j} w_{j \leftarrow i}\Big)\left(Y_i-\mu_0\left(X_i\right)\right)\Big] = o_\P(1). \yestag \label{eq:thm-dr2-p3:result}
\end{align*}

For the fourth term in \eqref{eq:thm-dr2-4part}, similar to the process of term \eqref{eq:thm-dr1-p5:result} and let $\eta_i(x)$ here be
$
K\left(h^{-1/2}\lVert X_i-x\rVert\right) \mu_1(X_i)-\E\left[K\left(h^{-1/2}\lVert X_i-x\rVert\right) \mu_1(X_i)\right]
$.
By the continuity and boundedness of $\mu_1(x)$ and $f(x)$, similar to \eqref{eq:thm-dr1-p5:wlln} and \eqref{eq:thm-dr1-p5:kr}, we have
\begin{equation} \label{eq:thm-dr2-p5:wlln}
h^{\frac{d-m}{2}} \frac{1}{n} \sum_{i=1}^n K_{h, i} \mu_1\left(X_i\right) \stackrel{\sf p}{\longrightarrow} h^{\frac{d-m}{2}} \E \left[K_{h, i} \mu_1\left(X_i\right)\right] \longrightarrow c_K\mu_1\left(x\right) g_x\left(\psi(x)\right).
\end{equation}
By combining Equations \eqref{eq:thm-dr:kde} and \eqref{eq:thm-dr2-p5:wlln}, we have
\begin{align}\label{eq:thm-dr2-p5:result}
    \frac{1}{\sum_{i=1}^n K_{h, i}}\sum_{i=1}^n K_{h, i}\left[\mu_1\left(X_i\right)-\mu_0\left(X_i\right)\right] \stackrel{\sf p}{\longrightarrow} \tau(x).
\end{align}

Substituting Equations \eqref{eq:thm-dr2-p1:result}, \eqref{eq:thm-dr2-p2:result}, \eqref{eq:thm-dr2-p3:result}, and \eqref{eq:thm-dr2-p5:result} into Equation \eqref{eq:thm-dr2-4part} concludes the proof.
\end{proof}

\subsection{Proof of Theorem~\ref{thm:mse}}

\begin{proof}[Proof of Theorem~\ref{thm:mse}]
Define $\epsilon_i := Y_i - \mu_{D_i}(X_i)$ for $i \in \zahl{n}$. We first decompose $\hat\tau_w$ as the following four parts:
\begin{align*}
\hat\tau_w(x) 
= & \frac{1}{\sum_{i=1}^n K_{h, i}} \sum_{i=1}^n K_{h, i} \Big[\mu_1(X_i) - \mu_0(X_i)\Big] \\
& + \frac{1}{\sum_{i=1}^n K_{h, i}} \sum_{i=1}^n\left(2 D_i-1\right)\Big(K_{h, i}+\sum_{j: D_j=1-D_i} K_{h, j} w_{j \leftarrow i}\Big)\epsilon_i\\
& + \frac{1}{\sum_{i=1}^n K_{h, i}} \sum_{i=1}^n\left(2 D_i-1\right)K_{h, i} \Big[\sum_{j:D_j=1-D_i} w_{i\leftarrow j} \mu_{1-D_i}(X_i)-\sum_{j:D_j=1-D_i} w_{i\leftarrow j} \mu_{1-D_i}(X_j)\Big]\\
& - \frac{1}{\sum_{i=1}^n K_{h, i}} \sum_{i=1}^n\left(2 D_i-1\right)K_{h, i} \Big[\sum_{j:D_j=1-D_i} w_{i\leftarrow j} \hat{\mu}_{1-D_i}(X_i)-\sum_{j:D_j=1-D_i} w_{i\leftarrow j} \hat{\mu}_{1-D_i}(X_j)\Big]\\
=:& \bar{\tau}(\mX, x) + F_n(x) + B_n(x) - \hat{B}_n(x).
\end{align*}

Therefore,
\begin{equation} \label{eq:mse3part}
\E \Big[\Big(\hat\tau_w(x)  - \tau(x)\Big)^2\Big] \lesssim \E \Big[\Big(\bar{\tau}(\mX, x)  - \tau(x)\Big)^2\Big] + \E \Big[F_n^2(x)\Big] + \E \Big[\Big(B_n(x) - \hat{B}_n(x)\Big)^2\Big].
\end{equation}

\begin{lemma}[inverse moment of kernel function] \label{lemma:inverse-kernel}
Assuming Assumptions~\ref{asp:causal}-\ref{asp:kernel} and \ref{asp:mse-k} hold, we then have, for any $\alpha>2$,
$$
\E\Big[\Big\lvert\frac{1}{\frac{1}{n}h^{-m/2} \sum_{i=1}^n K\Big(h^{-1/2}\lVert X_i-x\rVert\Big)} - \frac{1}{c_K g_x\left(z\right)}\Big\rvert^{\alpha}\Big] \lesssim \Big(h+\frac{1}{nh^{m/2}}\Big)^{\alpha/2}.
$$
\end{lemma}

\begin{lemma}\label{lemma-mse:tau.bar}
Suppose Assumptions~\ref{asp:causal}-\ref{asp:kernel}, \ref{asp:mse-mu}-\ref{asp:mse-k} and \ref{asp:permutation} hold. We then have, for any interior point $x \in \supp(\zeta)$,
\begin{equation*}
    \E \Big[\Big(\bar{\tau}(\mX, x)  - \tau(x)\Big)^2\Big] = O(h).
\end{equation*}
\end{lemma}

\begin{lemma}\label{lemma-mse:fn}
Suppose Assumptions~\ref{asp:causal}-\ref{asp:kernel}, \ref{asp:mse-mu}-\ref{asp:mse-k}, \ref{asp:permutation}, \ref{asp:dr-outcome}\ref{asp:dr-outcome-2} hold, with either Assumption~\ref{asp:dr-propensity}\ref{asp:dr-propensity-2} or Assumption~\ref{asp:dr-outcome}\ref{asp:dr-outcome-3} held. We then have, for any interior point $x \in \supp(\zeta)$,
\begin{equation*}
    \E \Big[F_n^2(x)\Big] = O\Big(\frac{1}{nh^{m/2}}\Big) + o(h).
\end{equation*}
\end{lemma}

\begin{lemma}\label{lemma-mse:b}
Suppose Assumptions~\ref{asp:causal}-\ref{asp:kernel}, \ref{asp:mse-mu}-\ref{asp:mse-k}, \ref{asp:permutation}-\ref{asp:wij}, \ref{asp:dr-outcome}\ref{asp:dr-outcome-2} hold, with either Assumptions~\ref{asp:mse-muhat}-\ref{asp:mse-discrepancy} or Assumption~\ref{asp:mse-weight} held. We then have, for any interior point $x \in \supp(\zeta)$,
\begin{equation*}
    \E \Big[\Big(B_n(x) - \hat{B}_n(x)\Big)^2\Big] = O(h).
\end{equation*}
\end{lemma}

Combining Lemmas~\ref{lemma-mse:tau.bar}-\ref{lemma-mse:b}, we have
\begin{equation}\label{eq:mse-resulth}
    \text{MSE}\left(\tau_w \left(x\right)\right)=\E \Big[\Big(\hat\tau_w(x)  - \tau(x)\Big)^2\Big] \lesssim h + \frac{1}{nh^{m/2}}.
\end{equation}
This concludes the proof.
\end{proof}

\subsection{Proof of Theorem~\ref{thm:clt}}
\begin{proof}[Proof of Theorem~\ref{thm:clt}]
First, we decompose $\hat\tau_w$ as the following:
\begin{align*}
\hat\tau_\omega(x)
= & \frac{1}{\sum_{i=1}^n K_{h, i}} \sum_{i=1}^n K_{h, i} \Big[\mu_1(X_i) - \mu_0(X_i)\Big] \\
& + \frac{1}{\sum_{i=1}^n K_{h, i}} \sum_{i=1}^n\left(2 D_i-1\right)K_{h, i} \left(D_i \left(1+\phi_1(X_i)\right) + (1-D_i) \left(1+\phi_0(X_i)\right)\right)\epsilon_i\\
& + \frac{1}{\sum_{i=1}^n K_{h, i}} \sum_{i=1}^n\left(2 D_i-1\right)\Big[\sum_{j:D_j=1-D_i} K_{h, j}w_{j\leftarrow 1} - K_{h, i}\Big(D_i \phi_1(X_i) + (1-D_i) \phi_0(X_i)\Big)\Big]\epsilon_i\\
& + \frac{1}{\sum_{i=1}^n K_{h, i}} \sum_{i=1}^n\left(2 D_i-1\right)K_{h, i} \Big[\sum_{j:D_j=1-D_i} w_{i\leftarrow j} \mu_{1-D_i}(X_i)-\sum_{j:D_j=1-D_i} w_{i\leftarrow j} \mu_{1-D_i}(X_j)\Big]\\
& - \frac{1}{\sum_{i=1}^n K_{h, i}} \sum_{i=1}^n\left(2 D_i-1\right)K_{h, i} \Big[\sum_{j:D_j=1-D_i} w_{i\leftarrow j} \hat{\mu}_{1-D_i}(X_i)-\sum_{j:D_j=1-D_i} w_{i\leftarrow j} \hat{\mu}_{1-D_i}(X_j)\Big]\\
=:& \bar{\tau}(\mX, x) + E_n(x) + \tilde{E}_n(x) + B_n(x) - \hat{B}_n(x) \yestag\label{eq:clt-4part},
\end{align*}
where under Assumption~\ref{asp:dr-propensity}\ref{asp:dr-propensity-2}, it is assumed that $\phi_1(x) = \frac{1-e(x)}{e(x)}$ and $\phi_0(x) = \frac{e(x)}{1-e(x)}$.


\begin{lemma}\label{lemma-clt:tau.bar}
Assume Assumptions~\ref{asp:causal}-\ref{asp:kernel}, \ref{asp:mse-mu} and \ref{asp:permutation} hold. Further, assume $n h^{m/2+1} \rightarrow 0$. We then have, for any interior point $x \in \supp(\zeta)$,
\begin{equation*}
    \sqrt{nh^{m/2}} \Big(\bar{\tau}(\mX, x)  - \tau(x)\Big) \stackrel{\sf p}{\longrightarrow} 0.
\end{equation*}
\end{lemma}

\begin{lemma}\label{lemma-clt:en}
Assume Assumptions~\ref{asp:causal}-\ref{asp:kernel}, \ref{asp:mse-mu}, \ref{asp:clt-se}-\ref{asp:clt-sigma}, \ref{asp:permutation} and \ref{asp:dr-outcome}\ref{asp:dr-outcome-2} hold with either Assumption~\ref{asp:dr-propensity}\ref{asp:dr-propensity-2} or Assumption~\ref{asp:clt-phi} held. Further, assume $n h^{m/2+1} \rightarrow 0$. We then have, for any interior point $x \in \supp(\zeta)$,
    \begin{align*}
        \sqrt{nh^{m/2}} E_n(x) \stackrel{\sf d}{\longrightarrow} N(0,\Sigma(x)).
    \end{align*}
Here under Assumption~\ref{asp:dr-propensity}\ref{asp:dr-propensity-2}, 
$$
\Sigma(x)=\frac{1}{c_K g_x\left(\psi\left(x\right)\right)}\left(\frac{\sigma^2_1(x)}{e(x)} + \frac{\sigma^2_0(x)}{1-e(x)}\right)\int_{\mathbb{R}^m} K^2(\lVert t\rVert) \d t,
$$
and under Assumption~\ref{asp:clt-phi},
$$
\Sigma(x)=\frac{\tilde{\Sigma}(x)}{\Big[c_K g_x(\psi(x))\Big]^2},
$$
which, if $\phi_0(x)$ and $\phi_1(x)$ are Lipchitz functions that bounded and bounded away from zero, is 
$$
\Sigma(x)=\frac{1}{c_K g_x(\psi(x))}\Big( e(x)\left(1+\phi_1(x)\right)^2\sigma^2_1(x) + \left(1-e(x)\right)\left(1+\phi_0(x)\right)^2\sigma^2_0(x)\Big)\int_{\mathbb{R}^m} K^2(\lVert t\rVert) \d t.
$$
\end{lemma}


\begin{lemma}\label{lemma-clt:en-t}
Assume Assumptions~\ref{asp:causal}-\ref{asp:kernel}, \ref{asp:permutation}, \ref{asp:dr-outcome}\ref{asp:dr-outcome-2} hold with either Assumption~\ref{asp:dr-propensity}\ref{asp:dr-propensity-2} or Assumption~\ref{asp:clt-phi} held. Further, assume $n h^{m/2+1} \rightarrow 0$. We then have, for any interior point $x \in \supp(\zeta)$,
    \begin{equation*}
        \sqrt{nh^{m/2}} \tilde{E}_n(x)\stackrel{\sf p}{\longrightarrow} 0.
    \end{equation*}
\end{lemma}

\begin{lemma}\label{lemma-clt:b}
Assume Assumptions~\ref{asp:causal}-\ref{asp:kernel}, \ref{asp:mse-mu}, \ref{asp:permutation} and \ref{asp:dr-outcome}\ref{asp:dr-outcome-2} hold with either Assumptions~\ref{asp:mse-muhat}-\ref{asp:mse-discrepancy} or Assumption~\ref{asp:mse-weight} held. Further, assume $n h^{m/2+1} \rightarrow 0$. We then have, for any interior point $x \in \supp(\zeta)$,
\begin{equation*}
    \sqrt{nh^{m/2}} \Big(B_n(x) - \hat{B}_n(x) \Big)\stackrel{\sf p}{\longrightarrow} 0.
\end{equation*}
\end{lemma}

Combining Lemma~\ref{lemma-clt:tau.bar}-\ref{lemma-clt:b} and \eqref{eq:clt-4part}, we complete the proof.
\end{proof}

\section{Proofs in Section \ref{sec:theory}} \label{sec:pf-theory}

For random forests, the weights can be written as
\begin{equation} \label{eq:rf-w}
w_{i \leftarrow j}:=\frac{1}{B} \sum_{b=1}^B \frac{\ind\left(j \in \mathcal{I}_b^{1-D_i}: X_j \in L_b^{1-D_i}\left(X_i\right)\right)}{\Big\lvert\left\{k \in \mathcal{I}_b^{1-D_i}: X_k \in L_b^{1-D_i}\left(X_i\right)\right\}\Big\rvert}.
\end{equation}

\subsection{Proof of Theorem~\ref{thm:rf-consistency}}
\begin{proof}[Proof of Theorem~\ref{thm:rf-consistency}]

\begin{lemma} \label{lemma:rf-clt}
Assume Assumptions~\ref{asp:causal}-\ref{asp:rf1} hold,  with either Assumption~\ref{asp:rf-honest} or Assumption~\ref{asp:rf-extr-honest} held, then
for $\omega\in\{0,1\}$,
\begin{align*}
    & \Big\lVert\E\Big[\Big(\frac{s}{\lvert L^{\omega}(X_1)\rvert} \zeta_\omega(L^{\omega}(X_1) \cap\cM)-1\Big)^2\Biggiven \mD, D_1=\omega,X_1,1\in\cI^\omega\Big]\Big\rVert_\infty = O(h^{2\epsilon}),\\
    & \Big\lVert\E\Big[\Big(\frac{s}{\lvert L^{\omega}(X_1)\rvert} \zeta_{1-\omega}(L^{\omega}(X_1) \cap\cM)\Big)^2\Biggiven \mD, D_1=\omega,X_1,1\in\cI^\omega\Big]\Big\rVert_\infty = O(1).
\end{align*}
\end{lemma}

\begin{lemma} [Theorem 5.1(i) in \cite{lin2022regression}] \label{lemma:rf-dr-permutation}
Assume Assumptions~\ref{asp:causal} and \ref{asp:kernel} hold. We then have Assumption
~\ref{asp:permutation} holds.
\end{lemma}

\begin{lemma} \label{lemma:rf-dr-propensity-2-sum}
Assume Assumptions~\ref{asp:causal}-\ref{asp:rf1} hold, with either Assumption~\ref{asp:rf-honest} or Assumption~\ref{asp:rf-extr-honest} held. We then obtain \eqref{asp:dr-propensity-2-sum} in Assumption~\ref{asp:dr-propensity}\ref{asp:dr-propensity-2}.
\end{lemma}

\begin{lemma} \label{lemma:rf-dr-propensity-2-single}
Assume Assumptions~\ref{asp:causal}-\ref{asp:rf1} hold, with either Assumption~\ref{asp:rf-honest} or Assumption~\ref{asp:rf-extr-honest} held. We then obtain \eqref{asp:dr-propensity-2-single} in Assumption~\ref{asp:dr-propensity}\ref{asp:dr-propensity-2}.
\end{lemma}

In \eqref{eq:causalf} setting $\hat\mu_0=\hat\mu_1=0$, Lemma~\ref{lemma:rf-dr-permutation}-\ref{lemma:rf-dr-propensity-2-single} and Theorem~\ref{thm:dr} yields Theorem~\ref{thm:rf-consistency}.
\end{proof}

\subsection{Proof of Theorem~\ref{thm:rf-mse}}
\begin{proof}[Proof of Theorem~\ref{thm:rf-mse}]

\begin{lemma} \label{lemma:rf-mse-discrepancy}
    Assume Assumptions~\ref{asp:causal}-\ref{asp:rf1} hold. We then have Assumption~\ref{asp:mse-discrepancy} hold.
\end{lemma}

In \eqref{eq:causalf} setting $\hat\mu_0=\hat\mu_1=0$, Lemmas~\ref{lemma:rf-dr-permutation}-\ref{lemma:rf-mse-discrepancy} and Theorem~\ref{thm:mse} yields Theorem~\ref{thm:rf-mse}.
\end{proof}

\subsection{Proof of Theorem~\ref{thm:rf-clt}}
\begin{proof}[Proof of Theorem~\ref{thm:rf-clt}]
In \eqref{eq:causalf} setting $\hat\mu_0=\hat\mu_1=0$, Lemmas~\ref{lemma:rf-dr-permutation}-\ref{lemma:rf-mse-discrepancy} and Theorem~\ref{thm:clt} yields Theorem~\ref{thm:rf-clt}.
\end{proof}

\section{Proofs of lemmas in Section \ref{sec:key-lemma}}

\subsection{Proof of Lemma~\ref{lemma:fx_gphi}}
\begin{proof}[Proof of Lemma~\ref{lemma:fx_gphi}]
For any set $A\subseteq \supp(\zeta)$,
\begin{align*}
    \int_A f(x) \d \cH^m(x) = \int_A 1 \d \zeta(x) = \int_{\psi(A)} 1 \d \left(\psi_* \zeta\right) = \int_{\psi(A)} g_x(z) \d \lambda(z).
\end{align*}
Thus, for any continuous point $x \in \supp(\zeta)$,
\begin{align*}
f(x) & = \lim _{r \rightarrow 0} \frac{1}{\cH^m\Big(B\left(x,r\right) \cap \supp(\zeta)\Big)} \int_{B\left(x,r\right) \cap \supp(\zeta)} f(x) \d \cH^m(x) \\
~~{\rm and}~~g_x\left(\psi(x)\right) & = \lim _{r \rightarrow 0} \frac{1}{\lambda\Big(B\left(x,r\right) \cap \supp(\zeta)\Big)} \int_{\psi\left(B\left(x,r\right) \cap \supp(\zeta)\right)} g_x(z) \d \lambda(z).
\end{align*}
By Lemma~\ref{lemma:gbound}, $f$ and $g_x$ are bounded and bounded away from zero on their supports, and the property of Hausdorff and Lebesgue measures \citep[Theorem 2.5]{gariepy2015measure} yields that
\begin{align*}
\frac{f(x)}{g_x\left(\psi(x)\right)} = \lim _{r \rightarrow 0} \frac{\lambda\Big(\psi\left(B\left(x,r\right) \cap \supp(\zeta)\right)\Big)}{\cH^m\Big(B\left(x,r\right) \cap \supp(\zeta)\Big)}\frac{\int_{B\left(x,r\right) \cap \supp(\zeta)} f(x) \d \cH^m(x)}{\int_{\psi\left(B\left(x,r\right) \cap \supp(\zeta)\right)} g_x(z) \d \lambda(z)} 
= 1,
\end{align*}
where the last step can be achieved by changing the probability measure in  \citet[Lemma 2.2]{han2022azadkia} to Hausdorff measure.
\end{proof}

\subsection{Proof of Lemma~\ref{lemma:thetafunc}}
\begin{proof}[Proof of Lemma~\ref{lemma:thetafunc}]
    This is a direct result from the proof of Lemma 3.1 in \cite{han2022azadkia} and by noticing that continuous differentiable function is locally Lipschitz.
\end{proof}

\subsection{Proof of Lemma~\ref{lemma:manifold}}
\begin{proof}[Proof of Lemma~\ref{lemma:manifold}]
Let $d_{Q}$ be the diameter of $\supp(Q)$ and $B(x, r)$ be the ball with center $x$  and radius $r$.

When $h$ is sufficiently small such that $B \left(x, h^{1/2}d_{Q}\right) \subset U$,
\begin{align*}
& \int_U Q\left( h^{-1/2} \lVert x_i-x\rVert \right) \d \zeta(x_i) = \int_{B \left(x, h^{1/2}d_{Q}\right)} Q\left( h^{-1/2} \lVert x_i-x\rVert \right) \d \zeta(x_i) \\
= & \int_{\psi\left(B \left(x, h^{1/2}d_{Q}\right)\cap \cM\right)} Q\left( h^{-1/2} \lVert \psi^{-1}(z_i)-\psi^{-1}(z)\rVert \right) \d \left(\psi_* \zeta\right)\\
= & \int_{\psi\left(B \left(x, h^{1/2}d_{Q}\right)\cap \cM\right)} {Q}\left( h^{-1/2} \lVert x_i-x\rVert \right) g_x\left(z_i\right)\d\lambda(z_i).\yestag\label{eq:lemma-manifold-eq}
\end{align*}

We first claim that there exists $c_x^{\prime}>0$ such that when $h$ is sufficiently small,
\begin{equation} \label{eq:lemma-manifold-conatin}
    B \left(\psi(x), (1-c^{\prime}_xh^{1/2}d_{Q})h^{1/2}d_{Q}\right) \subset \psi\left(B \left(x, h^{1/2}d_{Q}\right)\cap \cM\right) \subset B \left(\psi(x), h^{1/2}d_{Q}\right) \subset V.
\end{equation}

The second ``$\subset$'' of \eqref{eq:lemma-manifold-conatin} can be obtained by the property of orthogonal projection:
$$
\lVert \psi(x)-\psi\left(x^\prime\right)\rVert  \leq\lVert x-x^\prime\rVert.
$$

As for the first ``$\subset$'' in \eqref{eq:lemma-manifold-conatin}, we use the result from Lemma~\ref{lemma:thetafunc}. In detail, since $g$ and $\psi^{-1}$ are locally Lipschitz continuous ($\psi^{-1}$ is smooth and thus locally Lipschitz continuous), for every $z^{\prime} \in B\left(\psi(x), h^{1/2} d_Q\right) \backslash \{\psi (x)\}$,
\begin{align*}
& \frac{\lVert \psi^{-1}(z^{\prime})-x\rVert}{\lVert z^{\prime}-\psi\left(x\right)\rVert}-1 =  \frac{1}{\cos\left(\theta_x\left(\psi^{-1}(z^{\prime})\right)\right)} -1 \lesssim \lVert\psi^{-1}(z^{\prime})-x\rVert \lesssim \lVert z^{\prime}-\psi\left(x\right)\rVert
\leq h^{1/2} d_Q.
\end{align*}
Thus for every $z^{\prime} \in B\left(\psi(x),\left(1-c_x^{\prime} h^{1/2} d_Q\right) h^{1/2} d_Q\right) \backslash \{\psi (x)\}$, there exists $c_x^{\prime}>0$ such that
$$
\lVert\psi^{-1}(z^{\prime})-x\rVert \leq\left(1+c^{\prime}_x h^{1/2} d_Q\right)\lVert z^{\prime}-\psi\left(x\right)\rVert \leq h^{1/2}d_{Q}.
$$
Therefore,
$$
\psi^{-1}\left(B\left(\psi(x),\left(1-c_x^{\prime} h^{1/2} d_Q\right) h^{1/2} d_Q\right)\right) \subset B \left(x, h^{1/2}d_{Q}\right)\cap \cM.
$$
Applying the mapping $\psi$ to both sides and repeating the process for $x$ and $z$, we obtain the first ``$\subset$'' of Equation~\eqref{eq:lemma-manifold-conatin}.

With \eqref{eq:lemma-manifold-eq}, to prove \eqref{eq:lemma-manifold-result} we only need to show that
\begin{equation} \label{eq:lemma-manifold-result-1}
\Big\lvert 1- \frac{\int_{B \left(x, h^{1/2}d_{Q}\right)} Q\left( h^{-1/2} \lVert x_i-x\rVert \right) \d \zeta(x_i)}{\int_{\psi\left(B \left(x, h^{1/2}d_{Q}\right)\cap \cM\right)} {Q}\left( h^{-1/2} \lVert z_i-z\rVert \right) g_x\left(z_i\right)\d\lambda(z_i)}\Big\rvert \lesssim h^{1/2}d_Q
\end{equation}
and
\begin{equation} \label{eq:lemma-manifold-result-2}
\Big\lvert 1- \frac{\int_{\psi\left(B \left(x, h^{1/2}d_{Q}\right)\cap \cM\right)} {Q}\left( h^{-1/2} \lVert z_i-z\rVert \right) g_x\left(z_i\right)\d\lambda(z_i)}{\int_{B \left(z, h^{1/2}d_{Q}\right)} Q\left( h^{-1/2} \lVert z_i-z\rVert \right) g_x\left(z_i\right) \d\lambda(z_i)}\Big\rvert \lesssim h^{1/2}d_Q.
\end{equation}

We first prove \eqref{eq:lemma-manifold-result-2}. When $h$ is sufficiently small, by \eqref{eq:lemma-manifold-conatin},
\begin{align*}
& h^{-m/2} \Big\lvert \int_{\psi\left(B \left(x, h^{1/2}d_{Q}\right)\cap \cM\right)} {Q}\left( h^{-1/2} \lVert z_i-z\rVert \right) g_x\left(z_i\right)\d\lambda(z_i) - \int_{B \left(z, h^{1/2}d_{Q}\right)} Q\left( h^{-1/2} \lVert z_i-z\rVert \right) g_x\left(z_i\right) \d\lambda(z_i) \Big\rvert \\
\lesssim & h^{-m/2} \lambda\Big({B \left(z, h^{1/2}d_{Q}\right) \backslash \psi\left(B \left(x, h^{1/2}d_{Q}\right)\cap \cM\right) }\Big)\\
\leq & h^{-m/2} \Big[\lambda\Big(B \left(z, h^{1/2}d_{Q}\right)\Big) - \lambda\Big(B \left(z, (1-c_x^{\prime} h^{1/2}d_Q)h^{1/2}d_{Q}\right)\Big)\Big] \\
= & h^{-m/2} \left(1-\left(1-c_x^{\prime} h^{1/2}d_Q\right)^m\right)\left(h^{1/2}d_{Q}\right)^m U_m 
\lesssim m c_x^{\prime} h^{1/2} d_Q,
\end{align*}
where $U_m$ is the volume of ball $B(0,1) \subset \mathbb{R}^m$. And applying Assumption~\ref{asp-lemma:manifold-3}, we obtain \eqref{eq:lemma-manifold-result-2}.

To prove \eqref{eq:lemma-manifold-result-1}, by Assumption~\ref{asp-lemma:manifold-1} we notice that when $h$ is sufficiently small
\begin{align*}
& h^{-m/2} \Big\lvert \int_{\psi\left(B \left(x, h^{1/2}d_{Q}\right)\cap \cM\right)} {Q}\left( h^{-1/2} \lVert x_i-x\rVert \right) g_x\left(z_i\right)\d\lambda(z_i) \\
& - \int_{\psi\left(B \left(x, h^{1/2}d_{Q}\right)\cap \cM\right)} {Q}\left( h^{-1/2} \lVert z_i-z\rVert \right) g_x\left(z_i\right)\d\lambda(z_i) \Big\rvert\\
\lesssim & h^{-m/2} \int_{\psi\left(B \left(x, h^{1/2}d_{Q}\right)\cap \cM\right)} h^{-1/2} \Big\lvert \lVert x_i-x\rVert - \lVert z_i-z\rVert \Big\rvert \d\lambda(z_i) \\
\lesssim &  h^{-m/2} \int_{\psi\left(B \left(x, h^{1/2}d_{Q}\right)\cap \cM\right)} \frac{\Big\lvert \lVert x_i-x\rVert - \lVert z_i-z\rVert \Big\rvert}{\lVert x_i-x\rVert} \d\lambda(z_i) \\
= & h^{-m/2} \int_{\psi\left(B \left(x, h^{1/2}d_{Q}\right)\cap \cM\right)} 1-\cos\left(\theta_x\left(x_i\right)\right) \d\lambda(z_i)\\
\lesssim & h^{-m/2} \lambda\Big(B\left(z, h^{1/2} d_Q\right)\Big) h^{1/2} d_Q 
\lesssim h^{1/2} d_Q,
\end{align*}
where the last step is by Lemma~\ref{lemma:thetafunc}. Combining the result with \eqref{eq:lemma-manifold-conatin} and \eqref{eq:lemma-manifold-eq}, we finished proving \eqref{eq:lemma-manifold-result-1}.
\end{proof}

\subsection{Proof of Lemma~\ref{lemma:manifoldfunc}}
\begin{proof} [Proof of Lemma~\ref{lemma:manifoldfunc}]
Since $\phi$ is continuous, for any $\varepsilon>0$, there exists $\delta_x>0$ such that when $0<h<\delta_x$,
\begin{align*}
& h^{-m/2} \Big\lvert \int_{B \left(x, h^{1/2}d_{Q}\right)} Q\left( h^{-1/2} \lVert x_i-x\rVert \right) \phi(x_i) \d \zeta(x_i) - \int_{B \left(x, h^{1/2}d_{Q}\right)} Q\left( h^{-1/2} \lVert x_i-x\rVert \right) \phi(x) \d \zeta(x_i) \Big\rvert \\
\leq & h^{-m/2}  \int_{B \left(x, h^{1/2}d_{Q}\right)} Q\left( h^{-1/2} \lVert x_i-x\rVert \right) \Big\lvert \phi(x_i) - \phi(x) \Big\rvert\d \zeta(x_i) \\
\leq & \varepsilon h^{-m/2}  \int_{B \left(x, h^{1/2}d_{Q}\right)} Q\left( h^{-1/2} \lVert x_i-x\rVert \right) \d \zeta(x_i) \lesssim \varepsilon, \yestag\label{eq:lemma-manifoldfunc-diff}
\end{align*}
where the last step is by Assumption~\ref{asp-lemma:manifold-3} and Lemma~\ref{lemma:manifold}.

And if $\phi$ is (locally) Lipschitz, we have
\begin{equation} \label{eq:lemma-manifoldfunc-lipdiff}
h^{-m/2}\left|\int_{B\left(x, h^{1/2} d_Q\right)} Q\left(h^{-1/2}\lVert x_i-x\rVert \right) \phi\left(x_i\right) \mathrm{d} \zeta\left(x_i\right)-\int_{B\left(x, h^{1/2} d_Q\right)} Q\left(h^{-1/2}\lVert x_i-x\rVert \right) \phi(x) \mathrm{d} \zeta\left(x_i\right)\right| \lesssim h^{1/2}d_Q.
\end{equation}

Furthermore, when $\phi(x)>0$, by Assumption~\ref{asp-lemma:manifold-3} and Lemma~\ref{lemma:manifold}, 
\begin{equation*}
\Big\lvert 1- \frac{\phi(x)\int_{B \left(x, h^{1/2}d_{Q}\right)} Q\left( h^{-1/2} \lVert x_i-x\rVert \right) \d \zeta(x_i)}{\phi(x)\int_{B \left(\psi(x), h^{1/2}d_{Q}\right)} Q\left( h^{-1/2} \lVert z_i-\psi(x)\rVert \right) g_x\left(z_i\right) \d\lambda(z_i)} \Big\rvert \lesssim h^{1/2}d_Q.
\end{equation*}
Combining the above with \eqref{eq:lemma-manifoldfunc-diff} or \eqref{eq:lemma-manifoldfunc-lipdiff} leads to \eqref{lemma:manifoldfunc-1} and \eqref{lemma:manifoldfunc-1-h}.

And when $\phi(x)=0$, 
$$
h^{-m/2} \phi(x) \int_{B \left(z, h^{1/2}d_{Q}\right)} Q\left( h^{-1/2} \lVert z_i-z\rVert \right) g_x\left(z_i\right) \d\lambda(z_i)=0
$$
implies \eqref{lemma:manifoldfunc-2} and \eqref{lemma:manifoldfunc-2-h}.
\end{proof}

\section{Proofs of lemmas in Section \ref{sec:pf-framework}}

\subsection{Proof of Lemma~\ref{lemma:inverse-kernel}}

\begin{proof}[Proof of Lemma~\ref{lemma:inverse-kernel}]
First, we notice that for any given positive integer $k$,
\begin{align*}
& \E\Big[\Big\{\frac{1}{n} \sum_{i=1}^n h^{-m/2}K\Big(h^{-1/2}\lVert X_i-x\rVert\Big) - \E\Big[h^{-m/2}K\Big(h^{-1/2}\lVert X_i-x\rVert\Big)\Big] \Big\}^{2k} \Big] \\
= & \Big(\frac{1}{nh^{m/2}}\Big)^{2k}\E\Big[\sum_{\substack{0 \leq l_1, l_2, \ldots, l_n \leq k \\
l_1+l_2+\ldots+l_n=k}}\left(\begin{array}{c}
2k \\
2l_1, 2l_2, \ldots, 2l_n
\end{array}\right) \prod_{i=1}^n \Big\{K\Big(h^{-1/2}\lVert X_i-x\rVert\Big) - \E\Big[K\Big(h^{-1/2}\lVert X_i-x\rVert\Big)\Big]\Big\}^{2l_i} \Big] \\
\lesssim & \Big(\frac{1}{nh^{m/2}}\Big)^{2k}\sum_{\substack{0 \leq l_1, l_2, \ldots, l_n \leq k \\
l_1+l_2+\ldots+l_k=k}}\left(\begin{array}{c}
2k \\
2l_1, 2l_2, \ldots, 2l_n
\end{array}\right) h^{-\frac{m}{2}(\sum_{i=1}^n \ind(l_i>0))}
\lesssim \Big(\frac{1}{nh^{m/2}}\Big)^k. \yestag\label{eq:mse-kernel-center-moment}
\end{align*}

Therefore, applying the Minkowski inequality with \eqref{eq:mse-kernelexp} and \eqref{eq:mse-kernel-diff}, we obtain that for any given positive integer $k$,
\begin{align*}
& \Big\{\E\Big[\Big(c_K g_x\left(z\right) - \frac{1}{n}h^{-m/2} \sum_{i=1}^n K\Big(h^{-1/2}\lVert X_i-x\rVert\Big)\Big)^{2k}\Big]\Big\}^{1/2k}\\
\leq & \Big\{\E\Big[\Big\{c_K g_x\left(z\right) - \E\Big[h^{-m/2} K\Big(h^{-1/2}\lVert X_1-x\rVert\Big)\Big]\Big\}^{2k}\Big]\Big\}^{1/2k} \\
& + \Big\{\E\Big[\Big\{\frac{1}{n}h^{-m/2} \sum_{i=1}^n K\Big(h^{-1/2}\lVert X_i-x\rVert\Big) - \E\Big[h^{-m/2} K\Big(h^{-1/2}\lVert X_i-x\rVert\Big)\Big] \Big\}^{2k}\Big]\Big\}^{1/2k}\\
\lesssim & h^{1/2} + \Big(\frac{1}{nh^{m/2}}\Big)^{1/2}. \yestag\label{eq:mse-kernel-diff}
\end{align*}

Notably, through Assumption~\ref{asp:mse-k},
\begin{align*}
    & \E\Big[\Big\lvert\frac{1}{\frac{1}{n}\sum_{i=1}^n h^{\frac{d-m}{2}} K_{h, i}}\Big\rvert^\alpha\Big] 
    = \left(nh^{m/2}\right)^\alpha \E\Big[\Big\lvert\frac{1}{\sum_{i=1}^n K\left(h^{-1/2}\lVert X_i-x\rVert\right)}\Big\rvert^\alpha\Big] \\
    = & \left(nh^{m/2}\right)^\alpha \E\Big[\Big\lvert\frac{1}{K\left(h^{-1/2}\lVert X_{i_0}-x\rVert\right) + \sum_{i\neq i_0} K\left(h^{-1/2}\lVert X_i-x\rVert\right)}\Big\rvert^\alpha\Big] \\
    \lesssim & \left(nh^{m/2}\right)^\alpha \E\Big[\Big\lvert\frac{1}{1 + \sum_{i\neq i_0} \ind\left(h^{-1/2}\lVert X_i-x\rVert \in \supp(K)\right)}\Big\rvert^\alpha\Big],
\end{align*}
where $[\ind\left(h^{-1/2}\lVert X_i-x\rVert\in \supp(K)\right)]_{i \neq i_0}$, with $\ind(\cdot)$ representing the indicator function, are independent with Bernoulli distribution parameter $p_x \asymp h^{m/2}$.
According to the inverse moments of a binomial random variable \citep[Page 275]{cribari2000note}, we have
\begin{equation*}
    \E\Big[\Big\lvert\frac{1}{\frac{1}{n}\sum_{i=1}^n h^{\frac{d-m}{2}} K_{h, i}}\Big\rvert^\alpha\Big] 
    \lesssim \left(nh^{m/2}\right)^\alpha O\left(\left[(n-1)p_x\right]^{-\alpha}\right) \asymp 1.
\end{equation*}

Thus for any $\alpha>2$, Hölder's inequality indicates that
\begin{align*}
& \E\Big[\Big\lvert\frac{1}{\frac{1}{n}h^{-m/2} \sum_{i=1}^n K\Big(h^{-1/2}\lVert X_i-x\rVert\Big)} - \frac{1}{c_K g_x\left(z\right)}\Big\rvert^{\alpha}\Big] \\
\lesssim & \Big\{\E\Big[\Big(\frac{1}{n}h^{-m/2} \sum_{i=1}^n K\Big(h^{-1/2}\lVert X_i-x\rVert\Big)\Big)^{-\alpha\alpha_1}\Big]\Big\}^{1/{\alpha_1}} \Big\{\E\Big[\Big\lvert c_K g_x\left(z\right) - \frac{1}{n}h^{-m/2} \sum_{i=1}^n K\Big(h^{-1/2}\lVert X_i-x\rVert\Big)\Big\rvert^{-\alpha\alpha_2}\Big]\Big\}^{1/{\alpha_2}}\\
\lesssim & \Big(h+\frac{1}{nh^{m/2}}\Big)^{\alpha/2}.
\end{align*}
This completes the proof.
\end{proof}

\subsection{Proof of Lemma~\ref{lemma-mse:tau.bar}}

\begin{proof}[Proof of Lemma~\ref{lemma-mse:tau.bar}]
We first notice that
\begin{align*}
& \E \Big[\Big(\bar{\tau}(\mX, x)  - \tau(x)\Big)^2\Big]
= \E \Big[\Big(\frac{1}{\sum_{i=1}^n K_{h, i}} \sum_{i=1}^n K_{h, i} \Big[\mu_1(X_i) - \mu_0(X_i)\Big] - \Big[\mu_1(x) - \mu_0(x)\Big] \Big)^2 \Big]\\
\lesssim &  \E \Big[\Big(\frac{1}{\sum_{i=1}^n K_{h, i}} \sum_{i=1}^n K_{h, i} \Big[\mu_1(X_i) - \mu_1(x)\Big] \Big)^2\Big] + \E \Big[\Big(\frac{1}{\sum_{i=1}^n K_{h, i}} \sum_{i=1}^n K_{h, i} \Big[\mu_0(X_i) - \mu_0(x)\Big] \Big)^2\Big].
\end{align*}
Regarding the two aforementioned terms, we only have to establish the first portion under the treatment condition, and the second portion under the control condition can be similarly established. Notice that for any positive value of $\alpha$,
\begin{align*}
    & \E\Big[\Big\lvert\frac{1}{n}\sum_{i=1}^n h^{\frac{d-m}{2}} K_{h, i} \Big[\mu_1(X_i) - \mu_1(x)\Big]\Big\rvert^{\alpha}\Big] \\
    \leq & \left(\frac{1}{nh^{m/2}}\right)^{\alpha} \E\Big[\Big(\sum_{i=1}^n K\left(h^{-1/2}\lVert X_i-x\rVert\right) \Big\lvert\mu_1(X_i) - \mu_1(x)\Big\rvert\Big)^{\alpha}\Big] \\
    \lesssim & \left(\frac{1}{nh^{m/2}}\right)^{\alpha}\E\Big[\Big(\sum_{i=1}^n K\left(h^{-1/2}\lVert X_i-x\rVert\right)\lVert X_i-x\rVert\Big)^{\alpha}\Big] \\
    \lesssim & \left(\frac{1}{nh^{m/2}}\right)^{\alpha} h^{\alpha/2}\E\Big[\Big(\sum_{i=1}^n \ind\left(h^{-1/2}\lVert X_i-x\rVert\in\supp(K)\right)\Big)^{\alpha}\Big] \\
    \lesssim & \left(\frac{1}{nh^{m/2}}\right)^{\alpha} h^{\alpha/2} \left(nh^{m/2}\right)^\alpha\\
    \lesssim& h^{\alpha/2}.
\end{align*}

Therefore, applying Lemma~\ref{lemma:inverse-kernel}
\begin{align*}
& \E \Big[\Big(\frac{1}{\sum_{i=1}^n K_{h, i}} \sum_{i=1}^n K_{h, i} \Big[\mu_1(X_i) - \mu_1(x)\Big] \Big)^2\Big]
\lesssim \E \Big[\Big(\frac{1}{c_K g_x\left(z\right)} \frac{1}{n}\sum_{i=1}^n h^{\frac{d-m}{2}} K_{h, i} \Big[\mu_1(X_i) - \mu_1(x)\Big] \Big)^2\Big]\\
& + \E \Big[\Big(\frac{1}{\frac{1}{n}\sum_{i=1}^n h^{\frac{d-m}{2}} K_{h, i}} - \frac{1}{c_K g_x\left(z\right)} \Big)^2 \Big(\frac{1}{n}\sum_{i=1}^n h^{\frac{d-m}{2}} K_{h, i} \Big[\mu_1(X_i) - \mu_1(x)\Big] \Big)^2\Big]\\
\lesssim & h + \Big\{\E\Big[\Big(\frac{1}{\frac{1}{n}\sum_{i=1}^n h^{\frac{d-m}{2}} K_{h, i}} - \frac{1}{c_K g_x\left(z\right)} \Big)^4\Big]\Big\}^{1/2} \Big\{\E\Big[\Big(\frac{1}{n}\sum_{i=1}^n h^{\frac{d-m}{2}} K_{h, i} \Big[\mu_1(X_i) - \mu_1(x)\Big] \Big)^4\Big]\Big\}^{1/2}
\lesssim h.
\end{align*}

Thus,
\begin{equation} \label{eq:mse1-result}
    \E \Big[\Big(\bar{\tau}(\mX, x)  - \tau(x)\Big)^2\Big] \lesssim h.
\end{equation}
This completes the proof.
\end{proof}

\subsection{Proof of Lemma~\ref{lemma-mse:fn}}

\begin{proof}[Proof of Lemma~\ref{lemma-mse:fn}]
\quad

{\bf Part I.} Assume the accuracy of the propensity score model, specifically, the validity of Assumption~\ref{asp:dr-propensity}\ref{asp:dr-propensity-2}.
Notice that
\begin{align*}
F_n(x) 
= & \frac{1}{c_K g_x\left(z\right)} \Big[ \frac{1}{n} h^{-\frac{d-m}{2}}\sum_{i=1}^n\sum_{j:D_j=1-D_i} K_{h,j}w_{j\leftarrow i} -K_{h,i}\Big(D_i \frac{1-e(X_i)}{e(X_i)} + (1-D_i) \frac{e(X_i)}{1-e(X_i)} \Big)\Big]\epsilon_i\\
& + \Big(\frac{1}{\frac{1}{n}h^{\frac{d-m}{2}} \sum_{i=1}^n K_{h, i}} - \frac{1}{c_K g_x\left(z\right)}\Big)
\\
& \Big[ \frac{1}{n} h^{-\frac{d-m}{2}}\sum_{i=1}^n \sum_{j:D_j=1-D_i} K_{h,j}w_{j\leftarrow i} -K_{h,i}\Big(D_i \frac{1-e(X_i)}{e(X_i)} + (1-D_i) \frac{e(X_i)}{1-e(X_i)} \Big)\Big] \epsilon_i\\
& + \frac{1}{c_K g_x\left(z\right)} \frac{1}{n}h^{\frac{d-m}{2}}\sum_{i=1}^n K_{h,i}\Big(D_i \frac{1}{e(X_i)} + (1-D_i) \frac{1}{1-e(X_i)} \Big) \epsilon_i \\
& + \Big( \frac{1}{\frac{1}{n}h^{\frac{d-m}{2}} \sum_{i=1}^n K_{h, i}} - \frac{1}{c_K g_x\left(z\right)}\Big) \frac{1}{n}h^{\frac{d-m}{2}}\sum_{i=1}^n K_{h,i}\Big(D_i \frac{1}{e(X_i)} + (1-D_i) \frac{1}{1-e(X_i)} \Big) \epsilon_i \\
= :&  \tilde{F}_1 + \tilde{F}_2 +\tilde{F}_3 + \tilde{F}_4.
\end{align*}

Therefore,
\begin{equation} \label{eq:mse2-4part}
    \E [F_n^2(x)] \lesssim \sum_{i=1}^4 \E [\tilde{F}_i^2(x)].
\end{equation}

{\bf Part 1.}
For the first term in \eqref{eq:mse2-4part}, according to Assumption~\ref{asp:dr-outcome}\ref{asp:dr-outcome-2}, $\E [\tilde{F}_1 \given \mX, \mD]=0$.
Therefore, by the law of total variance and \eqref{asp:dr-propensity-2-single} in Assumption~\ref{asp:dr-propensity}\ref{asp:dr-propensity-2},
\begin{align*}
& \E[\tilde{F}_1^2] = \Var[\tilde{F}_1] = \E[\Var[\tilde{F}_1\given \mX, \mD]] \\
= & \E\Big[\frac{1}{c_K^2 g_x^2\left(z\right)}   \sum_{i=1}^n\Big[\frac{1}{n}h^{-\frac{d-m}{2}}\sum_{j:D_j=1-D_i} K_{h,j}w_{j\leftarrow i} -K_{h,i}\Big(D_i \frac{1-e(X_i)}{e(X_i)} + (1-D_i) \frac{e(X_i)}{1-e(X_i)} \Big)\Big]^2 \Var\Big[\epsilon_i\Biggiven \mX, \mD\Big]\Big] \\
\lesssim & \E\Big[ \sum_{i=1}^n\Big[\frac{1}{n}h^{-\frac{d-m}{2}}\sum_{j:D_j=1-D_i} K_{h,j}w_{j\leftarrow i} -K_{h,i}\Big(D_i \frac{1-e(X_i)}{e(X_i)} + (1-D_i) \frac{e(X_i)}{1-e(X_i)} \Big)\Big]^2 \Big] \\
= & \frac{1}{n} \E\Big[\Big(h^{-\frac{d-m}{2}}\sum_{j:D_j=1-D_1} K_{h,j}w_{j\leftarrow 1} -K_{h,i}\Big(D_1 \frac{1-e(X_1)}{e(X_1)} + (1-D_i) \frac{e(X_1)}{1-e(X_1)} \Big)\Big)^2 \Big] \\
= & \frac{1}{n} O\Big(h^{-m/2+\epsilon}+\frac{1}{nh^{m/2}}\Big)
= o\Big(\frac{1}{nh^{m/2}}\Big). \yestag\label{eq:mse2-1:result1}
\end{align*}

{\bf Part 2.}
Let $\alpha_1>1, \alpha_2>1$ be such that $\alpha_1^{-1} + \alpha_2^{-1}=1$. Then utilizing Lemma~\ref{lemma:inverse-kernel}, we have
\begin{align*}
\E[\tilde{F}_2^2]
\leq & \Big\{\E\Big[\Big\lvert\frac{1}{\frac{1}{n}h^{\frac{d-m}{2}} \sum_{i=1}^n K_{h, i}} - \frac{1}{c_K g_x\left(z\right)}\Big\rvert^{2\alpha_1}\Big]\Big\}^{1/{\alpha_1}} \\
& \Big\{\E\Big[\Big\lvert\frac{1}{n} h^{-\frac{d-m}{2}}\sum_{i=1}^n \sum_{j:D_j=1-D_i} K_{h,j}w_{j\leftarrow i} -K_{h,i}\Big(D_i \frac{1-e(X_i)}{e(X_i)} + (1-D_i) \frac{e(X_i)}{1-e(X_i)} \Big\rvert\Big) \epsilon_i\Big\}^{2\alpha_2}\Big]\Big\}^{1/{\alpha_2}} \\
\lesssim & \Big(h+\frac{1}{nh^{m/2}}\Big)
\Big\{\E\Big[\Big\{\frac{1}{n} h^{-\frac{d-m}{2}}\sum_{i=1}^n \sum_{j:D_j=1-D_i} K_{h,j}w_{j\leftarrow i} -K_{h,i}\Big(D_i \frac{1-e(X_i)}{e(X_i)} + (1-D_i) \frac{e(X_i)}{1-e(X_i)} \Big)\Big) \epsilon_i\Big\}^{2} \\
& h^{-m(2\alpha_2-2)/2}\Big]\Big\}^{1/{\alpha_2}}\\
\lesssim & \Big(h+\frac{1}{nh^{m/2}}\Big) \Big\{\E\Big[\tilde{F}_1^2 h^{-m(\alpha_2-1)}\Big]\Big\}^{1/{\alpha_2}}
\lesssim \Big(h+\frac{1}{nh^{m/2}}\Big) h^{-m(1-\frac{1}{\alpha_2})} \Big(h^{\epsilon} + \frac{1}{n}\Big)^{\frac{1}{\alpha_2}}.
\end{align*}
Letting $\alpha_2$ be closer enough to $1$, we then obtain,
\begin{equation}\label{eq:mse2-1:result2}
    \E[\tilde{F}_2^2] = o\Big(h + \frac{1}{nh^{m/2}}\Big).
\end{equation}

{\bf Part 3.}
For the third term,
\begin{align*}
\tilde{F}_3 & = \frac{1}{c_K g_x\left(z\right)} \frac{1}{n}h^{\frac{d-m}{2}}\sum_{i=1}^n K_{h,i}\Big(D_i \frac{1}{e(X_i)} + (1-D_i) \frac{1}{1-e(X_i)} \Big) \epsilon_i\\
\lesssim & \frac{1}{n}h^{\frac{d-m}{2}}\sum_{i=1}^n K_{h,i}\Big(D_i \frac{1}{e(X_i)} + (1-D_i) \frac{1}{1-e(X_i)} \Big) \epsilon_i
=: \frac{1}{n}\sum_{i=1}^n F_{n,i}(x),
\end{align*}
where $\{F_{n,i}(x)\}_{i=1}^n$ are i.i.d. with mean zero conditional on $X_i,D_i$. According to Assumption~\ref{asp:causal},
\begin{align*}
& \E \Big[\Big(\frac{1}{n}\sum_{i=1}^nF_{n,i}(x) \Big)^2 \Big]
= \frac{1}{n} \Var \Big[ F_{n,1}(x) \Big]
= \frac{1}{n} \E\Big[\Var\Big[ F_{n,1}(x) \Biggiven \mX,\mD\Big]\Big]\\
\lesssim & \frac{1}{nh^{m}} \E\Big[K^2\Big(h^{-1/2}\lVert X_1-x\lVert\Big)\Big(D_1 \frac{1}{e(X_1)} + (1-D_1) \frac{1}{1-e(X_1)}\Big)^2\Var[ \epsilon_1\given X_1]\Big]\\
\lesssim & \frac{1}{nh^{m}} \E\Big[K^2\Big(h^{-1/2}\lVert X_1-x\lVert\Big)\Big]
\lesssim \frac{1}{nh^{m/2}}.
\end{align*}
To conclude, 
\begin{equation}\label{eq:mse2-1:result3}
    \E[\tilde{F}_3^2] = O\Big(\frac{1}{nh^{-m/2}}\Big).
\end{equation}

{\bf Part 4.}
Notice that
\begin{align*}
\tilde{F}_4 = \Big(\frac{1}{\frac{1}{n}h^{\frac{d-m}{2}} \sum_{i=1}^n K_{h, i}} - \frac{1}{c_K g_x\left(z\right)}\Big) \frac{1}{n} \sum_{i=1}^n F_{n,i}(x),
\end{align*}
where
\begin{align*}
& \E \Big[\Big(\frac{1}{n}\sum_{i=1}^nF_{n,i}(x) \Big)^4 \Big]
= n^{-4} \Big\{ \E \Big[ \sum_{i=1}^nF_{n,i}^4(x) + \sum_{i_1,i_2}F_{n,i_1}^2(x)F_{n,i_2}^2(x) \Big] \Big\}
\lesssim n^{-3}\E \Big[F_{n,1}^4(x)\Big] + n^{-2}\Big\{\E\Big[F_{n,1}^2(x)\Big]\Big\}^2\\
\lesssim & n^{-3}\E \Big[h^{-2m}K^4\Big(h^{-1/2}\lVert X_1-x\lVert\Big)\Big] + n^{-2}\Big\{\E\Big[h^{-m}K^2\Big(h^{-1/2}\lVert X_1-x\lVert\Big)\Big]\Big\}^2
\lesssim \Big(\frac{1}{nh^{m/2}}\Big)^2.
\end{align*}

Thus, by Lemma~\ref{lemma:inverse-kernel},
\begin{align*}
\E[\tilde{F}_4^2]
\leq \Big\{\E \Big[\Big(\frac{1}{n}\sum_{i=1}^nF_{n,i}(x) \Big)^4 \Big]\Big\}^{1/2} \Big\{\E \Big[\Big( \frac{1}{\frac{1}{n}h^{\frac{d-m}{2}} \sum_{i=1}^n K_{h, i}} - \frac{1}{c_K g_x\left(z\right)}\Big)^4 \Big]\Big\}^{1/2}
= o\Big(\frac{1}{nh^{m/2}}\Big).\yestag\label{eq:mse2-1:result4}
\end{align*}

Combining \eqref{eq:mse2-1:result1}, \eqref{eq:mse2-1:result2}, \eqref{eq:mse2-1:result3} and \eqref{eq:mse2-1:result4}, we arrive at the conclusion expressed as
\begin{equation} \label{eq:mse2-1:result}
\E\left[\left(F_n(x)\right)^2\right] \lesssim \frac{1}{n h^{m/2}}.
\end{equation}

{\bf Part II.} Assume the misspecification for the propensity score, specifically, the validity of Assumptions~\ref{asp:dr-outcome}\ref{asp:dr-outcome-3}.  We then have
\begin{align*}
    & \lvert F_n(x)\rvert = \Big\lvert \frac{1}{\sum_{i=1}^n K_{h, i}} \sum_{i=1}^n\left(2 D_i-1\right)\Big(K_{h, i}+\sum_{j: D_j=1-D_i} K_{h, j} w_{j \leftarrow i}\Big)\epsilon_i \Big\rvert \\
    \leq & \Big\lvert\frac{1}{\sum_{i=1}^n K_{h, i}} \sum_{i=1}^n \left(2 D_i-1\right) K_{h, i}\epsilon_i\Big\rvert + \Big\lvert\frac{1}{\sum_{i=1}^n K_{h, i}} \sum_{i=1}^n \left(2 D_i-1\right) \Big(\sum_{j: D_j=1-D_i} K_{h, j} w_{j \leftarrow i}\Big)\epsilon_i\Big\rvert. \yestag\label{eq:mse2-2:2part}
\end{align*}

The first term in 
\eqref{eq:mse2-2:2part} can be exactly handled as that of \eqref{eq:mse2-1:result3} and \eqref{eq:mse2-1:result4}.

The second term in \eqref{eq:mse2-2:2part} can be broken down as follows:
\begin{align*}
& \Big\lvert\frac{1}{\sum_{i=1}^n K_{h, i}} \sum_{i=1}^n \left(2 D_i-1\right) \Big(\sum_{j: D_j=1-D_i} K_{h, j} w_{j \leftarrow i}\Big)\epsilon_i\Big\rvert \\
\leq & \frac{1}{c_K g_x\left(z\right)}\Big\lvert  \frac{1}{n}h^{\frac{d-m}{2}}\sum_{i=1}^n\left(2 D_i-1\right) \Big(\sum_{j: D_j=1-D_i} K_{h, j} w_{j \leftarrow i}\Big)\epsilon_i \Big\rvert \\
+ & \Big\lvert \frac{1}{\frac{1}{n}h^{\frac{d-m}{2}} \sum_{i=1}^n K_{h, i}} - \frac{1}{c_K g_x\left(z\right)}\Big\rvert \Big\lvert \frac{1}{n}h^{\frac{d-m}{2}}\sum_{i=1}^n\left(2 D_i-1\right)\Big(\sum_{j: D_j=1-D_i} K_{h, j} w_{j \leftarrow i}\Big)\epsilon_i \Big\rvert.
\yestag\label{eq:mse2-2:dom}
\end{align*}

As the first term in \eqref{eq:mse2-2:dom} satisfies the mean-zero condition given $\mX$ and $\mD$, Assumption~\ref{asp:dr-outcome} yields that
\begin{align*}
& \E \Big[\Big(\frac{1}{n}h^{\frac{d-m}{2}}\sum_{i=1}^n\Big(2 D_i-1\Big) \Big(\sum_{j: D_j=1-D_i} K_{h, j} w_{j \leftarrow i}\Big)\epsilon_i\Big)^2 \Big] \\
= & \Var \Big[\frac{1}{nh^{m/2}}\sum_{i=1}^n\Big(2 D_i-1\Big) \Big(\sum_{j: D_j=1-D_i} K\Big(h^{-1/2}\lVert X_j-x\rVert\Big) w_{j \leftarrow i}\Big)\epsilon_i \Big]\\
= & \frac{1}{n^2}\sum_{i=1}^n\E\Big[\Big(\sum_{j: D_j=1-D_i} h^{m/2} K\Big(h^{-1/2}\lVert X_j-x\rVert\Big) w_{j \leftarrow i}\Big)^2\Var[\epsilon_i\given \mX, \mD]\Big]\\
\lesssim & \frac{1}{n}\E\Big[\Big(\sum_{j: D_j=1-D_1} h^{-m/2} K\Big(h^{-1/2}\lVert X_j-x\rVert\Big) w_{j \leftarrow 1}\Big)^2\Big]
\lesssim \frac{1}{nh^{m/2}}. \yestag\label{eq:mse2-2:result3}
\end{align*}

Regarding the second term in \eqref{eq:mse2-2:dom} and considering Assumption~\ref{asp:wij}, we observe that
\begin{align*}
& \Big\lVert\frac{1}{nh^{m/2}}\sum_{i=1}^n\Big(2 D_i-1\Big) \Big(\sum_{j: D_j=1-D_i} K_{h, j} w_{j \leftarrow i}\Big)\epsilon_i \Big\lVert_\infty \\
\lesssim & \Big\lVert\frac{1}{nh^{m/2}}\sum_{i=1}^n \sum_{j: D_j=1-D_i} \lvert w_{j \leftarrow i} \rvert \Big\lVert_\infty 
= \frac{1}{nh^{m/2}}\Big\lVert\sum_{j=1}^n \sum_{i: D_i=1-D_j} \lvert w_{j \leftarrow i}\rvert \Big\lVert_\infty  = O(h^{-m/2}).
\end{align*}
Leveraging Lemma~\ref{lemma:inverse-kernel} and Hölder's inequality, we obtain that when $1<\alpha_2 \leq 1+1/m$,
\begin{align*}
& \E \Big[\Big(\frac{1}{\frac{1}{n}h^{\frac{d-m}{2}} \sum_{i=1}^n K_{h, i}} - \frac{1}{c_K g_x\left(z\right)}\Big)^2 \Big( \frac{1}{n}h^{\frac{d-m}{2}}\sum_{i=1}^n\left(2 D_i-1\right) \sum_{j: D_j=1-D_i} K_{h, j} w_{j \leftarrow i}\epsilon_i\Big)^2\Big]\\
\leq & \Big\{\E \Big[\Big\lvert\frac{1}{\frac{1}{n}h^{\frac{d-m}{2}} \sum_{i=1}^n K_{h, i}} - \frac{1}{c_K g_x\left(z\right)}\Big\rvert^{2\alpha_1}\Big]\Big\}^{1/\alpha_1} \Big\{\E \Big[\Big\lvert\frac{1}{n}h^{\frac{d-m}{2}}\sum_{i=1}^n\left(2 D_i-1\right) \sum_{j: D_j=1-D_i} K_{h, j} w_{j \leftarrow i}\epsilon_i\Big\rvert^{2\alpha_2}\Big]\Big\}^{1/\alpha_2}\\
\lesssim & \Big(h +\frac{1}{nh^{m/2}}\Big) \Big\{\E \Big[\Big( \frac{1}{n}h^{\frac{d-m}{2}}\sum_{i=1}^n\left(2 D_i-1\right) \sum_{j: D_j=1-D_i} K_{h, j} w_{j \leftarrow i}\epsilon_i\Big)^{2} h^{-m(\alpha_2-1)}\Big]\Big\}^{1/\alpha_2}\\
\lesssim & \Big(h +\frac{1}{nh^{m/2}}\Big) h^{-m(1-\frac{1}{\alpha_2})} \Big\{ \frac{1}{n} \E\Big[\Big(h^{-m/2}\sum_{j: D_j=1-D_1} K\Big(h^{-1/2}\lVert X_j-x \rVert\Big) w_{j \leftarrow 1}\Big)^{2} \Big]\Big\}^{1/\alpha_2} \\
\lesssim & \Big(h +\frac{1}{nh^{m/2}}\Big) h^{-m(1-\frac{1}{\alpha_2})} \Big(\frac{1}{nh^{m/2}}\Big)^{1/\alpha_2}
= o\Big(h +\frac{1}{nh^{m/2}}\Big).\yestag\label{eq:mse2-2:result4}
\end{align*}

Together with \eqref{eq:mse2-2:result3} and \eqref{eq:mse2-2:result4}, we finished our proof.
\end{proof}

\subsection{Proof of Lemma~\ref{lemma-mse:b}}

\begin{proof}[Proof of Lemma~\ref{lemma-mse:b}]
\quad

{\bf Part I.} Suppose Assumptions~\ref{asp:mse-muhat}-\ref{asp:mse-discrepancy} is true. We then have
\begin{align*}
& \lvert B_n(x) - \hat{B}_n(x) \rvert \\
\leq & \frac{1}{c_K g_x\left(z\right)} \frac{1}{n}h^{\frac{d-m}{2}} \sum_{i=1}^n K_{h, i} \sum_{j:D_j=1-D_i} \lvert w_{i\leftarrow j} \rvert \max_{\omega \in \{0,1\}} \Big\lvert \mu_\omega(X_i) - \mu_\omega(X_j) - \hat{\mu}_\omega(X_i) + \hat{\mu}_\omega(X_j) \Big\rvert\\
& + \Big\lvert \frac{1}{ \frac{1}{n}h^{\frac{d-m}{2}} \sum_{i=1}^n K_{h, i}} - \frac{1}{c_K g_x\left(z\right)}\Big\rvert \frac{1}{n}h^{\frac{d-m}{2}} \sum_{i=1}^n K_{h, i} \sum_{j:D_j=1-D_i} \lvert w_{i\leftarrow j} \rvert \\
& \max_{\omega \in \{0,1\}} \Big\lvert \mu_\omega(X_i) - \mu_\omega(X_j) - \hat{\mu}_\omega(X_i) + \hat{\mu}_\omega(X_j) \Big\rvert. \yestag\label{eq:mse3-dom}
\end{align*}
Concerning the first term in Equation \eqref{eq:mse3-dom}, utilizing Hölder's inequality in conjunction with Assumptions~\ref{asp:mse-mu}, \ref{asp:mse-muhat}, and \ref{asp:mse-discrepancy},
\begin{align*}
& \E\Big[\Big(\frac{1}{c_K g_x\left(z\right)} \frac{1}{n}h^{\frac{d-m}{2}} \sum_{i=1}^n K_{h, i} \sum_{j:D_j=1-D_i} \lvert w_{i\leftarrow j} \rvert \max_{\omega \in \{0,1\}} \Big\lvert \mu_\omega(X_i) - \mu_\omega(X_j) - \hat{\mu}_\omega(X_i) + \hat{\mu}_\omega(X_j) \Big\rvert\Big)^2\Big]\\
\lesssim & \E \Big[\Big(\frac{1}{n}h^{\frac{d-m}{2}} \sum_{i=1}^n K_{h, i} \sum_{j:D_j=1-D_i} \lvert w_{i\leftarrow j} \rvert \lVert X_i-X_j \rVert\Big)^2\Big] \\
\leq & \Big\{\E \Big[\Big(\frac{1}{n}h^{\frac{d-m}{2}} \sum_{i=1}^n K_{h, i} \sum_{j:D_j=1-D_i} \lvert w_{i\leftarrow j} \rvert \lVert X_i-X_j \rVert\Big)^{2\gamma}\Big]\Big\}^{1/\gamma}
\lesssim h.
\end{align*}
Regarding the second term in Equation \eqref{eq:mse3-dom}, we can apply Hölder's inequality, along with Lemma~\ref{lemma:inverse-kernel} and Assumption~\ref{asp:mse-discrepancy}, yielding
\begin{align*}
& \E\Big[\Big(\Big\lvert \frac{1}{ \frac{1}{n}h^{\frac{d-m}{2}} \sum_{i=1}^n K_{h, i}} - \frac{1}{c_K g_x\left(z\right)}\Big\rvert \frac{1}{n}h^{\frac{d-m}{2}} \sum_{i=1}^n K_{h, i} \sum_{j:D_j=1-D_i} \lvert w_{i\leftarrow j} \rvert \\
& \max_{\omega \in \{0,1\}} \Big\lvert \mu_\omega(X_i) - \mu_\omega(X_j) - \hat{\mu}_\omega(X_i) + \hat{\mu}_\omega(X_j) \Big\rvert\Big)^2\Big]\\
\lesssim & \Big\{\E \Big[\Big(\frac{1}{n}h^{\frac{d-m}{2}} \sum_{i=1}^n K_{h, i} \sum_{j:D_j=1-D_i} \lvert w_{i\leftarrow j} \rvert \lVert X_i-X_j \rVert\Big)^{2\gamma}\Big]\Big\}^{1/\gamma}
\lesssim h.
\end{align*}
Therefore,
\begin{align*}
\E \Big[\Big(B_n(x) - \hat{B}_n(x)\Big)^2\Big]
\lesssim h.  \yestag \label{eq:mse3-result}
\end{align*}

{\bf Part II.} Suppose that Assumption~\ref{asp:mse-weight} is true. According to Assumption~\ref{asp:wij},
\begin{align*}
& \lvert B_n(x) - \hat{B}_n(x) \rvert \\
= & \Big\lvert \frac{1}{\sum_{i=1}^n K_{h, i}} \sum_{i=1}^n\left(2 D_i-1\right)K_{h, i} \Big[\sum_{j:D_j=1-D_i} w_{i\leftarrow j} \Big(\mu_{1-D_i}(X_i)-\hat{\mu}_{1-D_i}(X_i)\Big)\Big]\\
& - \frac{1}{\sum_{i=1}^n K_{h, i}} \sum_{i=1}^n\left(2 D_i-1\right)K_{h, i} \Big[\sum_{j:D_j=1-D_i} w_{i\leftarrow j} \Big(\mu_{1-D_i}(X_j)-\hat{\mu}_{1-D_i}(X_j)\Big)\Big]\Big\rvert\\
\lesssim & \frac{1}{\sum_{i=1}^n K_{h, i}} \sum_{i=1}^n K_{h, i}\sum_{j:D_j=1-D_i} \lvert w_{i\leftarrow j} \rvert \max_{\omega \in \{0,1\}} \lVert \mu_\omega - \hat{\mu}_\omega \rVert_\infty\\
\lesssim & \frac{1}{\sum_{i=1}^n K_{h, i}} \sum_{i=1}^n K_{h, i} \max_{\omega \in \{0,1\}} \lVert \mu_\omega - \hat{\mu}_\omega \rVert_\infty
= \max_{\omega \in \{0,1\}} \lVert \mu_\omega - \hat{\mu}_\omega \rVert_\infty.
\end{align*}
Therefore,
\begin{align*}
\E \Big[\Big(B_n(x) - \hat{B}_n(x)\Big)^2\Big]
\lesssim h,  \yestag \label{eq:mse3-result-2}
\end{align*}
and the proof is thus complete.
\end{proof}

\subsection{Proof of Lemma~\ref{lemma-clt:tau.bar}}
\begin{proof}[Proof of Lemma~\ref{lemma-clt:tau.bar}]
The proof is grounded in Lemma~\ref{lemma-mse:tau.bar} and \eqref{eq:thm-dr:kde}.
\end{proof}

\subsection{Proof of Lemma~\ref{lemma-clt:en}}
\begin{proof}[Proof of Lemma~\ref{lemma-clt:en}]
For any $i \in \zahl{n}$, let $\epsilon_i = Y_i - \mu_{D_i}(X_i)$  and define
\begin{align*}
E_{n,i}(x) := \left(2 D_i-1\right)h^{\frac{d-m}{2}}K_{h, i} \left(D_i\left(1+\phi_1\left(X_i\right)\right)+\left(1-D_i\right)\left(1+\phi_0\left(X_i\right)\right)\right)\epsilon_i.
\end{align*}
It's noteworthy that $\E [E_{n,i}(x) \given \mX, \mD] = 0$ and $\E [E_{n,i}(x)] = 0$. Therefore,
\begin{align*}
\Var[E_{n,i}(x) \given \mX, \mD]
= h^{d-m}K^2_{h, i}\left(D_i\left(1+\phi_1\left(X_i\right)\right)+\left(1-D_i\right)\left(1+\phi_0\left(X_i\right)\right)\right)^2 \sigma_{D_i}^2(X_i).
\end{align*}
By applying the law of total variance, we obtain that
\begin{equation*}
\Var\Big[E_{n,i}(x)\Big]
= \E \Big[h^{d-m}K^2_{h, i}\left(D_i\left(1+\phi_1\left(X_i\right)\right)+\left(1-D_i\right)\left(1+\phi_0\left(X_i\right)\right)\right)^2 \sigma_{D_i}^2(X_i)\Big].
\end{equation*}
Since given $x$, $[E_{n,i}(x)]_{i=1}^n$ are independent, define
\begin{align*}
s^2_n :=  \Var\Big(\sum_{i=1}^n E_{n,i}(x)\Big) = n \E \Big[h^{d-m}K^2_{h, 1}\left(D_1\left(1+\phi_1\left(X_1\right)\right)+\left(1-D_1\right)\left(1+\phi_0\left(X_1\right)\right)\right)^2 \sigma_{D_1}^2(X_1)\Big].
\end{align*}
To apply the Lyapunov's central limit theorem \cite[Theorem 27.3]{billingsley1995probability}, it suffices to verify that under Assumption~\ref{asp:clt-se},
\begin{equation} \label{thm:clt-en-Lyapunov}
    \frac{1}{s^{2+k}_n} \sum_{i=1}^n \E\Big[\lvert E_{n,i}(x)\rvert^{2+k}\Big] \to 0.
\end{equation}
In detail, by Assumptions~\ref{asp:causal}\ref{asp:causal-2} and \ref{asp:clt-se},
\begin{align*}
& \frac{1}{s^{2+k}_n} \sum_{i=1}^n \E\Big[\lvert E_{n,i}(x)\rvert^{2+k}\Big] \\
= & \frac{n \E\Big[h^{\frac{(2+k)(d-m)}{2}}K^{2+k}_{h, 1} \Big\lvert D_1\left(1+\phi_1\left(X_1\right)\right)+\left(1-D_1\right)\left(1+\phi_0\left(X_1\right)\right)\Big\rvert^{2+k} \E[\lvert U_{D_1}(X_1)\rvert^{2+k}\given X=X_1]\Big]}{\left(n \E \Big[h^{d-m}K^2_{h, 1}\left(D_1\left(1+\phi_1\left(X_1\right)\right)+\left(1-D_1\right)\left(1+\phi_0\left(X_1\right)\right)\right)^2 \sigma_{D_1}^2(X_1)\Big]\right)^{\frac{2+k}{2}}}\\
\lesssim & n^{-k/2}\frac{\E\Big[h^{\frac{(2+k)(d-m)}{2}}K^{2+k}_{h, 1}\Big]}{\left(\E \Big[h^{d-m}K^2_{h, 1}\Big]\right)^{\frac{2+k}{2}}}.
\end{align*}
Noticing that by Lemma~\ref{lemma:manifold},
\begin{align*}
& \E\Big[h^{\frac{(2+k)(d-m)}{2}}K^{2+k}_{h, 1}\Big]
\asymp h^{-(2+k)m/2}\int K^{2+k}\left( h^{-1/2} \lVert z_i-\psi(x)\rVert \right) g_x\left(z_i\right) \d\lambda(z_i)\\
= & h^{-(1+k)m/2} g_x\left(\psi(x)\right) \int K^{2+k}(\lVert t\rVert)\d t + o\left(h^{-(1+k)m/2}\right) \asymp h^{-(1+k)m/2} .
\end{align*}
Thus, 
\begin{equation*}
\frac{1}{s^{2+k}_n} \sum_{i=1}^n \E\Big[\lvert E_{n,i}(x)\rvert^{2+k}\Big] 
\lesssim n^{-k/2}\frac{h^{-(1+k)m/2}}{\left(h^{-m/2}\right)^{\frac{2+k}{2}}} \lesssim n^{-k/2}h^{-km/4} \rightarrow 0.
\end{equation*}
By central limit theory on triangular arrays, we have
\begin{equation} \label{thm:clt-en-lln-main}
    \frac{\sum_{i=1}^n E_{n,i}(x)}{\left(n \E \Big[h^{d-m}K^2_{h, 1}\left(D_1\left(1+\phi_1\left(X_1\right)\right)+\left(1-D_1\right)\left(1+\phi_0\left(X_1\right)\right)\right)^2 \sigma_{D_1}^2(X_1)\Big]\right)^{1/2}} \stackrel{\sf d}{\longrightarrow} N(0,1).
\end{equation}

Under Assumption~\ref{asp:clt-phi}, by \eqref{eq:thm-dr:kde} and \eqref{thm:clt-en-lln-main}, 
\begin{equation}\label{eq:clt-result-21}
\sqrt{nh^{m/2}} E_n(x)
= \sqrt{nh^{m/2}} \frac{\frac{1}{n}\sum_{i=1}^n E_{n,i}(x)}{\frac{1}{n}\sum_{i=1}^n h^{\frac{d-m}{2}} K_{h, i}}
\stackrel{\sf d}{\longrightarrow} N\Big(0,\frac{\tilde{\Sigma}(x)}{c_K^2 f^2(x)}\Big).
\end{equation}
Furthermore, if $\phi_1(x)$ and $\phi_0(x)$ are Lipchitz functions that are bounded and bounded away from zero, 
then by Lemmas~\ref{lemma:fx_gphi} and \ref{lemma:manifold},
\begin{align*}
& h^{-m/2} \E \Big[K^2\left(\lVert h^{-1/2}\left(X_1-x\right)\rVert\right)\Big[D_1\left(1+\phi_1\left(X_1\right)\right)+\left(1-D_1\right)\left(1+\phi_0\left(X_1\right)\right)\Big]^2 \sigma_{D_1}^2(X_1)\Big] \\
= & h^{-m/2}\E \Big[K^2\left(\lVert h^{-1/2}\left(X_1-x\right)\rVert\right) \left(e(X_1)\left(1+\phi_1\left(X_1\right)\right)^2 \sigma_1^2(X_1)+ \left(1-e(X_1)\right)\left(1+\phi_0\left(X_1\right)\right)^2 \sigma_0^2(X_1)\right)\Big] \\
= & h^{-m/2}\int K^2\left(\lVert h^{-1/2}\left(x_1-x\right)\rVert\right) \Big(e(x)\left(1+\phi_1\left(x\right)\right)^2 \sigma_1^2(x)+ \left(1-e(x)\right)\left(1+\phi_0\left(x\right)\right)^2 \sigma_0^2(x) \\
& + O\left(\lVert x_1-x\rVert\right)\Big)\d\zeta(x_1) \\
= & \Big[ e(x)\left(1+\phi_1\left(x\right)\right)^2 \sigma_1^2(x)+ \left(1-e(x)\right)\left(1+\phi_0\left(x\right)\right)^2 \sigma_0^2(x) \Big] \frac{\int K^2\left(\lVert h^{-1/2}\left(x_1-x\right)\rVert\right) \d\zeta(x_1)}{\int K^2\left( h^{-1/2} \lVert z_1-\psi(x)\rVert \right) g_x\left(z_1\right)\d\lambda(z_1)} \\
& \int K^2\left(\lVert t\rVert \right) g_x\left(z\right)\d\lambda(t) +O\left(h^{(1-m)/2}\right)\\
= & \Sigma\left(x\right)c_K^2 f^2(x)+ O\left(h^{(1-m)/2}\right),
\end{align*}
where $\Sigma(x)=\frac{1}{c_K^2 f(x)}\Big( e(x)\left(1+\phi_1(x)\right)^2\sigma^2_1(x) + \left(1-e(x)\right)\left(1+\phi_0(x)\right)^2\sigma^2_0(x)\Big)\int_{\mathbb{R}^m} K^2(\lVert t\rVert) \d t$.
Thus, utilizing Assumption~\ref{asp:clt-phi} and \eqref{eq:clt-result-21}, we conclude that
\begin{equation}\label{eq:clt-result-22}
\sqrt{nh^{m/2}} E_n(x)
\stackrel{\sf d}{\longrightarrow} N(0,\Sigma(x)).   
\end{equation}

Under Assumption~\ref{asp:dr-propensity}\ref{asp:dr-propensity-2}, $\phi_1(x) = \frac{1-e(x)}{e(x)}$ and $\phi_0(x) = \frac{e(x)}{1-e(x)}$. Therefore,
\begin{equation} \label{eq:clt-result-1}
\sqrt{n h^{m / 2}} E_n(x) \stackrel{\mathrm{d}}{\longrightarrow} N\Big(0, \frac{1}{c_K^2 g_x(\psi(x))}\left(\frac{\sigma_1^2(x)}{e(x)}+\frac{\sigma_0^2(x)}{1-e(x)}\right) \int_{\mathbb{R}^m} K^2(\|t\|) \d t\Big).
\end{equation}

Combining \eqref{eq:clt-result-1}, \eqref{eq:clt-result-21}, and \eqref{eq:clt-result-22} completes the proof.
\end{proof}

\subsection{Proof of Lemma~\ref{lemma-clt:en-t}}
\begin{proof}[Proof of Lemma~\ref{lemma-clt:en-t}]
First, we can express
\begin{align*}
& \lvert \tilde{E}_n(x) \rvert \\
\leq & \frac{1}{c_K g_x\left(\psi(x)\right)}\Big\lvert  \frac{1}{n}\sum_{i=1}^n\Big(2 D_i-1\Big)h^{\frac{d-m}{2}} \Big(\sum_{j:D_j=1-D_i} K_{h, j}w_{j\leftarrow i} - K_{h, i}\Big(D_i \phi_1(X_i) + (1-D_i) \phi_0(X_i)\Big)\Big)\epsilon_i \Big\rvert \\
+ & \Big\lvert \frac{1}{ \frac{1}{n} \sum_{i=1}^n h^{\frac{d-m}{2}}K_{h, i}} - \frac{1}{c_K g_x\left(\psi(x)\right)}\Big\rvert \\
& \Big\lvert \frac{1}{n}\sum_{i=1}^n\Big(2 D_i-1\Big)h^{\frac{d-m}{2}}\Big(\sum_{j:D_j=1-D_i} K_{h, j}w_{j\leftarrow i} - K_{h, i}\Big(D_i \phi_1(X_i) + (1-D_i) \phi_0(X_i)\Big)\Big)\epsilon_i \Big\rvert.
\end{align*}
With \eqref{eq:thm-dr:kde}, by leveraging the same analysis as \eqref{eq:mse2-1:result1}, we conclude that
$$
\E[\tilde{E}_n(x)] = o\Big(\frac{1}{n h^{m/2}}\Big),
$$
which leads to
\begin{equation*}
    \sqrt{nh^{m/2}} \tilde{E}_n(x)\stackrel{\sf p}{\longrightarrow} 0
\end{equation*}
and thus completes the proof.
\end{proof}

\subsection{Proof of Lemma~\ref{lemma-clt:b}}
\begin{proof}[Proof of Lemma~\ref{lemma-clt:b}]
The proof is grounded in Lemma~\ref{lemma-mse:b} and \eqref{eq:thm-dr:kde}.
\end{proof}

\section{Proofs of lemmas in Section~\ref{sec:pf-theory}}

\subsection{Proof of Lemma~\ref{lemma:rf-clt}}

\begin{proof}[Proof of Lemma~\ref{lemma:rf-clt}]
Without loss of generosity, we assume $\omega=1$.

{\bf Part I.} Assume Assumption~\ref{asp:rf-honest} holds. According to Part III in \cite{lin2022regression},
\begin{align*}
& \E\Big[\Big(\frac{s}{\lvert L^{1}(X_1)\rvert} \zeta_1(L^{1}(X_1) \cap\cM)-1\Big)^2\Biggiven \mD,D_1=1,X_1,1\in\cI^1\Big] \\
= & \E\Big[\Big(\frac{\zeta_1(L^{1}(X_1) \cap\cM)-\frac{1}{s} \lvert L^{1}(X_1)\rvert}{\frac{1}{s} \lvert L^{1}(X_1) \rvert}\Big)^2\Biggiven \mD,D_1=1,X_1,1\in\cI^1\Big] \\
= & \E\Big[\Big(\frac{s}{\lvert L^{1}(X_1) \rvert} \zeta_1(L^{1}(X_1) \cap\cM)-1\Big)^2 \ind\Big(\Big\lvert\zeta_1(L^{1}(X_1) \cap\cM)-\frac{1}{s} \lvert L^{1}(X_1) \rvert \Big\rvert> h^{\epsilon} \frac{1}{s} \lvert L^{1}(X_1) \rvert\Big) \\
& \Biggiven \mD,D_1=1,X_1,1\in\cI^1\Big] + h^{2 \epsilon}\\
\leq & h^{2 \epsilon} + s^2\P\Big(\Big\lvert\zeta_1(L^{1}(X_1) \cap\cM)-\frac{1}{s} \lvert L^{1}(X_1) \rvert\Big\rvert>h^{\epsilon} \frac{1}{s} \lvert L^{1}(X_1) \rvert \Biggiven \mD,D_1=1,X_1,1\in\cI^1\Big) \\
\leq & h^{2 \epsilon} + s^2\Big\{\P\Big(\Big\lvert\zeta_1(L^{1}(X_1) \cap\cM)-\frac{1}{s} \lvert L^{1}(X_1) \rvert\Big\rvert>2^{d+1}s^{-1/2}(\log s)^{1/2} \Biggiven \mD,D_1=1,X_1,1\in\cI^1\Big)\\
& + \P\Big(\Big\rvert L^{1}(X_1) \Big\rvert <  h^{-\epsilon}2^{-d-1} (s\log s)^{1/2}\Biggiven \mD,D_1=1,X_1,1\in\cI^1\Big)\Big\} \\
\lesssim & h^{2 \epsilon} + s^{-6} +s^2 \P\Big(h^{-\epsilon}(s\log s)^{1/2} \Big\rvert L^{1}(X_1) \Big\rvert^{-1}>  2^{-d-1} \Biggiven \mD,D_1=1,X_1,1\in\cI^1\Big)\\
\leq &  h^{2 \epsilon} + s^{-6} +s^2 2^{(d+1)\beta}\E\Big[\Big\{h^{-\epsilon}(s\log s)^{1/2} \Big\rvert L^{1}(X_1) \Big\rvert^{-1}\Big\}^{\beta}\Biggiven \mD,D_1=1,X_1,1\in\cI^1\Big] \\
\lesssim & h^{2\epsilon}.
\end{align*}

Furthermore, according to \eqref{eq:rf-consistency4:leaf},
\begin{align*}
& \Big\lVert\E\Big[\Big(\frac{s}{\lvert L^{1}(X_1) \rvert} \zeta_0(L^{1}(X_1) \cap\cM)\Big)^2\Biggiven \mD,D_1=1,X_1,1\in\cI^1\Big]\Big\rVert_\infty \\
\lesssim & \Big\lVert\E\Big[\Big(\frac{s}{\lvert L^{1}(X_1) \rvert} \frac{g_{0,x}(Z_1)}{g_{1,x}(Z_1)} \zeta_1(L^{1}(X_1) \cap\cM)\Big)^2\Biggiven \mD,D_1=1,X_1,1\in\cI^1\Big]\Big\rVert_\infty \\
& +  \Big\lVert\E\Big[\Big(\frac{s}{\lvert L^{1}(X_1) \rvert} h^{1/2+\epsilon} \zeta_1(L^1_{t}\cap\cM)\Big)^2\Biggiven \mD,D_1=1,X_1,1\in\cI^1\Big]\Big\rVert_\infty\\
= & O(1). \yestag\label{eq:lemma-rf-zeta}
\end{align*}

{\bf Part II.} Assuming Assumption~\ref{asp:rf-extr-honest} holds, then $\{L_t\}_{t \ge 1}$ are independent of $\cI$. Utilizing Hölder's inequality, the inverse moments of a binomial random variable \citep[Page 275]{cribari2000note} and the fourth central moment of binomial distribution, we obtain
\begin{align*}
& \E\Big[\Big(\frac{s}{\lvert L^1(X_1)\rvert} \zeta_1(L^1(X_1)\cap\cM)-1\Big)^2\Biggiven \mD, D_1=1, X_1, 1 \in \cI^{1}, \{L^{1}_t\}_{t \ge 1}\Big] \\
= & \E\Big[\Big(\frac{\zeta_1(L^{1}_{t_0}\cap\cM) - \frac{1}{s}\sum_{k \in \cI^{1}} \ind(X_k \in L^{1}_{t_0})}{\frac{1}{s} \sum_{k \in \cI^{1}} \ind(X_k \in L^{1}_{t_0})}\Big)^2\Biggiven \mD, D_1=1,X_1,X_1 \in L^{1}_{t_0},L^{1}_{t_0}\Big] \\
\leq & \Big\{\E\Big[\Big(\frac{1}{s} \sum_{k \in \cI^{1}} \ind(X_k \in L^{1}_{t_0})\Big)^{-4}\Biggiven \mD, D_1=1,X_1,X_1 \in L^{1}_{t_0},L^{1}_{t_0}\Big]\Big\}^{1/2}\\ 
& \Big\{\E\Big[\Big(\zeta_1(L^{1}_{t_0}\cap\cM) - \frac{1}{s}\sum_{k \in \cI^{1}} \ind(X_k \in L^{1}_{t_0})\Big)^4\Biggiven \mD, D_1=1,X_1,X_1 \in L^{1}_{t_0},L^{1}_{t_0}\Big]\Big\}^{1/2} \\
\lesssim & \Big\{s^{4}\Big(\frac{1}{s \zeta_1(L^{1}_{t_0}\cap\cM)}\Big)^4\Big\}^{1/2} \Big\{s^{-4}\Big(s \zeta_1(L^{1}_{t_0}\cap\cM)\Big)^2\Big\}^{1/2}\\
=& \frac{1}{s \zeta_1(L^{1}_{t_0}\cap\cM)}.
\end{align*}
Thus,
\begin{align*}
& \E\Big[\Big(\frac{s}{\lvert L^1(X_1)\rvert} \zeta_1(L^1(X_1)\cap\cM)-1\Big)^2\Biggiven \mD, D_1=1, X_1, 1 \in \cI^{1}\Big] \\
= & \E\Big[\frac{1}{s \zeta_1(L^1(X_1)\cap\cM)}\Biggiven \mD, D_1=1, X_1, 1 \in \cI^{1}\Big]
\lesssim h^{2\epsilon}.
\end{align*}

Furthermore, similar to \eqref{eq:lemma-rf-zeta}, we also have
\begin{align*}
\Big\lVert\E\Big[\Big(\frac{s}{\lvert L(X_1)\rvert} \zeta_0(L(X_1))\Big)^2\Biggiven \mD, D_1=1, X_1, 1 \in \cI^{1}\Big]\Big\rVert_\infty = O(1).
\end{align*}

The proof is thus complete.
\end{proof}

\subsection{Proof of Lemma~\ref{lemma:rf-dr-propensity-2-sum}}

\begin{proof}[Proof of Lemma~\ref{lemma:rf-dr-propensity-2-sum}]
To verify \eqref{asp:dr-propensity-2-sum} in Assumption~\ref{asp:dr-propensity}\ref{asp:dr-propensity-2}, we first notice that
\begin{align*}
& h^{d-m}\E \Big[\Big(\frac{1}{n} \sum_{i=1}^n \sum_{j:D_j=1-D_1} K_{h,j}w_{j\leftarrow i} - K_{h,i}\Big(D_i \frac{1-e(X_i)}{e(X_i)} + (1-D_1) \frac{e(X_i)}{1-e(X_i)} \Big) \Big)^2\Big] \\
\lesssim &  h^{d-m}\E \Big[\Big( \frac{1}{n} \sum_{D_i=1} \sum_{j:D_j=1-D_1} K_{h,j}w_{j\leftarrow i} - K_{h,i} \frac{1-e(X_i)}{e(X_i)} \Big)^2\Big] \\
& + h^{d-m}\E \Big[\Big( \frac{1}{n} \sum_{D_i=0} \sum_{j:D_j=1-D_1} K_{h,j}w_{j\leftarrow i} - K_{h,i}\frac{e(X_i)}{1-e(X_i)} \Big) \Big)^2\Big].
\end{align*}

The following examines the first term above, and the same approach can be extended to derive the second term. 

Let $\{L^1_{bt}\}_{t \ge 1}$ be the set of terminal leaves in $L^1_b$ for $b \in \zahl{B}$, and \eqref{eq:rf-w} can be rewritten as
\begin{align*}
    w_{i \leftarrow j}= B^{-1} \sum_{b=1}^B \sum_{t\ge1} (\lvert \{k \in \cI^{D_j}_b:X_k \in L^{D_j}_{bt}\} \rvert)^{-1}\ind(j \in \cI^{D_j}_b:X_j \in L^{D_j}_{bt}) \sum_{D_j=0}  \ind(X_i \in L^{D_j}_{bt}).
\end{align*}
To streamline our proof, we simplify the argument by assuming that, under Assumption~\ref{asp:rf-extr-honest}, $\{L^{\omega}_{bt}\}_{t\ge 1, \omega\in\{0,1\}}$ are known. In other words, all calculations of expectation or variance are conditioned on $\{L^{\omega}_{bt}\}_{t\ge 1, \omega\in\{0,1\}}$.

Therefore, for the first term,
\begin{align*}
& \frac{1}{n} h^{-\frac{d-m}{2}}\sum_{D_i=1}\sum_{j:D_j=1-D_i} K_{h,j}w_{j\leftarrow i} -K_{h,i} \frac{1-e(X_i)}{e(X_i)} \\
= & \frac{1}{n}\sum_{D_i=1} B^{-1} \sum_{b=1}^B \sum_{t\ge1} (\lvert \{k \in \cI^1_b:X_k \in L^1_{bt}\} \rvert)^{-1}\ind(1 \in \cI^1_b:X_i \in L^1_{bt}) \sum_{D_j=0}  \ind\Big(X_j \in L^1_{bt}\cap B(x, h^{1/2}d_K)\Big)\\
& h^{-m/2} \Big[ K\left(h^{-1/2}\lVert X_j-x\rVert \right) -  K\left(h^{-1/2}\lVert X_i-x\rVert \right) \Big]\ind\Big(X_i \in B(x, h^{1/2}d_K)\Big) \\
& + \frac{1}{n}\sum_{D_i=1} h^{-m/2} B^{-1} \sum_{b=1}^B \sum_{t\ge1} (\lvert \{k \in \cI^1_b:X_k \in L^1_{bt}\} \rvert)^{-1}\ind(1 \in \cI^1_b:X_i \in L^1_{bt}) K\left(h^{-1/2}\lVert X_i-x\rVert \right) \\
& \sum_{D_j=0} \Big[ \ind\Big(X_j \in L^1_{bt}\cap B(x, h^{1/2}d_K)\Big) - \ind\Big(X_j \in L^1_{bt}\Big)\Big]\\
& + \frac{1}{n}\sum_{D_i=1} h^{-m/2} B^{-1} \sum_{b=1}^B \sum_{t\ge1} (\lvert \{k \in \cI^1_b:X_k \in L^1_{bt}\} \rvert)^{-1}\ind(1 \in \cI^1_b:X_i \in L^1_{bt}) K\left(h^{-1/2}\lVert X_i-x\rVert \right) \\
& \sum_{D_j=0} \Big[\ind\Big(X_j \in L^1_{bt}\Big) - \zeta_0\Big(L^1_{bt}\cap\cM\Big)\Big]\\
& + \frac{1}{n}\sum_{D_i=1} n_0 h^{-m/2} B^{-1} \sum_{b=1}^B \sum_{t\ge1} (\lvert \{k \in \cI^1_b:X_k \in L^1_{bt}\} \rvert)^{-1}\ind(1 \in \cI^1_b:X_i \in L^1_{bt})  K\left(h^{-1/2}\lVert X_i-x\rVert \right) \\
& \Big[\zeta_0\Big(L^1_{bt}\cap\cM\Big) - \frac{g_{0,x}(Z_i)}{g_{1,x}(Z_i)} \zeta_1\Big(L^1_{bt}\cap\cM\Big)\Big]\\
& + \frac{1}{n}\sum_{D_i=1} \frac{n_0}{n_1} \frac{g_{0,x}(Z_i)}{g_{1,x}(Z_i)} h^{-m/2} K\left(h^{-1/2}\lVert X_i-x\rVert \right) \Big[n_1 B^{-1} \sum_{b=1}^B \sum_{t\ge1} (\lvert \{k \in \cI^1_b:X_k \in L^1_{bt}\} \rvert)^{-1}\ind(1 \in \cI^1_b:X_i \in L^1_{bt}) \\
& \zeta_1(L^1_{bt}\cap\cM) - 1\Big] \\
& + \frac{1}{n}\sum_{D_i=1} h^{-m/2} K\left(h^{-1/2}\lVert X_i-x\rVert \right) \left(\frac{n_0}{n_1} \frac{g_{0,x}\left(Z_i\right)}{g_{1,x}\left(Z_i\right)} - \frac{1-e\left(X_i\right)}{e\left(X_i\right)}\right) \\
& + \frac{1}{n}\sum_{D_i=1} B^{-1} \sum_{b=1}^B \sum_{t\ge1} (\lvert \{k \in \cI^1_b:X_k \in L^1_{bt}\} \rvert)^{-1}\ind(1 \in \cI^1_b:X_i \in L^1_{bt}) \sum_{D_j=0}  \ind\Big(X_j \in L^1_{bt}\cap B(x, h^{1/2}d_K)\Big)\\
& h^{-m/2} K\left(h^{-1/2}\lVert X_j-x\rVert \right) \ind\Big(X_i \in B^c(x, h^{1/2}d_K)\Big) \\
=:& S_1 + S_2 + S_3 + S_4 + S_5 + S_6 +  S_7.
\end{align*}

According to Jensen's Inequality,
\begin{equation}\label{eq:rf-consistency:7part}
    h^{d-m}\E \Big[\Big( \frac{1}{n} \sum_{D_i=1} \sum_{j:D_j=1-D_1} K_{h,j}w_{j\leftarrow i} - K_{h,i} \frac{1-e(X_i)}{e(X_i)} \Big)^2\Big]\lesssim \sum_{i=1}^7 \E \Big[ {S_i}^2 \Big].
\end{equation}


{\bf Part 1.}
For $S_1$, by Assumption~\ref{asp:rf1}\ref{asp:rf1-2} we notice that
\begin{align*}
\lvert S_1\rvert \leq & \frac{1}{n}\sum_{D_i=1} B^{-1} \sum_{b=1}^B \sum_{t\ge1} (\lvert \{k \in \cI^1_b:X_k \in L^1_{bt}\} \rvert)^{-1}\ind(1 \in \cI^1_b:X_i \in L^1_{bt}) \sum_{D_j=0}  \ind\Big(X_j \in L^1_{bt}\cap B(x, h^{1/2}d_K)\Big)\\
& h^{-m/2} \Big\lvert K\left(h^{-1/2}\lVert X_j-x\rVert \right) -  K\left(h^{-1/2}\lVert X_i-x\rVert \right) \Big\rvert\ind\Big(X_i \in B(x, h^{1/2}d_K)\Big) \\
\lesssim & \frac{1}{n}\sum_{D_i=1} B^{-1} \sum_{b=1}^B \sum_{t\ge1} (\lvert \{k \in \cI^1_b:X_k \in L^1_{bt}\} \rvert)^{-1}\ind(1 \in \cI^1_b:X_i \in L^1_{bt}) \sum_{D_j=0}  \ind\Big(X_j \in L^1_{bt}\cap B(x, h^{1/2}d_K)\Big)\\
& h^{-m/2} h^{-1/2} \lVert X_j-X_i \rVert \ind\Big(X_i \in B(x, h^{1/2}d_K)\Big) \\
\lesssim & \frac{1}{n}\sum_{D_i=1} B^{-1} \sum_{b=1}^B \sum_{t\ge1} (\lvert \{k \in \cI^1_b:X_k \in L^1_{bt}\} \rvert)^{-1}\ind(1 \in \cI^1_b:X_i \in L^1_{bt}) \sum_{D_j=0}  \ind\Big(X_j \in L^1_{bt}\cap B(x, h^{1/2}d_K)\Big)\\
& h^{-(m+1)/2} {\rm diam}(L^1(X_i)\cap \cM)) \ind\Big(X_i \in B(x, h^{1/2}d_K)\Big) \\
\lesssim & \frac{1}{n h^{m/2}} h^{\epsilon} \sum_{D_i=1} B^{-1} \sum_{b=1}^B \sum_{t\ge1} (\lvert \{k \in \cI^1_b:X_k \in L^1_{bt}\} \rvert)^{-1}\ind(1 \in \cI^1_b:X_i \in L^1_{bt})\sum_{D_j=0}  \ind\Big(X_j \in L^1_{bt}\cap B(x, h^{1/2}d_K)\Big)\\
= & \frac{1}{n h^{m/2}} h^{\epsilon} \sum_{D_i=1} \sum_{D_j=0} w_{j\leftarrow i} \ind\Big(X_j \in B(x, h^{1/2}d_K)\Big)\\
=& \frac{1}{n h^{m/2}} h^{\epsilon} \sum_{D_j=0} \ind\Big(X_j \in B(x, h^{1/2}d_K)\Big). \yestag\label{eq:rf-consistency1:s1}
\end{align*}

Thus, 
\begin{align*}
\E[ S^2_1  ]
= h^{2\epsilon}\E \Big[ \Big\{\frac{1}{n h^{m/2}}  \sum_{D_j=0} \ind\Big(X_j \in B(x, h^{1/2}d_K)\Big) \Big\}^2\Big] 
\lesssim h^{2\epsilon}.\yestag\label{eq:rf-consistency1:result}
\end{align*}

{\bf Part 2.}
According to Assumption~\ref{asp:rf1}\ref{asp:rf1-2}, there exists $c_{\epsilon}>0$, such that ${\rm diam}(L_t) \leq c_{\epsilon} h^{1/2+\epsilon}$. Now define 
\begin{align*}
B^0(x, h^{1/2}d_K):=& \{\tilde{x}: h^{1/2}d_K-c_{\epsilon}h^{1/2+\epsilon}\leq \lVert \tilde{x}-x \rVert \leq h^{1/2}d_K\},\\
B^1(x, h^{1/2}d_K):=& \{\tilde{x}: h^{1/2}d_K \leq \lVert \tilde{x}-x \rVert \leq h^{1/2}d_K+c_{\epsilon}h^{1/2+\epsilon}\}.
\end{align*}
Therefore,
\begin{align*}
\lvert S_2\rvert
= &  \frac{1}{n}\sum_{D_i=1} h^{-m/2} B^{-1} \sum_{b=1}^B \sum_{t\ge1} (\lvert \{k \in \cI^1_b:X_k \in L^1_{bt}\} \rvert)^{-1}\ind(1 \in \cI^1_b:X_i \in L^1_{bt}) K\left(h^{-1/2}\lVert X_i-x\rVert \right) \\
& \sum_{D_j=0} \ind\Big(X_j \in L^1_{bt}\cap B^c(x, h^{1/2}d_K)\Big)\\
\lesssim &  \frac{1}{n}\sum_{D_i=1} h^{-m/2} \ind\Big(X_i \in B^0(x, h^{1/2}d_K)\Big) B^{-1} \sum_{b=1}^B \sum_{t\ge1} (\lvert \{k \in \cI^1_b:X_k \in L^1_{bt}\} \rvert)^{-1}\ind(1 \in \cI^1_b:X_i \in L^1_{bt})\\
& \sum_{D_j=0} \ind\Big(X_j \in L^1_{bt}\cap B^c(x, h^{1/2}d_K)\Big)
=: \tilde{S}_2.
\end{align*}

And
\begin{equation}\label{eq:rf-consistency2:2part}
\E[S_2^2\given \mD] \lesssim \E[\{\E[ \tilde{S}_2 \given \mD,\mX_1]\}^2\given\mD] + \E[\Var[ \tilde{S}_2 \given \mD,\mX_1]\given\mD].
\end{equation}

For the first term in \eqref{eq:rf-consistency2:2part},
\begin{align*}
& \E[\{\E[ \tilde{S}_2 \given \mD,\mX_1]\}^2\given\mD] \\ 
=  & \E\Big[\Big\{\E\Big[\frac{1}{n}\sum_{D_i=1} h^{-m/2} \ind\Big(X_i \in B^0(x, h^{1/2}d_K)\Big) B^{-1} \sum_{b=1}^B \sum_{t\ge1} (\lvert \{k \in \cI^1_b:X_k \in L^1_{bt}\} \rvert)^{-1}\ind(1 \in \cI^1_b:X_i \in L^1_{bt}) \\
& \E\Big[\sum_{D_j=0} \ind\Big(X_j \in L^1_{bt}\cap B^c(x, h^{1/2}d_K)\Big)\Biggiven \mD,\mX_1,\{\cI^1_b\}_{b=1}^B,\{L^1_{bt}\}_{t \ge 1}\Big]\Biggiven \mD,\mX_1\Big]\Big\}^2\Biggiven \mD\Big] \\
\leq & \E\Big[\Big\{\sum_{D_i=1} h^{-m/2} \ind\Big(X_i \in B^0(x, h^{1/2}d_K)\Big) \E\Big[\sum_{t\ge1} (\lvert \{k \in \cI^1:X_k \in L^1_{t}\} \rvert)^{-1} \ind(i \in \cI^1:X_i \in L^1_{t}) \\
& \zeta_0\Big(L^1_{t}\cap B^c(x, h^{1/2}d_K)\cap\cM\Big)\Biggiven \mD,\mX_1\Big]\Big\}^2\Biggiven \mD\Big] \\
\leq & \E\Big[\Big\{\sum_{D_i=1} h^{-m/2} \ind\Big(X_i \in B^0(x, h^{1/2}d_K)\Big) \E\Big[\frac{\ind(i \in I^1)}{\lvert L^1(X_i)\rvert} \zeta_0\Big(L^1(X_i)\cap\cM\Big)\Biggiven \mD,\mX_1\Big]\Big\}^2\Biggiven \mD\Big] \\
= & \E\Big[\Big\{\frac{1}{n_1} \sum_{D_i=1} h^{-m/2} \ind\Big(X_i \in B^0(x, h^{1/2}d_K)\Big) \E\Big[\frac{s}{\lvert L^1(X_i)\rvert}\zeta_0\Big(L^1(X_i)\cap\cM\Big)\Biggiven \mD,D_i=1,\mX_1,i\in \cI^1\Big]\Big\}^2\Biggiven \mD\Big]\\
= & \frac{1}{n_1h^{m/2}} \E\Big[ h^{-m/2} \ind\Big( X_1 \in B^0(x, h^{1/2}d_K)\Big)\Big\{\E\Big[\frac{s}{\lvert L^1(X_1)\rvert} \zeta_0\Big(L^1(X_1)\cap\cM\Big)\Biggiven \mD,D_1=1,\mX_1,1\in \cI^1\Big]\Big\}^2\Biggiven \mD,D_1=1\Big]\\
& + \frac{n_1-1}{n_1} \E\Big[ h^{-m} \ind\Big( X_1 \in B^0(x, h^{1/2}d_K)\Big) \ind\Big( X_2 \in B^0(x, h^{1/2}d_K)\Big) \\
& \E\Big[\frac{s}{\lvert L^1(X_1)\rvert} \zeta_0\Big(L^1(X_1)\cap\cM\Big)\Biggiven \mD,D_1=1,\mX_1,1\in \cI^1\Big] \\
& \E\Big[\frac{s}{\lvert L^1(X_2)\rvert} \zeta_0\Big(L^1(X_2)\cap\cM\Big)\Biggiven \mD,D_2=1,\mX_1,2\in \cI^1\Big]\Biggiven \mD, D_1=D_2=1\Big] \\
\lesssim & \frac{1}{n_1h^{m/2}} h^{\epsilon} \E\Big[\Big\{\frac{s}{\lvert L^1(X_1)\rvert} \zeta_0\Big(L^1(X_1)\cap\cM\Big)\Big\}^2\Biggiven \mD,D_1=1,X_1 \in B^0(x, h^{1/2}d_K), 1\in \cI^1 \Big]\\
& + h^{\epsilon}\E\Big[\E\Big[\frac{s}{\lvert L^1(X_1)\rvert} \zeta_0\Big(L^1(X_1)\cap\cM\Big)\Biggiven \mD,D_1=1,\mX_1,1\in \cI^1\Big] \E\Big[\frac{s}{\lvert L^1(X_2)\rvert} \zeta_0\Big(L^1(X_2)\cap\cM\Big)\\
& \Biggiven \mD,D_2=1,\mX_1,2\in \cI^1\Big] \Biggiven \mD, D_1=D_2=1,X_1\in B^0(x, h^{1/2}d_K),X_2\in B^0(x, h^{1/2}d_K)\Big] \\
\lesssim & \frac{1}{n_1h^{m/2}} h^{\epsilon} \E\Big[\Big\{\frac{s}{\lvert L^1(X_1)\rvert} \zeta_0\Big(L^1(X_1)\cap\cM\Big)\Big\}^2\Biggiven \mD,D_1=1,X_1 \in B^0(x, h^{1/2}d_K), 1\in \cI^1 \Big]\\
& + h^{\epsilon}\Big\{\E\Big[\Big\{\frac{s}{\lvert L^1(X_1)\rvert} \zeta_0\Big(L^1(X_1)\cap\cM\Big)\Big\}^{2}\Biggiven \mD, D_1=D_2=1,X_1\in B^0(x, h^{1/2}d_K),X_2\in B^0(x, h^{1/2}d_K)\Big]\Big\}^{1/2} \\
& \Big\{\E\Big[\Big\{\frac{s}{\lvert L^1(X_2)\rvert} \zeta_0\Big(L^1(X_2)\cap\cM\Big)\Big\}^{2}\Biggiven \mD, D_1=D_2=1,X_1\in B^0(x, h^{1/2}d_K),X_2\in B^0(x, h^{1/2}d_K)\Big]\Big\}^{1/2} \\
\lesssim & \Big(1+\frac{1}{n_1h^{m/2}}\Big)h^{\epsilon}, \yestag\label{eq:rf-consistency2:e}
\end{align*}
where we leverage the insight that the maximum difference between empirical distribution functions, built on $s-2$ sample points versus the same $s-2$ points plus an additional 1-2 sample points, is guaranteed to be less than $2/(s-2)$, and thus, similar to Lemma~\ref{lemma:rf-clt}, we have
\begin{equation}\label{eq:rf:honest-x2}
    \Big\lVert\E\Big[\Big(\frac{s}{\lvert L^{1}(X_1)\rvert} \zeta_1(L^{1}(X_1) \cap\cM)-1\Big)^2\Biggiven \mD, D_1=1,X_1,X_2,1\in\cI^1,2\in\cI^1\Big]\Big\rVert_\infty = O(h^{2\epsilon}).
\end{equation}

For the second term in \eqref{eq:rf-consistency2:2part},
\begin{align*}
& \E[\Var[ \tilde{S}_2 \given \mD,\mX_1]\given\mD] \\
= & B^{-1} \E\Big[\Var\Big[ \frac{1}{n}\sum_{D_i=1} h^{-m/2} \ind\Big(X_i \in B^0(x, h^{1/2}d_K)\Big) \sum_{t\ge1} (\lvert \{k \in \cI^1:X_k \in L^1_{t}\} \rvert)^{-1}\ind(1 \in \cI^1:X_i \in L^1_{t}) \\
& \sum_{D_j=0} \ind\Big(X_j \in L^1_{t}\cap B^c(x, h^{1/2}d_K)\Big) \Biggiven  \mD,\mX_1\Big]\Biggiven \mD\Big]\\
\leq & B^{-1} \E\Big[\E\Big[ \frac{1}{n} \sum_{D_i=1} h^{-m} \ind\Big(X_i \in B^0(x, h^{1/2}d_K)\Big) \sum_{t\ge1} (\lvert \{k \in \cI^1:X_k \in L^1_{t}\} \rvert)^{-2}\ind(1 \in \cI^1:X_i \in L^1_{t}) \\
& \Big\{\sum_{D_j=0} \ind\Big(X_j \in L^1_{t}\cap B^c(x, h^{1/2}d_K)\Big)\Big\}^2 \Biggiven  \mD,\mX_1\Big]\Biggiven \mD\Big]\\
\lesssim & B^{-1} \E\Big[ \frac{1}{n} \sum_{D_i=1} h^{-m} \sum_{L^1_t \cap B(x, h^{1/2}d_K) \neq\emptyset} (\lvert \{k \in \cI^1:X_k \in L^1_{t}\} \rvert)^{-2}\ind(1 \in \cI^1:X_i \in L^1_{t}) \\
& \E\Big[\Big\{\sum_{D_j=0} \ind\Big(X_j \in L^1_{t}\cap B^c(x, h^{1/2}d_K)\Big)\Big\}^2 \Biggiven \mD,\mX_1,\{\cI^1\},\{L^1_{t}\}_{t \ge 1} \Big]\Biggiven \mD\Big]\\
\lesssim & B^{-1} \E\Big[ \frac{1}{n} \sum_{D_i=1} h^{-m} \sum_{L^1_t \cap B(x, h^{1/2}d_K) \neq\emptyset} (\lvert \{k \in \cI^1:X_k \in L^1_{t}\} \rvert)^{-2}\ind(1 \in \cI^1:X_i \in L^1_{t}) \\
& \Big\{n_0^2 \zeta_0^2\Big(L^1_{t}\cap B^c(x, h^{1/2}d_K)\Big)+ n_0 \zeta_0\Big(L^1_{t}\cap B^c(x, h^{1/2}d_K)\Big)\Big\}\Biggiven \mD\Big]\\
\lesssim & \E\Big[ \frac{1}{n} \sum_{D_i=1} h^{-m} \sum_{L^1_t \cap B(x, h^{1/2}d_K) \neq\emptyset}(\lvert \{k \in \cI^1:X_k \in L^1_{t}\} \rvert)^{-1}\ind(1 \in \cI^1:X_i \in L^1_{t}) \\
& \Big\{n_0 \zeta_0^2\Big(L^1_{t}\cap B^c(x, h^{1/2}d_K)\Big)+ \zeta_0\Big(L^1_{t}\cap B^c(x, h^{1/2}d_K)\Big)\Big\}\Biggiven \mD\Big]\\
= & \E\Big[ h^{-m} \sum_{L^1_t \cap B(x, h^{1/2}d_K) \neq\emptyset} \zeta_0^2\Big(L^1_{t}\cap B^c(x, h^{1/2}d_K)\Big) + \frac{1}{nh^m} \sum_{L^1_t \cap B(x, h^{1/2}d_K) \neq\emptyset} \zeta_0\Big(L^1_{t}\cap B^c(x, h^{1/2}d_K)\Big) \Biggiven \mD\Big]\\
\lesssim & \E\Big[ h^{-m/2} \sum_{L^1_t \cap B(x, h^{1/2}d_K) \neq\emptyset} \zeta_0\Big(L^1_{t}\cap B^c(x, h^{1/2}d_K)\Big)\Biggiven \mD\Big]
\lesssim h^{\epsilon}. \yestag\label{eq:rf-consistency2:var}
\end{align*}

According to the inverse moments of a binomial random variable \citep[Page 275]{cribari2000note},
\begin{equation}\label{eq:n-inverse}
    \E\Big[\Big(\frac{n_0}{n_1}\Big)^2\Biggiven D_1\Big] \lesssim n^2 \E[n_1^{-2}\given D_1] =O(1).
\end{equation}
Thus, we conclude that
\begin{equation}\label{eq:rf-consistency2:result}
    \E[\lvert S_2\rvert^2\given \mD] \lesssim h^{\epsilon}.
\end{equation}

{\bf Part 3.} 
Given Assumption~\ref{asp:rf1}\ref{asp:rf1-3}, we have
\begin{align*}
& \E \Big[ S_3 \Biggiven \mD,\mX_1,\{\cI^1_b\}_{b=1}^B,\{L^1_{bt}\}_{t \ge 1}\Big] \\
= & \frac{1}{n}\sum_{D_i=1} h^{-m/2} B^{-1} \sum_{b=1}^B \sum_{t\ge1} (\lvert \{k \in \cI^1_b:X_k \in L^1_{bt}\} \rvert)^{-1}\ind(1 \in \cI^1_b:X_i \in L^1_{bt}) K\left(h^{-1/2}\lVert X_i-x\rVert \right) \\
& \E \Big[\sum_{D_j=0}[\ind(X_j \in L^1_{bt}) - \zeta_0(L^1_{bt})] \Biggiven \mD,\mX_1,\{\cI^1_b\}_{b=1}^B,\{L^1_{bt}\}_{t \ge 1}\Big]
= 0.
\end{align*}

Then, using the law of total variance, we obtain
\begin{align*}
& \E[ S^2_3 ] = \Var[ S_3 ]
= \E[\Var[ S_3 \given \mX_1,\mD] ] \\
\leq & \frac{1}{Bh^m}\frac{1}{n}\E\Big[\sum_{D_i=1}K^2\left(h^{-1/2}\lVert X_i-x\rVert \right)\Var\Big[  \sum_{t\ge1} (\lvert \{k \in \cI^1:X_k \in L^1_{t}\} \rvert)^{-1}\ind(1 \in \cI^1:X_i \in L^1_{t})  \\
& \sum_{D_j=0} \Big[\ind\Big(X_j \in L^1_{t}\Big) - \zeta_0\Big(L^1_{t}\cap\cM\Big)\Big] \Biggiven \mD,\mX_1\Big]\Big] \\
= & \frac{1}{Bh^m}\frac{1}{n}\E\Big[\sum_{D_i=1}K^2\left(h^{-1/2}\lVert X_i-x\rVert \right)\E\Big[\sum_{t\ge1} (\lvert \{k \in \cI^1:X_k \in L^1_{t}\} \rvert)^{-2}\ind(1 \in \cI^1:X_i \in L^1_{t})  \\
& \Var\Big[\sum_{D_j=0} \Big[\ind\Big(X_j \in L^1_{t}\Big) - \zeta_0\Big(L^1_{t}\cap\cM\Big)\Big]\Biggiven \mD,\mX_1,\cI^1,\{L^1_{t}\}_{t \ge 1}\Big] \Biggiven \mD,\mX_1\Big]\Big] \\
\leq & \frac{1}{Bh^m}\frac{n_0}{n}\E\Big[\sum_{D_i=1}K^2\left(h^{-1/2}\lVert X_i-x\rVert \right)\sum_{t\ge1} (\lvert \{k \in \cI^1:X_k \in L^1_{t}\} \rvert)^{-2}\ind(1 \in \cI^1:X_i \in L^1_{t})\zeta_0(L^1_{t}\cap\cM) \Big] \\
\lesssim & \frac{1}{Bh^m}\E\Big[\sum_{D_i=1}\sum_{L^1_t \cap B(x, h^{1/2}d_K) \neq\emptyset} (\lvert \{k \in \cI^1:X_k \in L^1_{t}\} \rvert)^{-1}\ind(1 \in \cI^1:X_i \in L^1_{t})\zeta_0(L^1_{t}\cap\cM)\Big] \\
\leq & \frac{n}{B}\frac{1}{nh^m}\E\Big[\sum_{L^1_t \cap B(x, h^{1/2}d_K) \neq\emptyset}\zeta_0(L^1_{t}\cap\cM)\Big]\\
\lesssim& \frac{1}{nh^{m/2}}. \yestag\label{eq:rf-consistency3:result}
\end{align*}

{\bf Part 4.}
Given Assumption~\ref{asp:manifold} and Lemma~\ref{lemma:gbound}, we can conclude that $g_{0,x}/g_{1,x}$ and $\psi$ are Lipchitz continuous on the $\supp\left(g\right)$. This is because $g_{1,x}$ bounded away from zero implies that $\supp\left(g_{1,x}\right)$ must be bounded. Consequently, for any $x_1,x_2 \in \supp(\zeta)$, the following holds:
$$
\lvert g_{0,x}(z_1)/g_{1,x}(z_1) - g_{0,x}(z_2)/g_{1,x}(z_2) \rvert \lesssim \lVert x_1-x_2 \rVert.
$$
Also, we notice that for any $x \in Supp(\zeta)\cap\cM$ and $\omega \in \{0,1\}$,
\begin{align*}
    \frac{p\left(x\given D\right) P\left(D\right)}{p\left(x\right)}
    = P\left(D \given x\right)
    = P\left(D \given \psi(x)\right)
    = \frac{p\left(\psi(x)\given D\right) P\left(D\right)}{p\left(\psi(x)\right)}.
\end{align*}
Therefore,
\begin{equation*}
    \frac{\d \zeta_0(x)}{\d \zeta_1(x)} = \frac{p\left(x\given D=0\right)}{p\left(x\given D=1\right)} =\frac{p\left(\psi(x)\given D=0\right)}{p\left(\psi(x)\given D=1\right)} = \frac{g_{0,x}(z)}{g_{1,x}(z)}.
\end{equation*}
Thus, by Assumption~\ref{asp:rf1}\ref{asp:rf1-2}, we obtain
\begin{align*}
& \Big\lvert\zeta_0(L^1_{t}\cap\cM) - \frac{g_{0,x}(Z_1)}{g_{1,x}(Z_1)} \zeta_1(L^1_{t}\cap\cM)\Big\rvert
= \Big\lvert \int_{L^1_{t}\cap\cM} 1 \d \zeta_0(\tilde{x}) - \frac{g_{0,x}(Z_1)}{g_{1,x}(Z_1)} \int_{L^1_{t}\cap\cM} 1 \d \zeta_1(\tilde{x}) \Big\rvert\\
\leq & \int_{L^1_{t}\cap\cM} \Big\lvert\frac{\d \zeta_0(x)}{\d \zeta_1(x)}- \frac{g_{0,x}(Z_1)}{g_{1,x}(Z_1)}\Big\rvert \d \zeta_1(\tilde{x})
= \int_{L^1_{t}\cap\cM} \Big\lvert\frac{g_{0,x}(z)}{g_{1,x}(z)}- \frac{g_{0,x}(Z_1)}{g_{1,x}(Z_1)}\Big\rvert \d \zeta_1(\tilde{x})\\
\lesssim & \int_{L^1_{t}\cap\cM}  {\rm diam}(L^1_{t}\cap \cM) \d \zeta_1(\tilde{x})
\leq h^{1/2+\epsilon} \zeta_1(L^1_{t}\cap \cM). \yestag\label{eq:rf-consistency4:leaf}
\end{align*}

Given Assumption~\ref{asp:rf1}\ref{asp:rf1-2}, define
\begin{align*}
\tilde{S}_4 & := \frac{1}{n}\sum_{D_i=1} h^{-m/2}  B^{-1} \sum_{b=1}^B \sum_{t\ge1} (\lvert \{k \in \cI^1_b:X_k \in L^1_{bt}\} \rvert)^{-1}\ind(1 \in \cI^1_b:X_i \in L^1_{bt})  \\
& K\left(h^{-1/2}\lVert X_i-x\rVert \right) n_0 \zeta_1(L^1_{t}\cap \cM).
\end{align*}
Then $\lvert S_4 \rvert \lesssim h^{1/2+\epsilon} \tilde{S}_4$, and
\begin{equation}\label{eq:rf-consistency4:2part-1}
\E[S_4^2\given \mD] \lesssim h^{1+\epsilon} \E[\{\E[ \tilde{S}_4 \given \mD,\mX_1]\}^2\given\mD] + h^{1+\epsilon} \E[\Var[ \tilde{S}_4 \given \mD,\mX_1]\given\mD].
\end{equation}

For the first term in \eqref{eq:rf-consistency4:2part-1}, similar to \eqref{eq:rf-consistency2:e}, we obtain that
\begin{equation}\label{eq:rf-consistency4:e}
    \E[\{\E[ \tilde{S}_4 \given \mD,\mX_1]\}^2\given\mD] = O(1).
\end{equation}

For the second term in \eqref{eq:rf-consistency4:2part-1},
\begin{align*}
& \E[\Var[ \tilde{S}_4 \given \mD,\mX_1]\given\mD] \\
= & B^{-1} \frac{n_0^2}{n^2} h^{-m} \E\Big[\Var\Big[\sum_{D_i=1} \sum_{t\ge1} (\lvert \{k \in \cI^1:X_k \in L^1_{t}\} \rvert)^{-1}\ind(1 \in \cI^1:X_i \in L^1_{t})  K\left(h^{-1/2}\lVert X_i-x\rVert \right) \\
& \zeta_1(L^1_{t}\cap \cM)\Biggiven \mD,\mX_1\Big]\Biggiven\mD\Big] \\
\leq & \frac{n_1}{B} h^{-m} \E\Big[\sum_{D_i=1} K^2\left(h^{-1/2}\lVert X_i-x\rVert \right) \sum_{t\ge1} (\lvert \{k \in \cI^1:X_k \in L^1_{t}\} \rvert)^{-2}\ind(1 \in \cI^1:X_i \in L^1_{t})   \\
& \zeta_1^2(L^1_{t}\cap \cM)\Biggiven\mD\Big]
\\
\lesssim & h^{-m} \E\Big[\sum_{D_i=1} \sum_{L^1_t \cap B(x, h^{1/2}d_K) \neq\emptyset} (\lvert \{k \in \cI^1:X_k \in L^1_{t}\} \rvert)^{-1}\ind(1 \in \cI^1:X_i \in L^1_{t})\zeta_1^2(L^1_{t}\cap\cM)\Biggiven\mD\Big]\\
= & h^{-m} \E\Big[\sum_{L^1_t \cap B(x, h^{1/2}d_K) \neq\emptyset} \zeta_1^2(L^1_{t}\cap\cM)\Biggiven\mD\Big]\\
\lesssim & h^{-m} \E\Big[ \lVert {\rm diam}^m(L_t\cap\cM) \rVert_\infty \Big(\sum_{L^1_t \cap B(x, h^{1/2}d_K) \neq\emptyset} \zeta_1(L^1_{t}\cap\cM)\Big)\Biggiven\mD\Big]
= O(h^{\epsilon}).\yestag\label{eq:rf-consistency4:var}
\end{align*}

Therefore, we conclude that
\begin{equation}\label{eq:rf-consistency4:result}
    \E[S_4^2] = O(h^{1+2\epsilon}).
\end{equation}

{\bf Part 5.}
For $S_5$, define
\begin{align*}
\tilde{S}_5 & := \frac{1}{n_1}\sum_{D_i=1} h^{-m/2} K\left(h^{-1/2}\lVert X_i-x\rVert \right) \Big[n_1 B^{-1} \sum_{b=1}^B \sum_{t\ge1} (\lvert \{k \in \cI^1_b:X_k \in L^1_{bt}\} \rvert)^{-1}\ind(1 \in \cI^1_b:X_i \in L^1_{bt}) \\
& \zeta_1(L^1_{bt}\cap\cM) - 1\Big],
\end{align*}
and
\begin{equation}\label{eq:rf-consistency5:2part-1}
\E[S_5^2\given \mD] \lesssim \E[\{\E[ \tilde{S}_5 \given \mD,\mX_1]\}^2\given\mD] +  \E[\Var[ \tilde{S}_5 \given \mD,\mX_1]\given\mD].
\end{equation}

For the first term in \eqref{eq:rf-consistency5:2part-1},
\begin{align*}
& \E[\{\E[ \tilde{S}_5 \given \mD,\mX_1]\}^2\given\mD] \\ 
\leq & \E\Big[\Big\{\frac{1}{n_1}\sum_{D_i=1} h^{-m/2} K\left(h^{-1/2}\lVert X_i-x\rVert \right) \E\Big[\frac{s}{\lvert L^1(X_i)\rvert}\zeta_1\Big(L^1(X_i)\cap\cM\Big)-1\Biggiven \mD,D_i=1,\mX_1,i\in \cI^1\Big]\Big\}^2\Biggiven \mD\Big]\\
= & \frac{1}{n_1h^{m/2}} \E\Big[ h^{-m/2} K\left(h^{-1/2}\lVert X_1-x\rVert \right)\Big\{\E\Big[\frac{s}{\lvert L^1(X_1)\rvert} \zeta_0\Big(L^1(X_1)\cap\cM\Big)-1\Biggiven \mD,D_1=1,\mX_1,1\in \cI^1\Big]\Big\}^2\Biggiven \mD,D_1=1\Big]\\
& + \frac{n_1-1}{n_1} \E\Big[ h^{-m} K\left(h^{-1/2}\lVert X_1-x\rVert \right) K\left(h^{-1/2}\lVert X_2-x\rVert \right) \\
& \E\Big[\frac{s}{\lvert L^1(X_1)\rvert} \zeta_0\Big(L^1(X_1)\cap\cM\Big)-1\Biggiven \mD,D_1=1,\mX_1,1\in \cI^1\Big] \\
& \E\Big[\frac{s}{\lvert L^1(X_2)\rvert} \zeta_0\Big(L^1(X_2)\cap\cM\Big)-1\Biggiven \mD,D_2=1,\mX_1,2\in \cI^1\Big]\Biggiven \mD, D_1=D_2=1\Big] \\
\lesssim & \frac{1}{n_1h^{m/2}} \E\Big[\Big\{\frac{s}{\lvert L^1(X_1)\rvert} \zeta_0\Big(L^1(X_1)\cap\cM\Big)-1\Big\}^2\Biggiven \mD,D_1=1,X_1 \in B(x, h^{1/2}d_K), 1\in \cI^1 \Big]\\
& + \E\Big[\E\Big[\frac{s}{\lvert L^1(X_1)\rvert} \zeta_0\Big(L^1(X_1)\cap\cM\Big)-1\Biggiven \mD,D_1=1,\mX_1,1\in \cI^1\Big] \E\Big[\frac{s}{\lvert L^1(X_2)\rvert} \zeta_0\Big(L^1(X_2)\cap\cM\Big)-1\\
& \Biggiven \mD,D_2=1,\mX_1,2\in \cI^1\Big] \Biggiven \mD, D_1=D_2=1,X_1\in B(x, h^{1/2}d_K),X_2\in B(x, h^{1/2}d_K)\Big] \\
\lesssim & \frac{1}{n_1h^{m/2}} \E\Big[\Big\{\frac{s}{\lvert L^1(X_1)\rvert} \zeta_0\Big(L^1(X_1)\cap\cM\Big)-1\Big\}^2\Biggiven \mD,D_1=1,X_1 \in B(x, h^{1/2}d_K), 1\in \cI^1 \Big]\\
& + \Big\{\E\Big[\Big\{\frac{s}{\lvert L^1(X_1)\rvert} \zeta_0\Big(L^1(X_1)\cap\cM\Big)-1\Big\}^{2}\Biggiven \mD, D_1=D_2=1,X_1\in B(x, h^{1/2}d_K),X_2\in B(x, h^{1/2}d_K)\Big]\Big\}^{1/2} \\
& \Big\{\E\Big[\Big\{\frac{s}{\lvert L^1(X_2)\rvert} \zeta_0\Big(L^1(X_2)\cap\cM\Big)-1\Big\}^{2}\Biggiven \mD, D_1=D_2=1,X_1\in B(x, h^{1/2}d_K),X_2\in B(x, h^{1/2}d_K)\Big]\Big\}^{1/2}, \\
\lesssim & \Big(1+\frac{1}{n_1h^{m/2}}\Big)h^{2\epsilon}. \yestag\label{eq:rf-consistency5:e}
\end{align*}

For the second term in \eqref{eq:rf-consistency5:2part-1}, similar to \eqref{eq:rf-consistency4:var}, we obtain
\begin{align*}
& \E[\Var[\tilde{S}_5 \given \mX_1,\mD] \given \mD] \\
= & \E\Big[\Var\Big[\frac{1}{n_1}\sum_{D_i=1} h^{-m/2} K\left(h^{-1/2}\lVert X_i-x\rVert \right) n_1 B^{-1} \sum_{b=1}^B \sum_{t\ge1} (\lvert \{k \in \cI^1_b:X_k \in L^1_{bt}\} \rvert)^{-1}\ind(1 \in \cI^1_b:X_i \in L^1_{bt}) \\
& \zeta_1(L^1_{bt}\cap\cM) \Biggiven \mD,\mX_1\Big] \Biggiven \mD\Big]
\lesssim h^{\epsilon}. \yestag\label{eq:rf-consistency5:var}
\end{align*}

Therefore, we conclude that
\begin{equation}\label{eq:rf-consistency5:result}
    \E[S_5^2\given D_1] = O(h^{\epsilon}).
\end{equation}

{\bf Part 6.} 
Notice that
\begin{align*}
e\left(X_1\right) = \P\left(D_1=1\Biggiven X=X_1\right)
= \P\left(D_1=1\Biggiven Z=Z_1\right)
= \frac{g_{1,x}\left(Z_1\right) \P\left(D=1\right)}{g\left(Z_1\right)}.
\end{align*}
By Assumption~\ref{asp:manifold}, we have
\begin{align*}
S_6 = & \frac{1}{n}\sum_{D_i=1} h^{-m/2} K\left(h^{-1/2}\lVert X_i-x\rVert \right) \frac{g_{0,x}\left(Z_i\right)}{g_{1,x}\left(Z_i\right)} \left(\frac{n_0}{n_1}-\frac{\P(D=0)}{\P(D=1)}\right) \\
\lesssim & \frac{1}{n}\sum_{D_i=1} h^{-m/2} K\left(h^{-1/2}\lVert X_i-x\rVert \right) \Big(\frac{n_0}{n_1}-\frac{\P(D=0)}{\P(D=1)}\Big)
=: \tilde{S}_6.
\end{align*}

Thus given $n_1\ge1$,
\begin{align*}
& \E[S_6^2] \lesssim \E[\tilde{S}_6^2] \\
= & \E\Big[\Big(\frac{n_0}{n_1}-\frac{\P(D=0)}{\P(D=1)}\Big)^2\E\Big[\Big\{\frac{1}{n}\sum_{D_i=1} h^{-m/2} K\left(h^{-1/2}\lVert X_i-x\rVert \right)\Big\}^2\Biggiven \mD\Big]\Big]\\
= & \E\Big[\Big(\frac{n_0}{n_1}-\frac{\P(D=0)}{\P(D=1)}\Big)^2\Big]
= \E\Big[\left(\frac{\frac{n_0}{n} \P(D=1)-\frac{n_1}{n} \P(D=0)}{\frac{n_1}{n} \P(D=1)}\right)^2 \Big]\\
\lesssim & \E\Big[\left(\frac{n}{n_1}\right)^{2}\left(\frac{n_0}{n} \P(D=1)-\P(D=1) \P(D=0)+\P(D=1) \P(D=0)-\frac{n_1}{n} \P(D=0)\right)^2 \Big] \\
\lesssim & \E\Big[{\left(\frac{n}{n_1} \P(D=1) \left(\frac{n_0}{n}-\P(D=0)\right) \right)}^2  \Biggiven X_1,D_1=1 \Big] + \E\Big[{\left(\frac{n}{n_1} \P(D=0) \left(\frac{n_1}{n}-\P(D=1)\right)\right)}^2 \Big] \\
\lesssim & \E\Big[{\left(\frac{n}{n_1} \left(\frac{n_0}{n}-\P(D=0)\right) \right)}^2 \Big] + \E\Big[{\left(\frac{n}{n_1} \left(\frac{n_1}{n}-\P(D=1)\right) \right)}^2  \Big].
\end{align*}

Regarding the two terms above, it is only necessary to focus on the first term; the second term can be derived using the same approach. According to the inverse moments of a binomial random variable \citep[Page 275]{cribari2000note} and Hölder's inequality,
\begin{align*}
& \E\Big[\Big(\frac{n}{n_1} \left(\frac{n_0}{n}-\P(D=0)\right) \Big)^2 \Big]
\leq \Big\{\E\Big[\Big(\frac{n}{n_1}\Big)^4\Big]\Big\}^{1/2} \cdot \Big\{\E\Big[\Big(\frac{n_0}{n}-\P(D=0)\Big)^4\Big]\Big\}^{1/2} \\
\leq & \Big\{\E\Big[n^4\Big(1+\sum_{i=1}^n\ind\left(D_i=1\right)\Big)^{-4}\Big]\Big\}^{1/2} \cdot \Big\{n^{-4} \E\Big[\Big\{\sum_{i=1}^n\Big(\ind\left(D_i=0\right)-\P(D=0)\Big)\Big\}^4\Big]\Big\}^{1/2} \\
\lesssim & \Big\{n^{-4}\E\Big[\Big\{\sum_{i=1}^n\Big(\ind\left(D_i=0\right)-\P(D=0)\Big)\Big\}^4\Big]\Big\}^{1/2}
\lesssim \Big\{n^{-4} n^2\Big\}^{1/2}
\lesssim \frac{1}{n},
\end{align*}
where in the last step we use the fourth central moment of binomial distribution.

Therefore, we conclude that
\begin{equation}\label{eq:rf-consistency6:result}
    \E [S_6^2] = O(n^{-1}).
\end{equation}


{\bf Part 7.}
\begin{align*}
\lvert S_7\rvert = & \frac{1}{n}\sum_{D_i=1} B^{-1} \sum_{b=1}^B \sum_{t\ge1} (\lvert \{k \in \cI^1_b:X_k \in L^1_{bt}\} \rvert)^{-1}\ind(1 \in \cI^1_b:X_i \in L^1_{bt}) \sum_{D_j=0}  \ind\Big(X_j \in L^1_{bt}\Big) \\
& h^{-m/2} K\left(h^{-1/2}\lVert X_j-x\rVert \right) \ind\Big(X_i \in B^c(x, h^{1/2}d_K)\Big) \\
\lesssim & \frac{1}{n}\sum_{D_i=1} B^{-1} \sum_{b=1}^B \sum_{\substack{L^1_{bt} \cap B(x, h^{1/2}d_K) \neq\emptyset, \\
L^1_{bt} \cap B^c(x, h^{1/2}d_K) \neq\emptyset}} (\lvert \{k \in \cI^1_b:X_k \in L^1_{bt}\} \rvert)^{-1}\ind(1 \in \cI^1_b:X_i \in L^1_{bt}) \\
& \sum_{D_j=0} \ind\Big(X_j \in L^1_{bt}\Big) h^{-m/2} \ind\Big(X_j \in B(x, h^{1/2}d_K)\Big) \\
\leq & \frac{1}{n}\sum_{D_i=1} \sum_{D_j=0} w_{j\leftarrow i} h^{-m/2} \ind\Big(X_j \in B^0(x, h^{1/2}d_K)\Big)
= \frac{1}{n}\sum_{D_j=0} h^{-m/2}\ind\Big(X_j \in B^0(x, h^{1/2}d_K)\Big).
\end{align*}

Thus, 
\begin{align*}
\E[ S^2_7 ] = \E\Big[\Big(\frac{1}{n}\sum_{D_j=0} h^{-m/2}\ind\Big(X_j\in B^0(x, h^{1/2}d_K)\Big)\Big)^2\Big]
\lesssim h^{\epsilon}.\yestag\label{eq:rf-consistency7:result}
\end{align*}

Plugging \eqref{eq:rf-consistency1:result}, \eqref{eq:rf-consistency2:result}, \eqref{eq:rf-consistency3:result}, \eqref{eq:rf-consistency4:result}, \eqref{eq:rf-consistency5:result}, \eqref{eq:rf-consistency6:result} and \eqref{eq:rf-consistency7:result} to \eqref{eq:rf-consistency:7part} yields \eqref{asp:dr-propensity-2-sum}.
\end{proof}

\subsection{Proof of Lemma~\ref{lemma:rf-dr-propensity-2-single}}

\begin{proof}[Proof of Lemma~\ref{lemma:rf-dr-propensity-2-single}]
Similar to Lemma~\ref{lemma:rf-dr-propensity-2-sum}, to verify \eqref{asp:dr-propensity-2-single} in Assumption~\ref{asp:dr-propensity}\ref{asp:dr-propensity-2}, we consider the first term above under $D_1 = 1$, and decompose the term as
\begin{align*}
& h^{\frac{d-m}{2}} \sum_{D_j=0} K_{h, j} w_{j \leftarrow 1} - h^{\frac{d-m}{2}}K_{h, 1} \frac{1-e\left(X_1\right)}{e\left(X_1\right)}\\
= & B^{-1} \sum_{b=1}^B \sum_{t\ge1} (\lvert \{k \in \cI^1_b:X_k \in L^1_{bt}\} \rvert)^{-1}\ind(1 \in \cI^1_b:X_1 \in L^1_{bt}) \sum_{D_j=0}  \ind\Big(X_j \in L^1_{bt}\cap B(x, h^{1/2}d_K)\Big)\\
& h^{-m/2} \Big[ K\left(h^{-1/2}\lVert X_j-x\rVert \right) -  K\left(h^{-1/2}\lVert X_1-x\rVert \right) \Big]\ind\Big(X_1 \in B(x, h^{1/2}d_K)\Big) \\
& + h^{-m/2} B^{-1} \sum_{b=1}^B \sum_{t\ge1} (\lvert \{k \in \cI^1_b:X_k \in L^1_{bt}\} \rvert)^{-1}\ind(1 \in \cI^1_b:X_1 \in L^1_{bt}) K\left(h^{-1/2}\lVert X_1-x\rVert \right) \\
& \sum_{D_j=0} \Big[ \ind\Big(X_j \in L^1_{bt}\cap B(x, h^{1/2}d_K)\Big) - \ind\Big(X_j \in L^1_{bt}\Big)\Big]\\
& + h^{-m/2} B^{-1} \sum_{b=1}^B \sum_{t\ge1} (\lvert \{k \in \cI^1_b:X_k \in L^1_{bt}\} \rvert)^{-1}\ind(1 \in \cI^1_b:X_1 \in L^1_{bt}) K\left(h^{-1/2}\lVert X_1-x\rVert \right) \\
& \sum_{D_j=0} \Big[\ind\Big(X_j \in L^1_{bt}\Big) - \zeta_0\Big(L^1_{bt}\cap\cM\Big)\Big]\\
& + n_0 h^{-m/2} B^{-1} \sum_{b=1}^B \sum_{t\ge1} (\lvert \{k \in \cI^1_b:X_k \in L^1_{bt}\} \rvert)^{-1}\ind(1 \in \cI^1_b:X_1 \in L^1_{bt})  K\left(h^{-1/2}\lVert X_1-x\rVert \right) \\
& \Big[\zeta_0\Big(L^1_{bt}\cap\cM\Big) - \frac{g_{0,x}(Z_1)}{g_{1,x}(Z_1)} \zeta_1\Big(L^1_{bt}\cap\cM\Big)\Big]\\
& + \frac{n_0}{n_1} \frac{g_{0,x}(Z_1)}{g_{1,x}(Z_1)} h^{-m/2} K\left(h^{-1/2}\lVert X_1-x\rVert \right) \Big[n_1 B^{-1} \sum_{b=1}^B \sum_{t\ge1} (\lvert \{k \in \cI^1_b:X_k \in L^1_{bt}\} \rvert)^{-1}\ind(1 \in \cI^1_b:X_1 \in L^1_{bt}) \\
& \zeta_1(L^1_{bt}\cap\cM) - 1\Big] \\
& + h^{-m/2} K\left(h^{-1/2}\lVert X_1-x\rVert \right) \left(\frac{n_0}{n_1} \frac{g_{0,x}\left(Z_1\right)}{g_{1,x}\left(Z_1\right)} - \frac{1-e\left(X_1\right)}{e\left(X_1\right)}\right) \\
& + B^{-1} \sum_{b=1}^B \sum_{t\ge1} (\lvert \{k \in \cI^1_b:X_k \in L^1_{bt}\} \rvert)^{-1}\ind(1 \in \cI^1_b:X_1 \in L^1_{bt}) \sum_{D_j=0}  \ind\Big(X_j \in L^1_{bt}\cap B(x, h^{1/2}d_K)\Big)\\
& h^{-m/2} K\left(h^{-1/2}\lVert X_j-x\rVert \right) \ind\Big(X_1 \in B^c(x, h^{1/2}d_K)\Big) \\
=:& T_1 + T_2 + T_3 + T_4 + T_5 + T_6 + T_7.
\end{align*}
Thus by Jensen's Inequality,
\begin{align*}
h^{d-m} \E \Big[ \Big(\sum_{j:D_j=0} K_h\left(X_j-x\right)w_{j\leftarrow 1} - K_h\left(X_1-x\right)\frac{1-e(X_1)}{e(X_1)}\Big)^2 \Biggiven D_1=1 \Big]
\lesssim \sum_{i=1}^6 \E \Big[ {T_i}^2 \Biggiven D_1=1\Big].
\end{align*}


{\bf Part 1.}
For $T_1$, similar to \eqref{eq:rf-consistency1:s1}, define
\begin{align*}
\tilde{T}_1 & :=  B^{-1} \sum_{b=1}^B \sum_{t\ge1} (\lvert \{k \in \cI^1_b:X_k \in L^1_{bt}\} \rvert)^{-1}\ind(1 \in \cI^1_b:X_1 \in L^1_{bt}) \sum_{D_j=0}  \ind\Big(X_j \in L^1_{bt}\cap B(x, h^{1/2}d_K)\Big)\\
& h^{-m/2} \ind\Big(X_1 \in B(x, h^{1/2}d_K)\Big),
\end{align*}
and
\begin{equation*}
\E[T_1^2\given \mD,D_1=1] \lesssim h^{2\epsilon} \E[\{\E[ \tilde{T}_1 \given \mD,D_1=1,\mX_1]\}^2\given\mD,D_1=1] + h^{2\epsilon} \E[\Var[ \tilde{T}_1 \given \mD,\mX_1]\given\mD,D_1=1].
\end{equation*}

For the first term, by Lemma~\ref{lemma:rf-clt},
\begin{align*}
& \E[\{\E[ \tilde{T}_1 \given \mD,D_1=1,\mX_1]\}^2\given\mD,D_1=1]\\
= & \E \Big[h^{-m} \ind\Big(X_1 \in B(x, h^{1/2}d_K)\Big) \Big\{\E \Big[\sum_{t\ge1} (\lvert \{k \in \cI^1:X_k \in L^1_{t}\} \rvert)^{-1}\ind(1 \in \cI^1:X_1 \in L^1_{t}) \\
& n_0 \zeta_0\Big(L^1_{t}\cap B(x, h^{1/2}d_K)\cap\cM\Big)\Biggiven \mD,D_1=1 \Big]\Big\}^2\Big]\\
\leq &  \E\Big[h^{-m}\ind\Big(X_1 \in B(x, h^{1/2}d_K)\Big)\E\Big[\Big\{\frac{s}{\lvert  L^1(X_1)\rvert} \zeta_0\Big(L^1(X_1)\cap \cM\Big)\Big\}^2\Biggiven \mD,D_1=1,\mX_1,1\in\cI^1\Big]\Biggiven \mD,D_1=1\Big] \\
= & h^{-m} \E\Big[\ind\Big(X_1 \in B(x, h^{1/2}d_K)\Big)\Big(\frac{s}{\lvert  L^1(X_1)\rvert} \zeta_0(L^1(X_1))\Big)^2\Biggiven \mD,D_1=1\Big]
\lesssim h^{-m/2}. \yestag\label{eq:rfnew-consistency1:e}
\end{align*}

For the second term,
\begin{align*}
& \E[\Var[ \tilde{T}_1 \given \mD,D_1=1,\mX_1]\given \mD,D_1=1]\\
= & B^{-1}  h^{-m} \E\Big[\ind\Big(X_1 \in B(x, h^{1/2}d_K)\Big) \Var\Big[ \sum_{t\ge1} (\lvert \{k \in \cI^1:X_k \in L^1_{t}\} \rvert)^{-1}\ind(1 \in \cI^1:X_1 \in L^1_{t})\\
& \sum_{D_j=0} \ind\Big(X_j \in L^1_{t}\cap B(x, h^{1/2}d_K)\Big)\Biggiven\mD,D_1=1,\mX_1\Big]\Biggiven \mD,D_1=1\Big]\\
\leq & B^{-1}  h^{-m} \E\Big[\ind\Big(X_1 \in B(x, h^{1/2}d_K)\Big) \E\Big[ \sum_{t\ge1} (\lvert \{k \in \cI^1:X_k \in L^1_{t}\} \rvert)^{-2}\ind(1 \in \cI^1:X_1 \in L^1_{t})\\
& n_0 \sum_{D_j=0} \ind\Big(X_j \in L^1_{t}\cap B(x, h^{1/2}d_K)\Big)\Biggiven \mD,D_1=1,\mX_1\Big]\Biggiven \mD,D_1=1\Big]\\
\lesssim & \frac{n_0}{n_1}  h^{-m} \E\Big[ \sum_{D_i=1} \sum_{t\ge1} (\lvert \{k \in \cI^1:X_k \in L^1_{t}\} \rvert)^{-1}\ind(i \in \cI^1:X_i \in L^1_{t})\zeta_0\Big(L^1_{t}\cap B(x, h^{1/2}d_K)\Big) \Biggiven \mD\Big]\\
= & \frac{n_0}{n_1}  h^{-m} \E\Big[\zeta_0\Big(B(x, h^{1/2}d_K)\Big) \Biggiven \mD\Big]\\
\lesssim& \frac{n_0}{n_1}  h^{-m/2}. \yestag\label{eq:rfnew-consistency1:var}
\end{align*}

Therefore, combining \eqref{eq:rfnew-consistency1:e}-\eqref{eq:rfnew-consistency1:var} with \eqref{eq:n-inverse}, we conclude that
\begin{equation}\label{eq:rfnew-consistency1:result}
    \E[T_1^2\given D_1=1] = O(h^{-m/2+2\epsilon}).
\end{equation}

{\bf Part 2.}
For
\begin{align*}
\lvert T_2\rvert
= &  h^{-m/2} K\left(h^{-1/2}\lVert X_1-x\rVert \right) B^{-1} \sum_{b=1}^B \sum_{t\ge1} (\lvert \{k \in \cI^1_b:X_k \in L^1_{bt}\} \rvert)^{-1}\ind(1 \in \cI^1_b:X_1 \in L^1_{bt}) \\
& \sum_{D_j=0} \ind\Big(X_j \in L^1_{bt}\cap B^c(x, h^{1/2}d_K)\Big),
\end{align*}
we decompose it as
\begin{align*}
\E[ T_2^2\given \mD, D_1=1]
\lesssim \E[\{\E[ \lvert T_2\rvert\given \mD, D_1=1,\mX_1]\}^2 \given \mD, D_1=1] + \E[\Var[ \lvert T_2\rvert\given \mD, D_1=1,\mX_1] \given \mD, D_1=1].
\end{align*}

For the first term, similar to \eqref{eq:rf-consistency2:e} and \eqref{eq:rfnew-consistency1:e}, 
\begin{align*}
& \E[\{\E[ \lvert T_2\rvert\given \mD, D_1=1,\mX_1]\}^2 \given \mD, D_1=1]\\
\lesssim & h^{-m} \E\Big[\ind\Big(X_1 \in B^0(x, h^{1/2}d_K)\Big)\Big(\frac{s}{\lvert  L^1(X_1)\rvert} \zeta_0(L^1(X_1))\Big)^2\Biggiven \mD,D_1=1\Big]\\
\lesssim& h^{-m/2+\epsilon}. \yestag\label{eq:rfnew-consistency2:e}
\end{align*}

For the second term, similar to \eqref{eq:rf-consistency2:var} and \eqref{eq:rfnew-consistency1:var}, we can express it as
\begin{align*}
& \E[\Var[ \lvert T_2\rvert\given \mD, D_1=1,\mX_1] \given \mD, D_1=1] \\
\leq & \frac{1}{B h^{m}} \E\Big[K^2\left(h^{-1/2}\lVert X_1-x\rVert \right)\E\Big[\sum_{t\ge1} (\lvert \{k \in \cI^1:X_k \in L^1_{t}\} \rvert)^{-2}\ind(1 \in \cI^1:X_1 \in L^1_{t}) \\
& n_0 \sum_{D_j=0} \ind\Big(X_j \in L^1_{t}\cap B^c(x, h^{1/2}d_K)\Big) \Biggiven \mD, D_1=1,\mX_1] \Biggiven \mD, D_1=1\Big] \\
\lesssim & n_0 h^{-m} \E\Big[\sum_{L^1_t \cap B(x, h^{1/2}d_K) \neq\emptyset} (\lvert \{k \in \cI^1:X_k \in L^1_{t}\} \rvert)^{-1}\ind(1 \in \cI^1:X_1 \in L^1_{t}) \\
& \zeta_0\Big(L^1_{t}\cap B^c(x, h^{1/2}d_K)\cap\cM\Big) \Biggiven \mD, D_1=1\Big]\\
\lesssim& \frac{n_0}{n_1} h^{-m/2}h^{\epsilon}. \yestag\label{eq:rfnew-consistency2:var}
\end{align*}

Therefore, combining \eqref{eq:rfnew-consistency2:e}-\eqref{eq:rfnew-consistency2:var}, we conclude that
\begin{equation}\label{eq:rfnew-consistency2:result}
    \E[T_2^2\given D_1=1] = O(h^{-m/2+\epsilon}).
\end{equation}

{\bf Part 3.} 
Similar to \eqref{eq:rf-consistency3:result}, 
$\E[T_3\given \mD,\mX_1, \{I_b^1\}_{b=1}^B, \{L_{bt}^1\}_{t\ge 1}] = 0$, and
\begin{align*}
& \E[ T^2_3 \given \mD,D_1=1]
= \E[\Var[ T_3 \given \mD,D_1=1,\mX_1] \given \mD,D_1=1]\\
= & \frac{n_0}{Bh^m} \E\Big[K^2\left(h^{-1/2}\lVert X_1-x\rVert \right) \sum_{t\ge1} (\lvert \{k \in \cI^1:X_k \in L^1_{t}\} \rvert)^{-2}\ind(1 \in \cI^1:X_1 \in L^1_{t})   \zeta_0 \Big(L^1_{t}\cap\cM\Big) \Biggiven \mD,D_1=1\Big]\\
\lesssim & h^{-m} \E\Big[\sum_{L^1_t \cap B(x, h^{1/2}d_K) \neq\emptyset} (\lvert \{k \in \cI^1:X_k \in L^1_{t}\} \rvert)^{-1}\ind(1 \in \cI^1:X_1 \in L^1_{t})\zeta_0(L_t^1)\Biggiven \mD,D_1=1\Big]\\
= & \frac{1}{n_1 h^{m}} \E\Big[\sum_{L^1_t \cap B(x, h^{1/2}d_K) \neq\emptyset} \zeta_0(L_t^1)\Biggiven \mD,D_1=1\Big]\\
\lesssim& \frac{1}{n_1 h^{m/2}}.
\end{align*}

Therefore, we conclude that
\begin{equation}\label{eq:rfnew-consistency3:result}
    \E[T_3^2\given D_1=1]  = O\Big(\frac{1}{n h^{m/2}}\Big).
\end{equation}

{\bf Part 4.}
For $T_4$, according to \eqref{eq:rf-consistency4:leaf}, we have
\begin{align*}
\lvert T_4\rvert \lesssim & n_0 h^{-m/2+1/2+\epsilon} B^{-1} \sum_{b=1}^B \sum_{t\ge1} (\lvert \{k \in \cI^1_b:X_k \in L^1_{bt}\} \rvert)^{-1}\ind(1 \in \cI^1_b:X_1 \in L^1_{bt})  K\left(h^{-1/2}\lVert X_1-x\rVert \right)
\\
& \zeta_1\Big(L^1_{bt}\cap\cM\Big)
=: h^{1/2+\epsilon} \tilde{T}_4,
\end{align*}
and we decompose it as
\begin{align*}
\E[ T_4^2\given \mD, D_1=1]
\lesssim & h^{1+\epsilon} \E[\{\E[ \tilde{T}_4\given \mD, D_1=1,\mX_1]\}^2 \given \mD, D_1=1] + \\
&~~h^{1+\epsilon} \E[\Var[ \tilde{T}_4\given \mD, D_1=1,\mX_1] \given \mD, D_1=1].
\end{align*}

The analysis of the first term can follow the same approach as \eqref{eq:rfnew-consistency2:e} and the second term can follow \eqref{eq:rf-consistency4:var}. And we conclude that
\begin{equation}\label{eq:rfnew-consistency4:result}
    \E[T_4^2\given D_1=1] = O(h^{-m/2+2\epsilon}).
\end{equation}

{\bf Part 5.}
Define
\begin{align*}
\tilde{T}_5 & := \frac{n_0}{n_1} h^{-m/2} K\left(h^{-1/2}\lVert X_1-x\rVert \right) \Big[n_1 B^{-1} \sum_{b=1}^B \sum_{t\ge1} (\lvert \{k \in \cI^1_b:X_k \in L^1_{bt}\} \rvert)^{-1}\ind(1 \in \cI^1_b:X_1 \in L^1_{bt}) \\
& \zeta_1(L^1_{bt}\cap\cM) - 1\Big],
\end{align*}
and similar to \eqref{eq:rf-consistency5:result}, we can break this term down as follows:
\begin{align*}
\E[ T_5^2\given \mD, D_1=1]
\lesssim \E[\{\E[ \tilde{T}_5\given \mD, D_1=1,\mX_1]\}^2 \given \mD, D_1=1] + \E[\Var[ \tilde{T}_5\given \mD, D_1=1,\mX_1] \given \mD, D_1=1].
\end{align*}

For the first term, similar to \eqref{eq:rf-consistency5:e},
\begin{align*}
& \E[\{\E[ \tilde{T}_5\given \mD, D_1=1,\mX_1]\}^2 \\ 
\lesssim & \Big(\frac{n_0}{n_1}\Big)^2\E\Big[h^{-m} K^2\left(h^{-1/2}\lVert X_1-x\rVert \right) \Big(\frac{s}{\lvert L^1(X_i)\rvert}\zeta_1\Big(L^1(X_1)\cap\cM\Big)-1\Big)^2\Biggiven \mD\Big]\\
\lesssim & \frac{n_0^2}{n_1^2} h^{-m/2+\epsilon}. \yestag\label{eq:rfnew-consistency5:e}
\end{align*}

And the second term can be obtained in the same way as \eqref{eq:rf-consistency5:var}.

Therefore, we conclude that
\begin{equation}\label{eq:rfnew-consistency5:result}
    \E[T_5^2\given D_1=1] = O(h^{-m/2+\epsilon}).
\end{equation}

{\bf Part 6.}
The same as the \eqref{eq:rf-consistency6:result}, we can deduce that
\begin{align*}
\E [{S_6}^2 \given D_1=1]
\lesssim h^{-m/2}\E\Big[\Big(\frac{n_0}{n_1}-\frac{\P(D=0)}{\P(D=1)}\Big)^2\Biggiven D_1=1\Big]
\lesssim \frac{1}{nh^{m/2}}.\yestag\label{eq:rfnew-consistency6:result}
\end{align*}

{\bf Part 7.}
Define
\begin{align*}
\tilde{T}_7 & := h^{-m/2} \ind\Big(X_1 \in B^1(x, h^{1/2}d_K)\Big) B^{-1} \sum_{b=1}^B \sum_{t\ge1} (\lvert \{k \in \cI^1_b:X_k \in L^1_{bt}\} \rvert)^{-1}\ind(1 \in \cI^1_b:X_1 \in L^1_{bt}) \\
& \sum_{D_j=0}  \ind(X_j \in L^1_{bt}),
\end{align*}
then
\begin{align*}
    \E[T_7^2\given \mD,D_1=1,\mX_1] \lesssim \{\E[\tilde{T}_7\given \mD,D_1=1,\mX_1]\}^2+ \Var[\tilde{T}_7\given \mD,D_1=1,\mX_1].
\end{align*}

For the bias term,
\begin{align*}
& \E[\tilde{T}_7\given \mD,D_1=1,\mX_1]\\
= &  \frac{n_0}{n_1} h^{-m/2} \ind\Big(X_1 \in B^1(x, h^{1/2}d_K)\Big)\E\Big[\frac{s}{\lvert L^1(X_1)}\zeta_0(L^1(X_1))\Biggiven \mD,D_1=1,\mX_1,1\in\cI^1\Big].
\end{align*}

For the variance term, 
\begin{align*}
& \Var[\tilde{T}_7\given \mD,D_1=1,\mX_1]\\
= & \frac{1}{Bh^m}\ind\Big(X_1 \in B^1(x, h^{1/2}d_K)\Big) \Var\Big[\sum_{t\ge1} (\lvert \{k \in \cI^1:X_k \in L^1_{t}\} \rvert)^{-1}\ind(1 \in \cI^1:X_1 \in L^1_{t}) \\
& \sum_{D_j=0}  \ind(X_j \in L^1_{bt})\Biggiven \mD,D_1=1,\mX_1\Big]\\
\leq & \frac{1}{Bh^m}\ind\Big(X_1 \in B^1(x, h^{1/2}d_K)\Big) \E\Big[\sum_{t\ge1} (\lvert \{k \in \cI^1:X_k \in L^1_{t}\} \rvert)^{-2}\ind(1 \in \cI^1:X_1 \in L^1_{t}) \\
&n_0 \sum_{D_j=0}  \ind(X_j \in L^1_{t})\Biggiven \mD,D_1=1,\mX_1\Big]\\
\lesssim & h^{-m} \ind\Big(X_1 \in B^1(x, h^{1/2}d_K)\Big) \E\Big[n_0 \sum_{t\ge1} (\lvert \{k \in \cI^1:X_k \in L^1_{t}\} \rvert)^{-1}\ind(1 \in \cI^1:X_1 \in L^1_{t}) \\
& \zeta_0 (L^1_{t}\cap\cM)\Biggiven \mD,D_1=1,\mX_1\Big]\\
= &  \frac{n_0}{n_1} h^{-m} \ind\Big(X_1 \in B^1(x, h^{1/2}d_K)\Big)\E\Big[\frac{s}{\lvert L^1(X_1)}\zeta_0(L^1(X_1))\Biggiven \mD,D_1=1,\mX_1,1\in\cI^1\Big].
\end{align*}

Thus by Lemma~\ref{lemma:rf-clt},
\begin{align*}
& \E[\tilde{T}_7\given \mD,D_1=1]
= \E[\E[\tilde{T}_7\given \mD,D_1=1,\mX_1]\given \mD,D_1=1]\\
\lesssim & \Big(\frac{n_0^2}{n_1^2}+\frac{n_0}{n_1}\Big)\E\Big[h^{-m} \ind\Big(X_1 \in B^1(x, h^{1/2}d_K)\Big)\Biggiven \mD,D_1=1\Big]
\lesssim \Big(\frac{n_0^2}{n_1^2}+\frac{n_0}{n_1}\Big)h^{-m/2+\epsilon}.
\end{align*}
Therefore,
\begin{equation}\label{eq:rfnew-consistency7:result}
    \E[ S^2_7 ] = O(h^{-m/2+\epsilon}).
\end{equation}

Combining \eqref{eq:rfnew-consistency1:result}, \eqref{eq:rfnew-consistency2:result}, \eqref{eq:rfnew-consistency3:result}, \eqref{eq:rfnew-consistency4:result}, \eqref{eq:rfnew-consistency5:result}, \eqref{eq:rfnew-consistency6:result}, and 
\eqref{eq:rfnew-consistency7:result} yields \eqref{asp:dr-propensity-2-single}.
\end{proof}

\subsection{Proof of Lemma~\ref{lemma:rf-mse-discrepancy}}
\begin{proof}[Proof of Lemma~\ref{lemma:rf-mse-discrepancy}]
By \eqref{eq:rf-w} and Assumption~\ref{asp:rf1}\ref{asp:rf1-2}, we have
\begin{align*}
& \E \Big[\Big(\frac{1}{n} \sum_{i=1}^n h^{-m/2} K\left(\lVert h^{-1/2}\left(X_i-x\right)\rVert\right) \sum_{j:D_j=1-D_i} \lvert w_{i\leftarrow j} \rvert \lVert X_i-X_j \rVert\Big)^{2\gamma}\Big] \\
\lesssim & \E \Big[\Big(\frac{1}{n} \sum_{i=1}^n h^{-m/2} K\left(\lVert h^{-1/2}\left(X_i-x\right)\rVert\right) \sum_{j:D_j=1-D_i} w_{i\leftarrow j} \Big\lVert {\rm diam}(L_t \cap \cM) \Big\rVert_\infty\Big)^{2\gamma}\Big] \\
\lesssim & h^{(1/2+\epsilon)2\gamma} \E \Big[\Big(\frac{1}{n} \sum_{i=1}^n h^{-m/2} K\left(\lVert h^{-1/2}\left(X_i-x\right)\rVert\right) \sum_{j:D_j=1-D_i}  w_{i\leftarrow j} \Big)^{2\gamma}\Big]\\
= & h^{(1+2\epsilon)\gamma} \E \Big[\Big(\frac{1}{n} \sum_{i=1}^n h^{-m/2} K\left(\lVert h^{-1/2}\left(X_i-x\right)\rVert\right) \Big)^{2\gamma}\Big]
\lesssim h^{(1+2\epsilon)\gamma}.
\end{align*}
This completes the proof.
\end{proof}

{
\bibliographystyle{apalike}
\bibliography{manifold}
}

\end{document}